\documentclass[12pt,oneside,english]{amsart}
\usepackage[LGR,T1]{fontenc}
\usepackage[latin9]{luainputenc}
\usepackage{geometry}
\usepackage{mathrsfs}
\usepackage{amstext}
\usepackage{amsthm}
\usepackage{amssymb}
\usepackage[all]{xy}

\makeatletter

\DeclareRobustCommand{\greektext}{%
  \fontencoding{LGR}\selectfont\def\encodingdefault{LGR}}
\DeclareRobustCommand{\textgreek}[1]{\leavevmode{\greektext #1}}
\ProvideTextCommand{\~}{LGR}[1]{\char126#1}

\numberwithin{equation}{section}
\numberwithin{figure}{section}
\theoremstyle{plain}
\newtheorem{thm}{\protect\theoremname}[section]
\theoremstyle{remark}
\newtheorem{rem}[thm]{\protect\remarkname}
\theoremstyle{plain}
\newtheorem{prop}[thm]{\protect\propositionname}
\theoremstyle{plain}
\newtheorem{lem}[thm]{\protect\lemmaname}
\theoremstyle{definition}
\newtheorem{defn}[thm]{\protect\definitionname}
\theoremstyle{definition}
\newtheorem{example}[thm]{\protect\examplename}
\theoremstyle{plain}
\newtheorem{question}[thm]{\protect\questionname}
\theoremstyle{plain}
\newtheorem{cor}[thm]{\protect\corollaryname}

\usepackage{amsthm}
\usepackage{amscd}

\makeatother

\usepackage{babel}
\providecommand{\corollaryname}{Corollary}
\providecommand{\definitionname}{Definition}
\providecommand{\examplename}{Example}
\providecommand{\lemmaname}{Lemma}
\providecommand{\propositionname}{Proposition}
\providecommand{\questionname}{Question}
\providecommand{\remarkname}{Remark}
\providecommand{\theoremname}{Theorem}

\begin{document}
\global\long\def\bbR{\mathbb{\mathbf{R}}}%

\global\long\def\R{\mathbb{\mathrm{R}}}%

\global\long\def\C{\mathbb{\mathbf{C}}}%

\global\long\def\Q{\mathbb{\mathbf{Q}}}%

\global\long\def\Z{\mathbf{Z}}%

\global\long\def\P{\mathbb{P}}%

\global\long\def\T{\mathbb{T}}%

\global\long\def\F{\mathbb{\mathbf{F}}}%

\global\long\def\bbN{\mathbb{\mathbf{N}}}%

\global\long\def\M{\mathrm{M}}%

\global\long\def\A{\mathrm{A}}%

\global\long\def\H{\mathrm{H}}%

\global\long\def\D{\mathrm{\mathbf{D}}}%

\global\long\def\B{\mathbf{B}}%

\global\long\def\N{\mathrm{\mathrm{N}}}%

\global\long\def\inf{\mathrm{inf}}%

\global\long\def\sup{\mathrm{sup}}%

\global\long\def\Hom{\mathbb{\mathrm{Hom}}}%

\global\long\def\Ext{\mathbb{\mathbb{\mathrm{Ext}}}}%

\global\long\def\Ker{\mathbb{\mathrm{Ker}}}%

\global\long\def\Gal{\mathrm{Gal}}%

\global\long\def\Aut{\mathrm{Aut}}%

\global\long\def\End{\mathrm{End}}%

\global\long\def\Char{\mathrm{Char}}%

\global\long\def\rank{\mathrm{rank}}%

\global\long\def\deg{\mathrm{deg}}%

\global\long\def\det{\mathrm{det}}%

\global\long\def\Tr{\mathrm{Tr}}%

\global\long\def\Id{\mathrm{Id}}%

\global\long\def\Spec{\mathrm{Spec}}%

\global\long\def\Lie{\mathrm{Lie}}%

\global\long\def\span{\mathrm{span}}%

\global\long\def\sep{\mathrm{sep}}%

\global\long\def\sgn{\mathrm{sgn}}%

\global\long\def\dist{\mathrm{dist}}%

\global\long\def\val{\mathrm{val}}%

\global\long\def\dR{\mathrm{dR}}%

\global\long\def\st{\mathrm{st}}%

\global\long\def\rig{\mathrm{rig}}%

\global\long\def\cris{\mathrm{cris}}%

\global\long\def\Mat{\mathrm{Mat}}%

\global\long\def\cyc{\mathrm{cyc}}%

\global\long\def\et{\mathrm{\acute{e}t}}%

\global\long\def\Frob{\mathrm{Frob}}%

\global\long\def\SL{\mathrm{SL}}%

\global\long\def\GL{\mathrm{GL}}%

\global\long\def\Spa{\mathrm{\mathrm{Spa}}}%

\global\long\def\Br{\mathrm{Br}}%

\global\long\def\Ind{\mathrm{Ind}}%

\global\long\def\LT{\mathrm{LT}}%

\global\long\def\Res{\mathrm{Res}}%

\global\long\def\SL{\mathrm{SL}_{2}}%

\global\long\def\Mod{\mathrm{Mod}}%

\global\long\def\sm{\mathrm{sm}}%

\global\long\def\lsm{\mathrm{lsm}}%

\global\long\def\psm{\mathrm{psm}}%

\global\long\def\plsm{\mathrm{plsm}}%

\global\long\def\Fil{\mathrm{Fil}}%

\global\long\def\la{\mathrm{la}}%

\global\long\def\pa{\mathrm{pa}}%

\global\long\def\Sen{\mathrm{Sen}}%

\global\long\def\dif{\mathrm{dif}}%

\global\long\def\HT{\mathrm{HT}}%

\global\long\def\FT{\mathrm{FT}}%

\global\long\def\FF{\mathrm{FF}}%

\global\long\def\Kla{K-\mathrm{la}}%

\global\long\def\Kpa{K-\mathrm{pa}}%

\global\long\def\an{\mathrm{an}}%

\global\long\def\han{\text{-an}}%

\global\long\def\hla{\text{-la}}%

\global\long\def\hpa{\text{-pa}}%

\global\long\def\Sol{\mathrm{Sol}}%

\global\long\def\grad{\mathrm{\triangledown}}%

\global\long\def\WL{\mathrm{\widetilde{\Lambda}}}%

\global\long\def\weak{\rightharpoonup}%

\global\long\def\weakstar{\overset{*}{\rightharpoonup}}%

\global\long\def\diam{\diamond}%

\title{Locally analytic vector bundles on the Fargues-Fontaine curve}
\author{Gal Porat}
\begin{abstract}
In this article, we develop a version of Sen theory for equivariant
vector bundles on the Fargues-Fontaine curve. We show that every equivariant
vector bundle canonically descends to a locally analytic vector bundle.
A comparison with the theory of $(\varphi,\Gamma)$-modules in the
cyclotomic case then recovers the Cherbonnier-Colmez decompletion
theorem. Next, we focus on the subcategory of de Rham locally analytic
vector bundles. Using the $p$-adic monodromy theorem, we show that
each locally analytic vector bundle $\mathcal{E}$ has a canonical
differential equation for which the space of solutions has full rank.
As a consequence, $\mathcal{E}$ and its sheaf of solutions $\Sol(\mathcal{E})$
are in a natural correspondence, which gives a geometric interpretation
of a result of Berger on $(\varphi,\Gamma)$-modules. In particular,
if $V$ is a de Rham Galois representation, its associated filtered
$(\varphi,N,G_{K})$-module is realized as the space of global solutions
to the differential equation. A key to our approach is a vanishing
result for the higher locally analytic vectors of representations
satisfying the Tate-Sen formalism, which is also of independent interest.
\end{abstract}

\maketitle
\tableofcontents{}

\section{Introduction}

The study of $p$-adic Galois representations has been conditioned
to an extent by two dogmas. One is the \emph{analytic} dogma; its
main idea is to associate to every such representation a $(\varphi,\Gamma)$-module
over the Robba ring and to study these objects using $p$-adic analysis.
The other dogma is \emph{geometric: }to every $p$-adic Galois representation
one associates an equivariant vector bundle over the Fargues-Fontaine
curve. The aim of this article is, roughly speaking, to find a framework
where both analysis and geometry can be carried out. In recent years,
much of the theory of $p$-adic Galois representations has been understood
in terms of the geometry of the Fargues-Fontaine curve. A notable
exception has been the $p$-adic Langlands program, where the analyic
approach plays a cruicial role. Thus we are motivated to reduce this
discrepancy by introducing corresponding objects on the Fargues-Fontaine
curve which are also amenable to analytic methods. These are the locally
analytic vector bundles, the main new objects introduced in this article.

We shall now explain this in more detail. Let $K$ be a finite extension
of $\Q_{p}$ with absolute Galois group $G_{K}$. Let $K_{\cyc}$
be the cyclotomic extension of $K$ and write $\Gamma=\Gal(K_{\cyc}/K)$.
For the sake of simplifying the introduction, we shall focus now on
the cyclotomic setting, though as we shall explain later, the content
of this paper will apply to a wider class of Galois extensions $K_{\infty}/K$.
We have the category $\mathrm{Rep}_{\Q_{p}}(G_{K})$ of finite dimensional
$\Q_{p}$-representations of $G_{K}$. 

On the one hand, $\mathrm{Rep}_{\Q_{p}}(G_{K})$ can be studied via
$p$-adic analysis. To do this, one introduces the Robba ring $\mathcal{R}$,
which is the ring of power series over a certain finite extension
of $\Q_{p}$ in a variable $T$ which converge in some annuli $r\leq|T|<1$.
It has an action of a Frobenius operator $\varphi$ as well as an
action of $\Gamma$. By work of Cherbonnier-Colmez, Fontaine and Kedlaya,
it is known that there is a fully faithful embedding
\[
\mathrm{Rep}_{\Q_{p}}(G_{K})\hookrightarrow\text{\{(\ensuremath{\varphi,\Gamma)}-modules over \ensuremath{\mathcal{R}}\}},
\]
with the essential image consisting of the semistable slope 0 objects.
If $D$ is a $(\varphi,\Gamma)$-module over $\mathcal{R}$, a fundamental
fact is that the $\Gamma$-action on $D$ can be differentiated, namely,
there is a well defined action of $\Lie(\Gamma)$ on $D$. Since $\Lie(\Gamma)$
is 1-dimensional, this data is the same as that of a connection $\nabla$
which acts on functions of $T$ by a multiple of $d/dT$ . It is precisely
this structure which allows the introduction of $p$-adic analysis
into the picture. For example, in the construction of the $p$-adic
Langlands correspondence for $\GL_{2}(\Q_{p})$ given in \cite{Co10}
the use of this analytic structure is ubiquitous.

On the other hand, $\mathrm{Rep}_{\Q_{p}}(G_{K})$ can be studied
via geometry. The Fargues-Fontaine curve, studied extensively in \cite{FF18},
is defined as the analytic adic space
\[
\mathcal{X}=\mathcal{X}(\widehat{K}_{\cyc}):=(\Spa\A_{\inf}-\{p[p^{\flat}]=0\})/(\varphi^{\Z},\Gal(\overline{K}/K_{\cyc}))
\]
 (see $\mathsection3$) and has a natural action of $\Gamma$. By
the work of Fargues and Fontaine, there is a fully faithful embedding
\[
\mathrm{Rep}_{\Q_{p}}(G_{K})\hookrightarrow\{\text{\ensuremath{\Gamma}-equivariant vector bundles on \ensuremath{\mathcal{X}}}\},
\]
again with the essential image consisting of the semistable slope
0 objects. In fact, Fargues and Fontaine show there is an equivalence
\[
\text{\{(\ensuremath{\varphi,\Gamma)}-modules over \ensuremath{\mathcal{R}}\}}\cong\{\text{\ensuremath{\Gamma}-equivariant vector bundles on \ensuremath{\mathcal{X}}}\},
\]
compatible with each of the aforementioned embeddings of $\mathrm{Rep}_{\Q_{p}}(G_{K})$.

Unfortunately, the action of $\Gamma$ on an equivariant vector bundle
on $\mathcal{X}$ \emph{cannot }be differentiated. This is already
true for the structure sheaf $\mathcal{O}_{\mathcal{X}}$. Here is
a simplified model of the situation which illustrates why there is
no action of $\Lie(\Gamma)$ on $\mathcal{O}_{\mathcal{X}}$. The
functions on an open subset of $\mathcal{X}$ can roughly be thought
of as power series in $T^{1/p^{\infty}}$ satisfying certain convergence
conditions. When we try to apply the operator $d/dT$ to such a power
series, the result will often not converge since the derivative 
\[
d(T^{1/p^{n}})/dT=(1/p^{n})T^{1/p^{n}-1}
\]
grows exponentially larger $p$-adically as $n$ goes to infinity.
Nevertheless, there is a way to single out these sections for which
the action of $\Lie(\Gamma)$ does not explode. This is achieved by
considering only these sections on which the action of $\Gamma$ is
regular enough. In this toy model picture, this will amount to considering
only these power series where the coefficient of the exponent of $T^{k/p^{n}}$
will decay proportionally to $p^{n}$. 

More canonically and more generally, these elements for which differentation
is possible are precisely the locally analytic elements. Given an
equivariant vector bundle $\widetilde{\mathcal{E}}$ on $\mathcal{X}$,
there is a subsheaf of locally analytic sections $\widetilde{\mathcal{E}}^{\la}\subset\widetilde{\mathcal{E}}$.
This sheaf is a module over $\mathcal{O}_{\mathcal{X}}^{\la}$ which
is preserved under the $\Gamma$-action, and, crucially, $\Lie(\Gamma)$
acts on $\widetilde{\mathcal{E}}^{\la}$. We are thus naturally led
to the definition of a \emph{locally analytic vector bundle on $\mathcal{X}$}:
by this we shall mean a locally free $\mathcal{O}_{\mathcal{X}}^{\la}$-module
together with a $\Gamma$-action. The point is that locally analytic
vector bundles capture both analytic and geometric information, both
of which has proven important for the study of $\mathrm{Rep}_{\Q_{p}}(G_{K})$.

Our first main result is saying that there is no loss of information
in this process: each equivariant vector bundle canonically descends
to a locally analytic vector bundle.

\textbf{Theorem A. }\emph{The functor $\widetilde{\mathcal{E}}\mapsto\widetilde{\mathcal{E}}^{\la}$
gives rise to an equivalence of categories from the category of $\Gamma$-equivariant
vector bundles on $\mathcal{X}$ to the category of locally analytic
vector bundles on $\mathcal{X}$}. \emph{Its inverse is given by the
functor $\mathcal{E}\mapsto\mathcal{O}_{\mathcal{X}}\otimes_{\mathcal{O}_{\mathcal{X}}^{\la}}\mathcal{E}$. }

This theorem fits naturally into the framework of Sen theory, as we
shall now explain. Let $V\in\mathrm{Rep}_{\Q_{p}}(G_{K})$. Then according
to Sen's theory, proven in \cite{Se81}, there is a canonical isomorphism
\[
(V\otimes_{\Q_{p}}\C_{p})^{\Gal(\overline{K}/K_{\cyc})}\cong\widehat{K}_{\cyc}\otimes_{K_{\cyc}}\D_{\Sen}(V)
\]
where $\D_{\Sen}(V)$ is the $K_{\cyc}$-subspace of elements with
finite $\Gamma$-orbit in $V\otimes_{\Q_{p}}\C_{p}$. Later, Fontaine
(see $\mathsection3\text{-}4$ of \cite{Fo04}) proved an analogue
of this theorem for $\B_{\dR}^{+}$: he showed there is an isomorphism
\[
(V\otimes_{\Q_{p}}\B_{\dR}^{+})^{\Gal(\overline{K}/K_{\cyc})}\cong(\B_{\dR}^{+})^{\Gal(\overline{K}/K_{\cyc})}\otimes_{K_{\cyc}[[t]]}\D_{\dif}^{+}(V)
\]
where $\D_{\dif}^{+}(V)$ is a canonical $K_{\cyc}[[t]]$-submodule
of $V\otimes_{\Q_{p}}\B_{\dR}^{+}$.

In fact, both of these results are implied by Theorem A by specializing
at the ``point at infinity'' $x_{\infty}\in\mathcal{X}$. Indeed,
when $\mathcal{\widetilde{\mathcal{E}}}$ is the equivariant vector
bundle associated to $V\in\mathrm{Rep}_{\Q_{p}}(G_{K})$ and $\mathcal{E}=\widetilde{\mathcal{E}}^{\la}$,
specializing the isomorphism $\mathcal{\widetilde{\mathcal{E}}}\cong\mathcal{O}_{\mathcal{X}}\otimes_{\mathcal{O}_{\mathcal{X}}^{\la}}\mathcal{E}$
at the fiber of $x_{\infty}$ gives rise to an isomorphism
\[
\mathcal{\widetilde{\mathcal{E}}}_{k(x_{\infty})}\cong\mathcal{O}_{\mathcal{X},k(x_{\infty})}\otimes_{\mathcal{O}_{\mathcal{X},k(x_{\infty})}^{\la}}\mathcal{E}_{k(x_{\infty})}
\]
which is none other than Sen's theorem. Similarly, there is an isomorphism
of the completed stalks at $x_{\infty}$
\[
\mathcal{\widetilde{\mathcal{E}}}_{x_{\infty}}^{\wedge,+}\cong\mathcal{O}_{\mathcal{X},x_{\infty}}^{\wedge,+}\otimes_{\mathcal{O}_{\mathcal{X},x_{\infty}}^{\la,\wedge,+}}\mathcal{E}_{x_{\infty}}^{\wedge,+}
\]
which recovers Fontaine's theorem. In this way, Theorem A is a sheaf
theoretic version of Sen theory on $\mathcal{X}$ which specializes
at $x_{\infty}$ to classical Sen theory.

In the interest of applications, we give a proof of this equivalence
not just for the cyclotomic extension, but more generally for any
$p$-adic Lie group $\Gamma=\Gal(K_{\infty}/K)$ where $K_{\infty}$
is an infinitely ramified Galois extension of $K$ which contains
an unramified twist of the cyclotomic extension. Notably, this condition
holds when $K_{\infty}$ is the extension generated by the torsion
points of a formal group.

As we shall explain in the article, these ideas are closely related
to the decompletion of $(\varphi,\Gamma)$-modules, especially in
the case $K_{\infty}=K_{\cyc}$. This is not too surprising, because
such $(\varphi,\Gamma)$-modules are also obtained by a Sen theory
type of idea through the theorem of Cherbonnier and Colmez in \cite{CC98},
and further, these objects relate to $\D_{\Sen}$ and $\D_{\dif}^{+}$
in a similar way. In fact, Theorem A is equivalent to the Cherbonnier-Colmez
theorem on decompletion of $(\varphi,\Gamma)$-modules (after inverting
$p$). Our proof is not independent from the ideas of Cherbonnier-Colmez,
since we still use their trace maps in our arguments. However, it
is logically different - more on this below. 

First, let us discuss an application of Theorem A, which was a major
source of motivation for this work. We give a geometric reinterpertation
of Berger's work on $p$-adic differential equations and filtered
$(\varphi,N)$-modules (\cite{Be08B}). In that article, Berger establishes
several results regarding de Rham $(\varphi,\Gamma)$-modules (for
example, these $(\varphi,\Gamma)$-modules arising from de Rham $p$-adic
Galois representations). To such a $(\varphi,\Gamma)$-module $D$,
Berger associates another $(\varphi,\Gamma)$-module $\N_{\dR}(D)$
(a so called $p$-adic differential equation), and a $\overline{K}$-vector
space of solutions
\[
\Sol(D):=\varinjlim_{[L:K]<\infty}(\mathcal{R}_{L}[\log T]\otimes_{\mathcal{R}}\N_{\dR}(V))^{G_{L}},
\]
where $\mathcal{R}_{L}$ is the Robba ring with respect to $L$. The
following results can be derived from the main results of \cite{Be08B},
for $D$ a de Rham $(\varphi,\Gamma)$-module:

(i) $\Sol(D)$ is a $\overline{K}$-vector space of rank equal to
the rank of $D$.

(ii) There is a canonical isomorphism
\[
\mathcal{R}_{\overline{K}}[\log T]\otimes_{K^{\mathrm{un}}}\Sol(D)\cong\mathcal{R}_{\overline{K}}[\log T]\otimes_{\mathcal{R}}\N_{\dR}(V).
\]

(iii) $\overline{K}\otimes_{K^{\mathrm{un}}}\Sol(D)$ is canonically
isomorphic to $\overline{K}\otimes_{K}\D_{\dR}(D)$.

(iv) $\Sol(D)$ is naturally a filtered $(\varphi,N,G_{K})$-module.

Furthermore, the functor $D\mapsto\Sol(D)$ gives rise to an equivalence
of categories from the category of de Rham $(\varphi,\Gamma)$-modules
over $\mathcal{R}$ to the category of filtered $(\varphi,N,G_{K})$-module. 

The functor of solutions is ultimately understood in \cite{Be08B}
by solving the differential equation $\Lie(\Gamma)=0$, and as such,
uses $p$-adic analysis in a crucial way. It is therefore natural
to apply Theorem A to give a geometric interpertation of these results,
something previously inaccessible in the framework of vector bundles
on the Fargues-Fontaine curve. In fact, when interperted in a geometric
way, \cite[Théorème A]{Be08B} turns out to to be reminiscient of
the Riemann-Hilbert correspondence.

Our second main result is the desired geometric interpertation of
Berger's results. To describe it, we need to introduce some notation.
We have
\[
\mathcal{X}_{\log,\overline{K}}:=\varprojlim_{[L:K]<\infty}\mathcal{X}_{\log,L},
\]
where each $\mathcal{X}_{\log,L}$ is a the analytic line bundle over
$\mathcal{X}_{L}:=\mathcal{X}(\widehat{L}_{\cyc})$ corresponding
to $\mathcal{O}_{\mathcal{X}_{L}}(1)$, endowed with the projection
$p_{\log,L}:\mathcal{X}_{\log,L}\rightarrow\mathcal{X}_{L}$ (see
$\mathsection8.3$). Essentially, $\mathcal{X}_{\log,\overline{K}}$
is obtained by adjoining all $\overline{K}$-scalars and a logarithm
to the functions on $\mathcal{X}$. Now let $\mathcal{E}$ be a de
Rham locally analytic vector bundle, i.e., suppose that $\dim_{K}\widehat{\mathscr{\mathcal{E}}}_{x_{\infty}}^{\Gamma=1}=\rank(\mathcal{E})$
(see $\mathsection8.2$). For example, if $V$ is a de Rham $p$-adic
Galois representation, then its associated locally analytic vector
bundle is de Rham. To such $\mathcal{E}$, we associate a sheaf $\Sol(\mathcal{E})$
on $\mathcal{X}$, given by
\[
\mathrm{Sol}(\mathcal{E}):=\varinjlim_{[L:K]<\infty}p_{\log,L,*}(p_{\log,L}^{*}\mathcal{N}_{\dR}(\mathcal{E}))^{\Lie(\Gamma)=0},
\]
where $\mathcal{N}_{\dR}(\mathcal{E})$ is a modification of $\mathcal{E}$
corresponding to the de Rham lattice of $\mathcal{E}$ at $x_{\infty}$.
Roughly speaking, $\mathrm{Sol}(\mathcal{E})$ is the sheaf of solutions
to the differential equation $\nabla=0$ on the modification $\mathcal{N}_{\dR}(\mathcal{E})$.
We shall also consider a variant $\Sol^{\varphi}(\mathcal{E})$, which
are the solutions on the pullback of $\mathcal{E}$ along the usual
covering $\mathcal{Y}_{(0,\infty)}\rightarrow\mathcal{X}$ for $\mathcal{Y}_{(0,\infty)}=\Spa\A_{\inf}-\{p[p^{\flat}]=0\}/\Gal(\overline{K}/K_{\cyc})$.
We then have the following result, by analogy with the results of
\cite{Be08B} (see $\mathsection8$ for yet more precise statements).

\textbf{Theorem B.} \emph{Let $\mathcal{E}$ be a de Rham locally
analytic vector bundle. }

\emph{$(i)$ The sheaf of solutions $\Sol(\mathcal{E})$ is locally
free over the subsheaf of potentially log smooth sections $\mathcal{O}_{\mathcal{X}}^{\plsm}\subset\mathcal{O}_{\mathcal{X}}^{\la}$
and its rank is equal to the rank of $\mathcal{E}$.}

\emph{$(ii)$ There is a canonical isomorphism
\[
\mathcal{O}_{\mathcal{X}_{\log,\overline{K}}}^{\la}\otimes_{\mathcal{O}_{\mathcal{X}}^{\plsm}}\mathrm{Sol}(\mathcal{E})\xrightarrow{\sim}\mathcal{O}_{\mathcal{X}_{\log,\overline{K}}}^{\la}\otimes_{\mathcal{O}_{\mathcal{X}}^{\la}}\mathcal{N}_{\dR}(\mathcal{E}).
\]
}

\emph{$(iii)$ The stalk of $\Sol(\mathcal{E})$ at $x_{\infty}$
is canonically isomorphic to $\overline{K}\otimes_{K}\D_{\dR}(\mathcal{E})$.}

\emph{$(iv)$ The space of global solutions $\H^{0}(\mathcal{Y}_{(0,\infty)},\Sol^{\varphi}(\mathcal{E}))$
is naturally a filtered $(\varphi,N,G_{K})$-module.}

\emph{Furthermore, the functor $\mathcal{E}\mapsto\H^{0}(\mathcal{Y}_{(0,\infty)},\Sol^{\varphi}(\mathcal{E}))$
gives rise to an equivalence of categories from the category of de
Rham locally analytic vector bundles to the category of filtered $(\varphi,N,G_{K})$-modules.}
\begin{rem}
1. In particular, if $V$ is a de Rham representation of $G_{K}$
with associated locally analytic vector bundle $\mathcal{E}$, then
$\H^{0}(\mathcal{Y}_{(0,\infty)},\Sol^{\varphi}(\mathcal{E}))=\D_{\mathrm{pst}}(V)$
and the stalk $\Sol(\mathcal{E})_{x_{\infty}}$ is identified with
$\overline{K}\otimes_{K}\D_{\dR}(V)$. The localization map corresponds
to the natural map $\D_{\mathrm{pst}}(V)\rightarrow\overline{K}\otimes_{K}\D_{\dR}(V)$.

2. If $\mathcal{E}$ becomes crystalline after extending $K$ to a
finite extension $L\subset K_{\infty}$, the sheaf $\mathcal{N}_{\dR}(\mathcal{E})^{\nabla=0}\subset\Sol(\mathcal{E})$
is locally free over the subsheaf of smooth sections $\mathcal{O}_{\mathcal{X}}^{\sm}\subset\mathcal{O}_{\mathcal{X}}^{\la}$
of rank equal to the rank of $\mathcal{E}$, and there is a simpler
canonical isomorphism
\[
\mathcal{O}_{\mathcal{X}}^{\la}\otimes_{\mathcal{O}_{\mathcal{X}}^{\sm}}\mathcal{N}_{\dR}(\mathcal{E})^{\nabla=0}\xrightarrow{\sim}\mathcal{N}_{\dR}(\mathcal{E}).
\]

3. The sheaf $\mathcal{O}_{\mathcal{X}}^{\plsm}$ is much smaller
than $\mathcal{O}_{\mathcal{X}}^{\la}$. Though we have not been quite
able to show this, $\mathcal{O}_{\mathcal{X}}^{\plsm}$ seems to be
``almost'' a locally constant sheaf except that the base field becomes
slightly larger when localizing; for that reason, we think of $\Sol(\mathcal{E})$
as morally being close to a local system on $\mathcal{X}$. In this
sense the $(\varphi,N,G_{K})$-structure is related to the monodromy
of the $p$-adic differential equation $\nabla=0$.
\end{rem}

Finally, let us discuss the proof of Theorem A. The essential point
is to show that if $\widetilde{\mathcal{E}}$ is an equivariant vector
bundle on $\mathcal{X}$, the natural map $\mathcal{O}_{\mathcal{X}}\otimes_{\mathcal{O}_{\mathcal{X}}^{\la}}\widetilde{\mathcal{E}}^{\la}\rightarrow\widetilde{\mathcal{E}}$
is an isomorphism. Fargues and Fontaine observe that the only point
of $\mathcal{X}$ with finite $\Gamma$-orbit is $x_{\infty}$. The
idea is then to use a very simple geometric argument: once one knows
that $\mathcal{O}_{\mathcal{X}}\otimes_{\mathcal{O}_{\mathcal{X}}^{\la}}\widetilde{\mathcal{E}}^{\la}\rightarrow\widetilde{\mathcal{E}}$
is injective, everything can be understood by arguing locally at $x_{\infty}$.
Indeed, if this map is an isomorphism after localizing and completing
along $\mathcal{O}_{\mathcal{X}}\rightarrow\widehat{\mathcal{O}}_{\mathcal{X},x_{\infty}}^{+}$,
then the cokernel has to be supported at finitely many points outside
$x_{\infty}$. But these points also form a finite $\Gamma$-orbit,
so the cokernel cannot be supported anywhere.

It therefore remains to understand the properties of our spaces of
locally analytic vectors under certain localizations and completions.
To do this, we are naturally led to consider higher locally analytic
vectors and their vanishing, and we prove a representation-theoretic
result which is of independent interest. To state the result, let
$G$ be a $p$-adic Lie group and let $\WL$ be a Banach ring with
a continuous action of $G$. Assume the topology on $\WL$ is $p$-adic. 

\textbf{Theorem C. }\emph{Suppose $G$ and $\WL$ satisfy the Tate-Sen
axioms (TS1)-(TS3) of \cite{BC08} as well as an additional axiom
(TS4). Then for any finite free $\widetilde{\Lambda}$-semilinear
representation $M$ of $G$, the higher locally analytic vectors $\R_{G\hla}^{i}(M)$
are zero for $i\geq1$.}

Here are two special cases of the theorem where we conclude that $\R_{G\hla}^{i}(M)=0$
for $i\geq1$.

1. If $M$ is a finite dimensional $\widehat{K}_{\infty}$-module
with a semilinear action of $\Gamma$, for $K_{\infty}$ containing
an unramified twist of $K_{\cyc}$. In fact, the vanishing of $\R_{G\hla}^{i}(M)$
can be established for arbitrary $K_{\infty}$, see $\mathsection5$.

2. If $M$ a finite free $\widetilde{\B}_{I}(\widehat{K}_{\infty})$-module
with a semilinear action of $\Gamma$, under the same assumptions
on $K_{\infty}$.

Note that the vanishing of higher locally analytic vectors is automatic
for admissible representations, but the examples above are not admissible.
Theorem C illustrates how the Tate-Sen axioms can serve as a substitute
for admissibility.

Theorem C is especially useful for making cohomological computations.
Here is an example application, which follows directly from the main
results of \cite{RJRC21} (see $\mathsection5$): if $M$ satisfies
assumptions of the theorem, then for $i\geq0$ we have natural isomorphisms
\[
\H^{i}(G,M)\cong\H^{i}(G,M^{\la})\cong\H^{i}(\Lie G,M^{\la})^{G}.
\]

Finally, let us mention that in the recent work \cite{RC22}, Juan
Esteban Rodríguez Camargo proves similar results to our Theorem C.
He then applies them in the setting of rigid adic spaces with fantastic
applications to the Calegari-Emerton conjecture, among others.

\subsection{Structure of the article}

$\mathsection2$ contains reminders on locally analytic vectors and
their derived functors. In $\mathsection3$ we give reminders on the
Fargues-Fontaine curve and equivariant vector bundles. In $\mathsection4$
we introduce locally analytic bundles and we discuss their basic properties.
$\mathsection5$ is the longest and most technical section of the
paper, in which we prove Theorem C. Theorem A is proved in $\mathsection6$.
In $\mathsection7$ we compare our results to the theory of $(\varphi,\Gamma)$-modules.
Finally, in $\mathsection8$ we discuss $p$-adic differential equations
on the Fargues-Fontaine curve and explain Theorem B.

At several points in the article we have taken the liberty to raise
speculations and ask questions to which we do not yet know the answer.

\subsection{Notations and conventions}

The field $K$ denotes a finite extension of $\Q_{p}$. We write $K_{\cyc}=K(\mu_{p^{\infty}})$
for the cyclotomic extension. Its Galois group $\Gamma_{\cyc}=\Gal(K_{\cyc}/K)$
is an open subgroup of $\Z_{p}^{\times}$. We denote by $K_{\infty}$
an infinitely ramified Galois extension of $K$ with $\Gamma=\Gal(K_{\infty}/K)$
a $p$-adic Lie group. If $\overline{K}$ denotes the algebraic closure
of $K$, we let $G_{K}=\Gal(\overline{K}/K)$ and $H=\Gal(\overline{K}/K_{\infty})$
so that $G_{K}/H=\Gamma$.

The $p$-adic completion $\widehat{K}_{\infty}$ of $K_{\infty}$
is a perfectoid field. Write $\varpi$ for a pseudouniformizer of
$\widehat{K}_{\infty}$ with valuation $\val\left(\varpi\right)=p$
that admits a sequence of $p$'th power roots $\varpi^{1/p^{n}}$
(such a choice is always possible, and the constructions in this paper
never depend on this choice). Let $\varpi^{\flat}=(\varpi,\varpi^{1/p},...)$
be the corresponding pseudouniformizer of the tilt $\widehat{K}_{\infty}^{\flat}$.

Denote by $\Lie(\Gamma)$ the Lie algebra of $\Gamma$. It is a finite
dimensional $\Q_{p}$-vector space, and if $v\in\Lie(\Gamma)$ is
sufficiently small, we have a corresponding element $\exp(v)\in\Gamma$.

All representations and group actions appearing in this article are
assumed to be continuous. Galois cohomology groups are always taken
in the continuous sense.

If $W$ is a Banach space over $\Q_{p}$ we write $W^{+}$ for its
unit ball.

All completed tensor products appearing in this article are projective.
In other words, if $V^{+}$ and $W^{+}$ are unit balls of two Banach
spaces $V$ and $W$ over $\Q_{p}$, then
\[
V^{+}\widehat{\otimes}_{\Z_{p}}W^{+}=\varprojlim_{n}(V^{+}\otimes_{\Z_{p}}W^{+})/p^{n}
\]
and $V\widehat{\otimes}_{\Q_{p}}W=(V^{+}\widehat{\otimes}_{\Z_{p}}W^{+})[1/p]$. 

\subsection{Acknowledgments}

I would like to thank my PhD advisor Matthew Emerton for his constant
support, advice and supply of ideas. I would also like to thank Laurent
Berger, Pierre Colmez, Ian Gleason, Lue Pan, Léo Poyeton and Joaquín
Rodrigues Jacinto for answering my questions. Special thanks to Kiran
Kedlaya for his explanations regarding smooth functions appearing
in $\mathsection8.4$. Finally, I would like to thank Laurent Berger,
Ehud de-Shalit, Hui Gao, Louis Jabouri, Hao Lee, Léo Poyeton, Stephan
Snegirov and the anonymous referee for their interest and comments.
We hope the contents of this article will make it clear we were inspired
by \cite{BC16} and \cite{Pa21}.

\section{Locally analytic and pro analytic vectors}

In this section we give reminders on locally analytic and pro analytic
vectors and quote results that will be used in $\mathsection4,\mathsection5$
and $\mathsection6$. We shall freely use our conventions in $\mathsection1.2$
regarding Banach spaces.

\subsection{Locally analytic and pro analytic vectors}

We shall say a compact $p$-adic Lie group $G$ is \emph{small} if
there exists a saturated integral valued $p$-valuation on $G$ which
defines its topology and if for some $N\in\Z_{\geq1}$ there exists
an embedding of $G$ into $1+p^{2}M_{N}(\Z_{p})$, the group of $N$
by $N$ matrices congruent to $1$ mod $p^{2}$. See $\mathsection23$
and $\mathsection26$ of \cite{Sch11} for the first condition. If
$G$ is small, there exists an ordered basis $g_{1},...,g_{d}$ such
that $(x_{1},...,x_{d})\mapsto g_{1}^{x_{1}}\cdot...\cdot g_{1}^{x_{d}}$
gives a homeomorphism of $\Z_{p}^{d}$ with $G$. We then have coordinates
on $G$ 
\[
c=(c_{1},...,c_{d}):G\xrightarrow{\sim}\Z_{p}^{d}
\]
defined by the inverse map where $c_{i}(g_{1}^{x_{1}}\cdot...\cdot g_{1}^{x_{d}})=x_{i}$. 

Now let $G$ is an be any compact $p$-adic Lie group. By \cite[Theorem 27.1]{Sch11}
and Ado's theorem (see \cite[Proposition 2.1.3]{Pa21}), the collection
of small open subgroups of $G$ forms a fundamental system of open
neighborhoods of the identity element. Let $W$ be a Banach $\Q_{p}$-linear
representation of $G$ (or $G$-Banach space for short). If $H$ is
a small open subgroup of $G$, choose coordinates $c$ on $H$ and
write $c(h)^{\mathbf{k}}=\prod_{i=1}^{d}c_{i}(h)^{k_{i}}$ if $\mathbf{k}=(k_{1},...,k_{d})$
for $h\in H$. We have the subspace $W^{H\han}$ of $H$-analytic
vectors in $W$; it is the subspace of elements $w\in W$ for which
there exists a sequence of vectors $\left\{ w_{\mathbf{k}}\right\} _{\mathbf{k}\in\bbN^{d}}$
with $w_{\mathbf{k}}\rightarrow0$ and 
\[
h(w)=\sum_{\mathbf{k}\in\bbN^{d}}c(h)^{\mathbf{k}}w_{\mathbf{k}}
\]
for all $h\in H$. The norm $\left|\left|w\right|\right|_{H\han}=\sup_{\mathbf{k}}\left|\left|w_{\mathbf{k}}\right|\right|$
makes $W^{H\han}$ into a Banach space. Note that $W^{H\han}$ does
not depend on the choice of coordinates. We write $W^{\la}=\bigcup_{H}W^{H\han}$
for the subspace of locally analytic vectors of $W$, and endow it
with the inductive limit topology, which makes it into an LB space.
If $W$ is a Fréchet space whose topology is defined by a countable
sequence of seminorms, let $W_{i}$ be the Hausdorff completion of
$W$ for the $i$'th seminorm, so that $W=\varprojlim W_{i}$ is a
projective limit of Banach spaces. We write $W^{\pa}=\varprojlim W_{i}^{\la}$
for the subspace of pro analytic vectors. Finally, we extend the definitions
of locally analytic vectors and pro analytic vectors to LB and LF
spaces in the obvious way. 

The Lie algebra $\Lie(G)$ acts on each $W^{H\han}$ (and hence also
on $W^{\la}$ and $W^{\pa}$) through derivations. This action is
given as follows. If $v\in\Lie(G)$ then $\exp(p^{k}v)\in H$ for
$k\gg0$, and we define
\[
\nabla_{v}(w)=\lim_{k\rightarrow\infty}\frac{\exp(p^{k}v)(w)-w}{p^{k}}.
\]

The operator $\nabla_{v}:W^{H\han}\rightarrow W^{H\han}$ is bounded,
see \cite[Lemma 2.6]{BC16}. 

Locally analytic and pro analytic vectors behave well when we have
a basis of such vectors (\cite[Proposition 2.3]{BC16} and \cite[Proposition 2.4]{Be16}): 
\begin{prop}
\label{2.1}Let $B$ be a Banach or Fréchet $G$-ring and let $W$
be a free $B$-module of finite rank, equipped with a $B$-semilinear
action of $G$. If the $B$-module $W$ has a basis $w_{1},...,w_{d}$
in which the function $G\rightarrow\GL_{d}(B)\subset\mathrm{M}_{d}(B),$
$g\mapsto\Mat(g)$ is $H$-analytic (resp. locally analytic, resp.
pro analytic), then $W^{H\han}=\oplus_{j=1}^{d}B^{H\han}\cdot w_{i}$
(resp. $W^{\la}=\oplus_{j=1}^{d}B^{\la}\cdot w_{i}$, resp. $W^{\pa}=\oplus_{j=1}^{d}B^{\pa}\cdot w_{i}$).
\end{prop}

It will often be useful for us to choose a specific fundamental system
of open neighborhoods of $G$ as follows. Fix a small compact open
$G_{0}\subset G$ which with coordinates $c$. For $n\geq0$ we set
\[
G_{n}=G^{p^{n}}=\left\{ g^{p^{n}}:g\in G_{0}\right\} .
\]
These are subgroups (\cite[Remark 26.9]{Sch11}) which have induced
coordinates $c|_{G_{n}}:G_{n}\xrightarrow{\sim}(p^{n}\Z_{p})^{d}$.
The normalization is such that for $w\in W^{G_{n}\han}$ we can write
\[
g(w)=\sum_{\mathbf{k\in}\mathbb{N}^{d}}c(g)^{\mathbf{k}}w_{\mathbf{k}}
\]
for $g\in G_{n}$ and $\left\{ w_{\mathbf{k}}\right\} _{\mathbf{k}\in\bbN^{d}}$
with $p^{n\left|\mathbf{k}\right|}w_{\mathbf{k}}\rightarrow0$, and
the Banach norm is given by
\[
\left|\left|w\right|\right|_{G_{n}\han}=\sup_{\mathbf{k}}\left|\left|p^{n\left|\mathbf{k}\right|}w_{\mathbf{k}}\right|\right|.
\]
It is easy to check if $w\in W^{G_{n}\han}$ then $\left|\left|w\right|\right|_{G_{m}\han}\leq\left|\left|w\right|\right|_{G_{m+1}\han}$
for $m\geq n$ and $\left|\left|w\right|\right|_{G_{m}\han}=\left|\left|w\right|\right|$
for $m\gg n$ (see \cite[Lemme 2.4]{BC16}).

\subsection{Rings of analytic functions}

Suppose first that $G$ is small. Let $\mathcal{C}^{\an}(G,\Q_{p})$
be the space of analytic functions on $G$. These are those functions
that after pullback by the coordinates $c:G\xrightarrow{\sim}\Z_{p}^{d}$
are of the form 
\[
\mathbf{x}=(x_{1},\cdots,x_{d})\mapsto\sum_{\mathbf{k}=(k_{1},\cdots,k_{d})\in\mathbb{N}^{d}}b_{\mathbf{k}}\mathbf{x}^{\mathbf{k}}.
\]

where $b_{\mathbf{k}}\rightarrow0$ as $|\mathbf{k}|\rightarrow\infty$.
The norm $\left|\left|f\right|\right|_{G}=\sup_{\mathbf{k}\in\mathbb{N}^{d}}\left|\left|b_{\mathbf{k}}\right|\right|$
makes $\mathcal{C}^{\an}(G,\Q_{p})$ into a Banach space. We shall
regard $\mathcal{C}^{\an}(G,\Q_{p})$ as a $G$-representation through
the left regular action of $G$. 

If now $G$ is any compact $p$-adic Lie group with a system of small
neighborhoods $\left\{ G_{n}\right\} _{n\geq0}$ as in $\mathsection2.1$,
we have for each $n\ge0$ the space of analytic functions $\mathcal{C}^{\an}(G_{n},\Q_{p})$
on $G_{n}$. Using the coordinates $c:G_{n}\xrightarrow{\sim}(p^{n}\Z_{p})^{d}$
as in $\mathsection2.1$, we shall regard $\mathcal{C}^{\an}(G_{n},\Q_{p})$
as the ring of functions that under the bijection are identified with
functions of the form 
\[
\mathbf{x}=(x_{1},\cdots,x_{d})\mapsto\sum_{\mathbf{k}=\left(k_{1},\cdots,k_{d}\right)\in\mathbb{N}^{d}}b_{\mathbf{k}}\mathbf{x}^{\mathbf{k}}.
\]
where $p^{n|\mathbf{k}|}b_{\mathbf{k}}\rightarrow0$ as $|\mathbf{k}|\rightarrow\infty$.
Under this normalization 
\[
\left|\left|f\right|\right|_{G_{n}}=\sup_{\mathbf{k}\in\mathbb{N}^{d}}\left|\left|p^{n|\mathbf{k}|}b_{\mathbf{k}}\right|\right|
\]
for $f\in\mathcal{C}^{\an}(G_{n},\Q_{p})$. 

The following lemma will be used in $\mathsection5$.
\begin{lem}
\label{2.2}For $k\geq1$ the subgroup $G_{n+k}$ acts trivially on
$\mathcal{C}^{\an}(G_{n},\Q_{p})^{+}/p^{k}$.
\end{lem}

\begin{proof}
This is an easy exercise using the coordinates. See \cite[Lemma 2.1.2]{Pa21}
for the case $k=1$.
\end{proof}
The following is shown in \cite[Proposition 2.1.3]{Pa21} and in
its proof (originally in the proof of \cite[Théoréme 6.1]{BC16}). 
\begin{prop}
\label{2.3}Suppose that $G$ is small. There is a dense subspace
$\varinjlim_{\ell\in\mathbb{N}}V_{\ell}\subset\mathcal{C}^{\an}(G,\Q_{p})$,
where each $V_{l}$ is a finite-dimensional $G$-subrepresentation
of $\mathcal{C}^{\an}(G,\Q_{p})$ with coefficients in $\Q_{p}$ such
that for any $k,\ell\in\mathbb{N}$, we have $V_{k}\cdot V_{\ell}\subset V_{k+\ell}$. 

Furthermore, if we fix $G$ and consider small open subgroups $G'\subset G$,
we may choose $V_{\ell}(G')\subset\text{\ensuremath{\mathcal{C}^{\an}}}(G',\Q_{p})$
at once for all $G'$ in such a way that the natural map $\mathcal{C}^{\an}(G,\Q_{p})\rightarrow\text{\ensuremath{\mathcal{C}^{\an}}}(G',\Q_{p})$
restricts to $V_{\ell}(G)\rightarrow V_{\ell}(G')$.
\end{prop}

\subsection{Higher locally analytic vectors}

Suppose first that $G$ is small and let $W$ be a $G$-Banach space.
There is a $G$-equivariant isometry 
\[
W\widehat{\otimes}_{\Q_{p}}\mathcal{C}^{\an}(G,\Q_{p})\cong\mathcal{C}^{\an}(G,W),
\]
where $\mathcal{C}^{\an}(G,W)$ is the space of $W$-valued analytic
functions on $G$, with its $G$-Banach structure given by the sup
norm and the action $(gf)(x)=g(f(g^{-1}(x))$ for $f\in\mathcal{C}^{\an}(G,W)$.
We then have $(\mathcal{C}^{\an}(G,W))^{G}=W^{G\han}$, the identification
given by $f\mapsto f(1)$. This gives an alternative description of
$G$-analytic vectors that we shall use in what follows.

The functor $W\mapsto W^{G\han}$ is left exact. Following $\mathsection2.2$
of \cite{Pa21} and \cite{RJRC21}, define right derived functors
for $i\geq0$:
\[
\mathrm{R}_{G\han}^{i}(W)=\H^{i}(G,W\widehat{\otimes}_{\Q_{p}}\mathcal{C}^{\an}(G,\Q_{p}))
\]
(taking continuous cohomology on the right hand side).

If $G$ is a compact $p$-adic Lie group with subgroups $\left\{ G_{n}\right\} _{n\geq1}$
as in $\mathsection2.1-\mathsection2.2$, taking the colimit over
$n$, there are right derived functors for $W\mapsto W^{G\hla}$ given
by
\[
\mathrm{R}_{G\hla}^{i}(W)=\varinjlim_{n}\mathrm{R}_{G_{n}\han}^{i}(W)=\varinjlim_{n}\H^{i}(G_{n},W\widehat{\otimes}_{\Q_{p}}\mathcal{C}^{\an}(G_{n},\Q_{p})).
\]

We shall call these groups the higher locally analytic vectors of
$W$. If $G$ is understood from the context we shall just write $\mathrm{R}_{\la}^{i}$
instead of $\mathrm{R}_{G\hla}^{i}$.

If 
\[
0\rightarrow V\rightarrow W\rightarrow X\rightarrow0
\]
is a short exact sequence of $G$-Banach spaces, then it is strict
by the open mapping theorem, and so we have a long exact sequence
\[
0\rightarrow V^{\la}\rightarrow W^{\la}\rightarrow X^{\la}\rightarrow\mathrm{R}_{\la}^{1}\left(V\right)\rightarrow\mathrm{R}_{\la}^{1}\left(W\right)\rightarrow\mathrm{R}_{\la}^{1}\left(X\right)\rightarrow...
\]

\begin{lem}
\label{2.4}Let $H$ be an open subgroup of $G$ and let $H_{n}=G_{n}\cap H$.
Then for $n\gg0$ and each $i\geq0$ there are natural isomorphisms
$\mathrm{R}_{H_{n}\han}^{i}\cong\mathrm{R}_{G_{n}\han}^{i}$. In particular,
$\mathrm{R}_{H\hla}^{i}\cong\mathrm{R}_{G\hla}^{i}$.
\end{lem}

\begin{proof}
We have $H_{n}=G_{n}$ for $n\gg0$.
\end{proof}
Suppose that $G$ be a small compact $p$-adic Lie group, and let
$H$ be a small closed normal subgroup. Let $W$ be a $G$-Banach
space. Using the method of Hochshild-Serre we obtain the following
spectral sequences.
\begin{prop}
\label{2.5}(i) There is a spectral sequence
\[
E_{2}^{ij}=\H^{i}(G/H,\H^{j}(H,W\widehat{\otimes}_{\Q_{p}}\mathcal{C}^{\an}(G,\Q_{p})))\Rightarrow\mathrm{R}_{G\han}^{i+j}(W).
\]

(ii) There is a spectral sequence
\[
E_{2}^{ij}=\mathrm{R}_{G/H\han}^{i}(\H^{j}(H,W))\Rightarrow\H^{i+j}(G,W\widehat{\otimes}_{\Q_{p}}\mathcal{C}^{\an}(G/H,\Q_{p})).
\]

\end{prop}

\begin{proof}
Apply the Hochshild-Serre spectral sequence to $W\widehat{\otimes}_{\Q_{p}}\mathcal{C}^{\an}(G,\Q_{p})$
and $W\widehat{\otimes}_{\Q_{p}}\mathcal{C}^{\an}(G/H,\Q_{p})$ (see
\cite[Proposition 5.16]{RJRC21}) .
\end{proof}

\section{Equivariant vector bundles}

In this section we give reminders on the Fargues-Fontaine curve and
equivariant vector bundles. For more details, see \cite[Chapter 9]{FF18}
and \cite[Lectures 12-13]{SW20}.

\subsection{The spaces $\mathcal{Y}_{(0,\infty)}$ and $\mathcal{X}$}

Let $F$ be a perfectoid field, with tilt $F^{\flat}$. We have Fontaine's
ring $\A_{\inf}=\A_{\inf}(F),$ defined as the Witt vectors of the
ring of integers $\mathcal{O}_{F}^{\flat}$ of $F^{\flat}$. Write
$\Spa(\A_{\inf})$ for the adic space associated to the Huber pair
$(\A_{\inf},\A_{\inf})$ . 

Let $\varpi$ be a pseudouniformizer of $F$, and let $f$ be the
residue field of $\mathcal{O}_{F}$. Then there is a point $x_{f}\in\Spa(\A_{\inf})$
with residue field $f$, which is the intersection of the two closed
subspaces $\{p=0\}$ and $\{[\varpi]=0\}$. We set 
\[
\mathcal{Y}=\mathcal{Y}(F)=\Spa\A_{\inf}-\{x_{f}\}.
\]

and
\[
\mathcal{Y}_{(0,\infty)}=\mathcal{Y}_{(0,\infty)}(F)=\Spa\A_{\inf}-\{p\left[\varpi\right]=0\}.
\]

The spaces $\mathcal{Y}$ and $\mathcal{Y}_{\left(0,\infty\right)}$
have a Frobenius automorphism $\varphi$ induced from the Witt vectors
strucure of $\A_{\inf}$. 

The space $\mathcal{Y}_{(0,\infty)}$ is a preperfectoid space. The
(adic) Fargues-Fontaine curve associated to $F$ is defined as the
quotient
\[
\mathcal{X}=\mathcal{X}(F)=\mathcal{Y}_{(0,\infty)}(F)/\varphi^{\Z},
\]
which makes sense because the Frobenius action is proper and discontinuous.
The natural projection $\pi:\mathcal{Y}_{(0,\infty)}\rightarrow\mathcal{X}$
is a local isomorphism, so $\mathcal{X}$ is a preperfectoid space,
by virtue of $\mathcal{Y}_{(0,\infty)}$ being so. The space $\mathcal{Y}_{(0,\infty)}$
has a canonical point called $x_{\infty}$, the ``point at infinity''.
It corresponds to the kernel of Fontaine's map
\[
\theta:\A_{\inf}\rightarrow\mathcal{O}_{F},
\]
\[
\sum_{n\geq0}[a_{n}]p^{n}\mapsto\sum_{n\geq0}a_{n}^{\sharp}p^{n},
\]
where for $a\in\mathcal{O}_{F}$, $a^{\sharp}$ is defined to be the
first coordinate of $a\in\mathcal{O}_{F}^{\flat}=\varprojlim_{x\mapsto x^{p}}\mathcal{O}_{F}$.
Identify $x_{\infty}$ with its image $\pi(x_{\infty})\in\mathcal{X}$.
We shall sometimes use the fact that $\ker\theta$ is a prinicipal
ideal, generated by $\xi=\varpi-[\varpi^{\flat}]$ (for example). 

If $F=\widehat{K}_{\infty}$, there is an induced action of the group
$\Gamma=\Gal(K_{\infty}/K)$ on each of the spaces mentioned above,
and the map $\mathcal{Y}_{(0,\infty)}\rightarrow\text{\ensuremath{\mathcal{X}}}$
is $\Gamma$-equivariant. The point $x_{\infty}\in\mathcal{X}$ is
the unique $\Gamma$-fixed point; in fact, it is the unique point
with finite $\Gamma$-orbit (\cite[Proposition 10.1.1]{FF18}). From
now on, if $F$ is omitted from the notation of $\mathcal{Y}_{(0,\infty)}$
and $\mathcal{X}$, we always take $F=\widehat{K}_{\infty}$.

\subsection{The spaces $\mathcal{Y}_{I}$ and $\mathcal{X}_{I}$}

It will be fruitful to consider certain open subsets of $\mathcal{Y}_{(0,\infty)}$
and $\mathcal{X}$. By \cite[Lecture 12]{SW20} there is a surjective
continuous map $\kappa:\mathcal{Y}\rightarrow\left[0,\infty\right]$
given by\footnote{Our normalization of $\kappa$ is the inverse of loc. cit.}
\[
\kappa(x)=\frac{\log\left|p(\widetilde{x})\right|}{\log\left|[\varpi^{\flat}](\widetilde{x})\right|},
\]
where $\widetilde{x}$ is the maximal generization of $x$. For each
interval $I\subset(0,\infty)$, let $\mathcal{Y}_{I}$ be the interior
of the preimage of $\mathcal{Y}$ under $\kappa$. These spaces are
$\Gamma$-stable if such a $\Gamma$ action is present. Furthermore,
the map $\varphi$ induces isomorphisms $\varphi:\mathcal{Y}_{pI}\xrightarrow{\sim}\mathcal{\mathcal{Y}}_{I}$.
Write $\log(I)=\{\log x:x\in I\}.$ Whenever $I$ is sufficiently
small so that the inequality $|\log(I)|<\log(p)$ holds, we have $\overline{I}\cap p\overline{I}=0$
and $\pi$ maps $\mathcal{Y}_{I}$ isomorphically onto its image $\pi(\mathcal{Y}_{I})=\mathcal{X}_{I}\subset\mathcal{X}$.
Note that $x_{\infty}\in\mathcal{X}_{I}$ if and only if $I$ contains
an element of $(p-1)p^{\Z}$, because $\kappa(x_{\infty})=(p-1)/p$.

For $I\subset(0,\infty)$, we have the coordinate rings
\[
\widetilde{\mathrm{B}}_{I}=\widetilde{\mathrm{B}}_{I}(\widehat{K}_{\infty})=\H^{0}(\mathcal{Y}_{I},\mathcal{O}_{\mathcal{Y}_{(0,\infty)}}).
\]
If $I$ is compact, the geometry of $\mathcal{Y}_{I}$ is simple.
\begin{prop}
\label{3.1}Suppose $I\subset(0,\infty)$ is a compact interval.

(i) $\mathcal{Y}_{I}=\Spa(\widetilde{\mathrm{B}}_{I},\widetilde{\A}_{I})$,
where $\widetilde{\A}_{I}$ is the ring of power bounded elements
of $\widetilde{\mathrm{B}}_{I}$. In particular, $\mathcal{Y}_{I}$
is affinoid.

(ii) \textup{$\widetilde{\mathrm{B}}_{I}$} is a principal ideal domain.

(iii) The global sections functor induces an equivalence of categories
between vector bundles on $\mathcal{Y}_{I}$ and finite free $\widetilde{\mathrm{B}}_{I}$-modules.
\end{prop}

\begin{proof}
Parts (i) and (ii) follow from \cite[Théorème 3.5.1]{FF18}. Part
(iii) follows from \cite[Theorem 5.2.8]{SW20} (originally \cite[Theorem 2.7.7]{KL13}),
since finite projective $\widetilde{\mathrm{B}}_{I}$-modules are
finite free.
\end{proof}

\subsection{Equivariant vector bundles}

The action of $\Gamma$ on $\mathcal{X}$ gives an automorphism $\gamma:\mathcal{X}\xrightarrow{\sim}\mathcal{X}$
for each $\gamma\in\Gamma$. 
\begin{defn}
A $\Gamma$-equivariant vector bundle (or simply $\Gamma$-vector
bundle) on $\mathcal{X}$ is a vector bundle $\widetilde{\mathcal{E}}$
on $\mathcal{X}$ equipped with an isomorphism $c_{\gamma}:\gamma^{*}\widetilde{\mathcal{E}}\xrightarrow{\sim}\widetilde{\mathcal{E}}$
for each $\gamma\in\Gamma$ such that the cocycle condition $c_{\gamma_{2}}\circ\gamma_{2}^{*}c_{\gamma_{1}}=c_{\gamma_{1}\gamma_{2}}$
holds for every $\gamma_{1},\gamma_{2},\in\Gamma$.
\end{defn}

Similarly, we have a notion of a $(\varphi,\Gamma)$-vector bundle
on $\mathcal{Y}_{(0,\infty)}$. This consists of a $\Gamma$-vector
bundle $\widetilde{\mathcal{M}}$ on $\mathcal{Y}_{\left(0,\infty\right)}$
together with an additional isomorphism $c_{\varphi}:\varphi^{*}\widetilde{\mathcal{M}}\xrightarrow{\sim}\widetilde{\mathcal{M}}$
such that $c_{\varphi}\circ\varphi^{*}c_{\gamma}=c_{\gamma}\circ\gamma^{*}c_{\varphi}$
for every $\gamma\in\Gamma$. 

Descent along $\varphi$ gives the following.
\begin{prop}
\label{3.3}There is an equivalence of categories
\[
\left\{ \Gamma\text{-vector bundles on }\mathcal{X}\right\} \cong\{(\varphi,\Gamma)\text{-vector bundles on }\mathcal{Y}_{\left(0,\infty\right)}\}.
\]
The equivalence is given by the following functors: if $\widetilde{\mathcal{E}}$
is an equivariant vector bundle, we map it to $\mathcal{\mathcal{O}}_{\mathcal{Y}_{\left(0,\infty\right)}}\otimes_{\mathcal{O}_{\mathcal{X}}}\mathcal{\widetilde{E}}$.
Conversely, if $\widetilde{\mathcal{M}}$ is a $(\varphi,\Gamma)$-vector
bundle on $\mathcal{Y}_{\left(0,\infty\right)}$, we map it to $\pi_{*}(\widetilde{\mathcal{M}})^{\varphi=1}$.
\end{prop}

If $\mathcal{\widetilde{E}}$ is a $\Gamma$-vector bundle on $\mathcal{X}$
and $U\subset\mathcal{X}$ is an open subset stable under $\Gamma$,
there is an induced action of $\Gamma$ on $\H^{0}(U,\mathcal{\widetilde{E}})$.
In particular, there is a natural action of $\Gamma$ on $\H^{0}(\mathcal{X}_{I},\mathcal{\mathcal{\widetilde{E}}})$
when $\left|\log\left(I\right)\right|<\log(p)$. For a general open
subset $U$, one only has a map 
\[
c_{\gamma}:\H^{0}(U,\mathcal{O}_{\mathcal{X}})\otimes_{\H^{0}(\gamma\left(U\right),\mathcal{O}_{\mathcal{X}})}\H^{0}(\gamma\left(U\right),\mathcal{\widetilde{E}})\rightarrow\H^{0}(U,\mathcal{\widetilde{E}}).
\]
Similar remarks apply for $(\varphi,\Gamma)$-equivariant vector bundles
on $\mathcal{Y}_{(0,\infty)}$.
\begin{example}
\label{3.4}Let $V$ be a finite dimensional $\Q_{p}$-representation
of $G_{K}$. Recall that $H=\Gal(K_{\infty}/K)$. Then by \cite[Théorème 10.1.5]{FF18},
\[
\widetilde{\mathcal{E}}(V):=(V\otimes_{\Q_{p}}\mathcal{O}_{\mathcal{X}(\C_{p})})^{H}
\]
is a $\Gamma$-vector bundle on $\mathcal{X}$. More generally, by
loc. cit, the category of finite dimensional $G_{K}$-representations
embeds fully faithfully to the category of $(\varphi,\Gamma)$-modules,
with essential image the subcategory of étale $(\varphi,\Gamma)$-modules.
We can extend the domain of the functor $V\mapsto\widetilde{\mathcal{E}}(V)$
from $G_{K}$-representations to $(\varphi,\Gamma)$-modules. Conversely,
any $\Gamma$-vector bundle on $\mathcal{X}$ gives rise to a $(\varphi,\Gamma)$-module,
and this correspondence results in a equivalence of categories (see
\cite[Préface, Remark 5.10]{FF18}). This will be discussed in detail
in $\mathsection7$.
\end{example}

\section{Locally analytic vector bundles}

In this section, we introduce the category of locally analytic vector
bundles and discuss their basic properties. 

\subsection{Locally analytic functions of $\mathcal{Y}_{(0,\infty)}$ and $\mathcal{X}$}

Let $U\subset\mathcal{X}$ be an open affinoid. Then $U$ is quasicompact
and hence stable under the action of a finite index subgroup $\Gamma'\leq\Gamma$.
The space of functions $\H^{0}(U,\mathcal{O}_{\mathcal{X}})$ is a
Banach $\Gamma'$-ring, and so it makes sense to speak of its subring
of $\Gamma'$-locally analytic functions. This does not depend on
the choice of $\Gamma'$, and so we shall write $\H^{0}(U,\mathcal{O}_{\mathcal{X}})^{\la}$
for the $\Gamma'$-locally analytic functions in $\H^{0}(U,\mathcal{O}_{\mathcal{X}})$
for any $\Gamma'$. Since taking locally analytic vectors is left
exact, these can be glued and we obtain a sheaf of rings $\mathcal{O}_{\mathcal{X}}^{\la}$
on $\mathcal{X}$ that satisfies
\[
\H^{0}(U,\mathcal{O}_{\mathcal{X}}^{\la})=\H^{0}(U,\mathcal{O}_{\mathcal{X}})^{\la}
\]
for every open affinoid $U\subset\mathcal{X}$. 

More generally, suppose $U$ is an open subset of $\mathcal{X}$ which
is not necessarily affinoid, but for which there is an increasing
cover $U=\bigcup_{i}U_{i}$ with each $U_{i}$ affinoid and a single
finite index subgroup $\Gamma'\leq\Gamma$ stabilizing all of the
$U_{i}$ simultaneously. This condition will be satisfied in any situation
we shall consider. Then the sections of $\mathcal{O}_{\mathcal{X}}^{\la}$
on $U$ are the pro analytic functions
\[
\H^{0}(U,\mathcal{O}_{\mathcal{X}}^{\la})=\varprojlim_{i}\H^{0}(U_{i},\mathcal{O}_{\mathcal{X}})^{\la}=\H^{0}(U,\mathcal{O}_{\mathcal{X}})^{\pa}.
\]

\begin{lem}
\label{4.1}The sheaf $\mathcal{O}_{\mathcal{X}}^{\la}$ is stable
for the action of $\Gamma$ on $\mathcal{O}_{\mathcal{X}}$, in the
sense that the inclusion $\mathcal{O}_{\mathcal{X}}^{\la}\subset\mathcal{O}_{\mathcal{X}}$
induces isomorphisms
\[
c_{\gamma}:\gamma^{*}\mathcal{O}_{\mathcal{X}}^{\la}\xrightarrow{\sim}\mathcal{O}_{\mathcal{X}}^{\la}.
\]
\end{lem}

\begin{proof}
The action of $\Gamma$ on $\mathcal{O}_{\mathcal{X}}$ gives rise
to an isomorphism $c_{\gamma}:\gamma^{*}\mathcal{O}_{\mathcal{X}}\xrightarrow{\sim}\mathcal{O}_{\mathcal{X}}$.
Upon taking $U\subset\mathcal{X}$ affinoid, evaluating the morphism
$c_{\gamma}$ at $U$ and taking locally analytic vectors, we get
an induced map $c_{\gamma}(U):\H^{0}(U,\gamma^{*}\mathcal{O}_{\mathcal{X}})^{\la}\xrightarrow{\sim}\H^{0}(U,\mathcal{O}_{\mathcal{X}})^{\la}$.
But this is the same as $\H^{0}(U,\gamma^{*}\mathcal{O}_{\mathcal{X}}^{\la})\xrightarrow{\sim}\H^{0}(U,\mathcal{O}_{\mathcal{X}}^{\la})$
because of the equality $\H^{0}(U,\mathcal{O}_{\mathcal{X}}^{\la})=\H^{0}(U,\mathcal{O}_{\mathcal{X}})^{\la}$.
By writing an arbitrary open set as a union of affinoids, we get the
desired induced isomorphism $c_{\gamma}:\gamma^{*}\mathcal{O}_{\mathcal{X}}^{\la}\xrightarrow{\sim}\mathcal{O}_{\mathcal{X}}^{\la}.$
\end{proof}
The preceding discussion then applies equally well to $\mathcal{Y}_{(0,\infty)}$,
so we have a sheaf $\mathcal{O}_{\mathcal{Y}_{\left(0,\infty\right)}}^{\la}$
of locally analytic functions on $\mathcal{Y}_{(0,\infty)}$ endowed
with isomorphisms $c_{\gamma}$. Since the $\varphi$-action on $\mathcal{Y}_{(0,\infty)}$
commutes with the $\Gamma$-action, it preserves the $\Gamma$-locally
analytic functions, and this gives an isomorphism
\[
c_{\varphi}:\varphi^{*}\mathcal{O}_{\mathcal{Y}_{(0,\infty)}}^{\la}\xrightarrow{\sim}\mathcal{O}_{\mathcal{Y}_{(0,\infty)}}^{\la}
\]
which commutes with the $\Gamma$-action as usual.

\subsection{A flatness result}

For our application at $\mathsection6$ it would be useful to know
the inclusion $\mathcal{O}_{\mathcal{X}}^{\la}\subset\mathcal{O}_{\mathcal{X}}$
is flat. We are only able to establish this in the cyclotomic case
where $K_{\infty}=K_{\cyc}$, and only for certain open subsets. Nevertheless,
this will suffice for our needs.

So in this subsection suppose $K_{\infty}=K_{\cyc}$ and let $I$
be a closed interval of the form $I=[r,s]$ with $r\geq\left(p-1\right)/p$.
We write $\widetilde{\B}_{I,\cyc}$ for $\widetilde{\B}_{I}(\widehat{K}_{\cyc})$
of $\mathsection3.2$. Let $K'_{0}$ be the maximal unramified extension
of $\Q_{p}$ contained in $K_{\cyc}$. Then we write $\B_{I,\cyc,K}$
for the ring of power series $f(T)=\sum_{k\in\Z}a_{k}T^{k}$ with
$a_{k}\in K'_{0}$ , such that $f(T)$ converges on the nonempty annulus
where $\left|T\right|\in I$. By a classical result, $\B_{I,\cyc,K}$
is a principal ideal domain (\cite[Corollaire to Proposition 4]{La62}).
There is an embedding $\B_{I,\cyc,K}\hookrightarrow\widetilde{\B}_{I,\cyc}$
for which $\B_{I,\cyc,K}$ is $\Gamma_{\cyc}$-stable. If $K$ is
unramified over $\Q_{p}$, this embedding can be described as follows:
the variable $T$ is mapped to $\left[\varepsilon\right]-1$, where
$\varepsilon=(1,\zeta_{p},\zeta_{p^{2}},...)\in\widehat{K}_{\cyc}^{\flat}$.
Further, one calculates that $\gamma\left(T\right)=(1+T)^{\chi_{\cyc}(\gamma)}-1$,
so $\B_{I,\cyc,K}$ is indeed stable under the action of $\Gamma_{\cyc}$. 
\begin{prop}
\label{4.2}Suppose $I=[r,(p-1)p^{k-1}]$ with $k\geq$ 1. Then

(i) $\widetilde{\B}_{I,\cyc}^{\la}=\bigcup_{n\geq0}\varphi^{-n}(\B_{p^{n}I,\cyc,K})$.

(ii) $\widetilde{\B}_{I,\cyc}^{\la}$ is a Prüfer domain.

(iii) The natural ring morphism $\widetilde{\B}_{I,\cyc}^{\la}\rightarrow\widetilde{\B}_{I,\cyc}$
is flat.
\end{prop}

\begin{proof}
Part (i) is \cite[Theorem 4.4 (2)]{Be16}. Note that in loc. cit.
this is stated only for $I$ of the form $\left[(p-1)p^{l-1},(p-1)p^{k-1}\right]$,
but the argument given there (see also $\mathsection13$ of \cite{Be21})
is valid for any interval of the form $[r,(p-1)p^{k-1}]$. (ii) follows,
because each $\B_{p^{n}I,\cyc}$ is a principal ideal domain, and
an increasing union of such rings is a Prüfer domain. Finally, the
ring $\widetilde{\B}_{I,\cyc}$ is a domain and hence torsionfree
over the subring $\widetilde{\B}_{I,\cyc}^{\la}$. Part (iii) is established
by recalling that a torsionfree module over a Prüfer domain is flat
(\cite[Proposition 4.20]{La99}).
\end{proof}
\begin{question}
\label{4.3}To what extent do (ii) and (iii) of Proposition \ref{4.2}
hold for coordinate rings of general open subsets in $\mathcal{X}$
and general $K_{\infty}$? We do not expect $\widetilde{\B}_{I}^{\la}$
to be a Prüfer domain when $\Gamma$ has dimension larger than 1.
Nevertheless, it might still be the case that $\widetilde{\B}_{I}^{\la}\rightarrow\widetilde{\B}_{I}$
is flat.
\end{question}

\subsection{Locally analytic vector bundles}
\begin{defn}
\label{4.4}A locally analytic vector bundle on $\mathcal{X}$ is
a locally finite free $\mathcal{O}_{\mathcal{X}}^{\la}$-module $\mathcal{E}$
on $\mathcal{X}$ equipped with an isomorphism $c_{\gamma}:\gamma^{*}\mathcal{\mathcal{E}}\xrightarrow{\sim}\mathcal{\mathcal{E}}$
for each $\gamma\in\Gamma$ such that the cocycle condition $c_{\gamma_{2}}\circ\gamma_{2}^{*}c_{\gamma_{1}}=c_{\gamma_{1}\gamma_{2}}$
holds for every $\gamma_{1},\gamma_{2},\in\Gamma$. We require the
action to be continuous with respect to the locally analytic topology. 
\end{defn}

\begin{example}
\label{4.5}1. Let $\widetilde{\mathcal{E}}$ be a $\Gamma$-vector
bundle on $\mathcal{X}$. Define a sheaf $\mathcal{\widetilde{E}}^{\la}$
by generalizing the definition of $\mathcal{O}_{\mathcal{X}}^{\la}$.
Namely, for every open affinoid $U\subset\mathcal{X}$ choose $\Gamma'\leq\Gamma$
stabilizing $U$. Then $\H^{0}(U,\widetilde{\mathcal{E}})$ is a Banach
$\Gamma'$-ring and it makes sense to speak of $\H^{0}(U,\widetilde{\mathcal{E}})^{\la}$,
which does not depend on the choice of $\Gamma'$. Glue these together
to form a sheaf $\widetilde{\mathcal{E}}^{\la}$. The sheaf $\mathcal{\widetilde{\mathcal{E}}}^{\la}$
is an $\mathcal{O}_{\mathcal{X}}^{\la}$ -module with a $\Gamma$-action.
We shall show in $\mathsection6$ that $\mathcal{\widetilde{\mathcal{E}}}^{\la}$
is locally free and therefore an example of a locally analytic vector
bundle.

2. Conversely, if $\mathcal{E}$ is a locally analytic vector bundle,
we can associate to it a $\Gamma$-vector bundle $\widetilde{\mathcal{E}}=\mathcal{O}_{\mathcal{X}}\otimes_{\mathcal{O}_{\mathcal{X}}^{\la}}\mathcal{E}$.
If $U\subset\mathcal{X}$ is an open affinoid such that $\mathcal{E}|_{U}$
is free, it follows from Proposition \ref{2.1} that 
\[
\H^{0}(U,\mathcal{E})=\H^{0}(U,\widetilde{\mathcal{E}})^{\la},
\]
and so $\mathcal{E}=\widetilde{\mathcal{E}}^{\la}$. This shows that
the functor from $\Gamma$-vector bundles to locally analytic vector
bundles mapping $\widetilde{\mathcal{E}}$ to $\mathcal{\widetilde{\mathcal{E}}}^{\la}$
is essentially surjective.
\end{example}

It follows from Example \ref{4.5}.2 that if $\mathcal{E}$ is a locally
analytic vector bundle, we have an action by derivations
\[
\Lie\left(\Gamma\right)\times\mathcal{E}\rightarrow\mathcal{E},
\]
or, what amounts to the same, a connection
\[
\nabla:\mathcal{E}\rightarrow\mathcal{E}\otimes_{\Q_{p}}\left(\Lie\Gamma\right)^{\vee}
\]
satisfying the identity
\[
\nabla(fx)=\nabla(f)x+f\nabla(x)
\]
for local sections $f$ of $\mathcal{O}_{\mathcal{X}}^{\la}$ and
$x$ of $\mathcal{E}$.
\begin{rem}
\label{4.6}We emphasize that if $U\subset\mathcal{X}$ is an arbitrary
open subset then we have an induced action of $\Lie\left(\Gamma\right)$
on $\H^{0}\left(U,\mathcal{E}\right)$. This is unlike the $\Gamma$-action,
which only maps $\H^{0}\left(U,\mathcal{E}\right)$ to itself if $U$
is $\Gamma$-stable. This is one pleasant aspect of working with locally
analytic vector bundles instead of $\Gamma$-vector bundles.
\end{rem}

Finally, we have the following propositions computing sections of
interest. They will not be used elsewhere in the article. We may define
a locally analytic $\varphi$-vector bundle on $\mathcal{Y}_{(0,\infty)}$
by imitating Definition \ref{4.4}. Then given a $(\varphi,\Gamma)$-vector
bundle $\mathcal{\widetilde{M}}$ on $\mathcal{Y}_{(0,\infty)}$,
one can define a locally analytic $\varphi$-vector bundle $\widetilde{\mathcal{M}}^{\la}$
on $\mathcal{Y}_{(0,\infty)}$ as in Example \ref{4.5}.
\begin{prop}
\label{4.7}Let $\mathcal{\widetilde{E}}$ (resp. $\mathcal{\widetilde{M}}$)
be a $\Gamma$-vector bundle on $\mathcal{X}$ (resp. a $(\varphi,\Gamma)$-vector
bundle on $\mathcal{Y}_{(0,\infty)}$) and let $\mathcal{\widetilde{E}}^{\la}$
(resp. $\mathcal{\widetilde{M}}^{\la}$) be its associated locally
analytic vector bundle (resp. locally analytic $\varphi$-vector bundle).
There are natural isomorphisms

(i) $\H^{0}(\mathcal{\mathcal{Y}}_{I},\widetilde{\mathcal{M}}^{\la})\cong\H^{0}(\mathcal{\mathcal{Y}}_{I},\mathcal{\widetilde{M}})^{\la}$
for $I$ a closed interval.

(ii) $\H^{0}(\mathcal{\mathcal{Y}}_{I},\mathcal{\widetilde{M}}^{\la})\cong\H^{0}(\mathcal{\mathcal{Y}}_{I},\mathcal{\widetilde{M}})^{\pa}$
for $I$ an open interval.

(iii) $\H^{0}(\mathcal{X}_{I},\mathcal{\mathcal{\widetilde{E}}}^{\la})\cong\H^{0}(\mathcal{\mathcal{\mathcal{X}}}_{I},\mathcal{\widetilde{E}})^{\la}$
for $I$ a closed interval with $|\log\left(I\right)|<\log\left(p\right)$.

(iv) $\H^{0}(\mathcal{X}_{I},\mathcal{\mathcal{\widetilde{E}}}^{\la})\cong\H^{0}(\mathcal{\mathcal{\mathcal{X}}}_{I},\mathcal{\mathcal{\widetilde{E}}})^{\pa}$
for $I$ an open interval with $|\log\left(I\right)|<\log\left(p\right)$.

(v) $\H^{0}(\mathcal{X},\mathcal{\mathcal{\widetilde{E}}}^{\la})\cong\H^{0}(\mathcal{X},\mathcal{\widetilde{E}})^{\la}.$

(vi) $\H^{0}(\mathcal{X}-x_{\infty},\mathcal{\widetilde{E}}^{\la})\cong\H^{0}(\mathcal{X}-x_{\infty},\mathcal{\mathcal{\widetilde{E}}})^{\pa}$. 
\end{prop}

\begin{proof}
Parts (i) and (iii) are immediate from the definition. For (ii) and
(iv), use the coverings $\mathcal{Y}_{I}=\bigcup_{J\subset I}\mathcal{Y}_{J}$
and $\mathcal{\mathcal{X}}_{I}=\bigcup_{J\subset I}\mathcal{\mathcal{X}}_{J}$
ranging over $J\subset I$ closed. For (v), consider the covering
\[
\mathcal{X}=\mathcal{\mathcal{X}}_{[1,\sqrt{p}]}\cup\mathcal{\mathcal{X}}_{[\sqrt{p},p]}
\]
with intersection $\mathcal{X}_{[\sqrt{p},\sqrt{p}]}\amalg\mathcal{X}_{[1,1]}$
(identifying $1$ with $p$ via $\varphi$). This yields exact sequences
\[
0\rightarrow\H^{0}(\mathcal{X},\mathcal{\mathcal{\widetilde{E}}}^{\la})\rightarrow\H^{0}(\mathcal{\mathcal{X}}_{[1,\sqrt{p}]},\mathcal{\widetilde{E}}^{\la})\oplus\H^{0}(\mathcal{\mathcal{X}}_{[\sqrt{p},p]},\mathcal{\mathcal{\widetilde{E}}}^{\la})\rightarrow\H^{0}(\mathcal{X}_{[\sqrt{p},\sqrt{p}]},\mathcal{\mathcal{\widetilde{E}}}^{\la})\oplus\H^{0}(\mathcal{X}_{[1,1]},\mathcal{\mathcal{\widetilde{E}}}^{\la})
\]
and
\[
0\rightarrow\H^{0}(\mathcal{X},\mathcal{\mathcal{\widetilde{E}}})^{\la}\rightarrow\H^{0}(\mathcal{\mathcal{X}}_{[1,\sqrt{p}]},\mathcal{\mathcal{\widetilde{E}}})^{\la}\oplus\H^{0}(\mathcal{\mathcal{X}}_{[\sqrt{p},p]},\mathcal{\mathcal{\widetilde{E}}})^{\la}\rightarrow\H^{0}(\mathcal{X}_{[\sqrt{p},\sqrt{p}]},\mathcal{\mathcal{\widetilde{E}}})^{\la}\oplus\H^{0}(\mathcal{X}_{[1,1]},\mathcal{\mathcal{\widetilde{E}}})^{\la}.
\]
By virtue of (iii) the kernels of these sequences are identified.
This proves part (v).

For (vi), use the covering 
\[
\mathcal{X}-x_{\infty}=\mathcal{\mathcal{X}}_{[1,\sqrt{p}]}\cup(\mathcal{\mathcal{X}}_{[\sqrt{p},p]}-x_{\infty})
\]
with intersection $\mathcal{X}_{[\sqrt{p},\sqrt{p}]}\amalg\mathcal{X}_{[1,1]}$.
We may write $\mathcal{\mathcal{X}}_{[\sqrt{p},p]}-x_{\infty}$ as
a union of $\Gamma$-stable rational open subsets
\[
\mathcal{\mathcal{X}}_{[\sqrt{p},p]}-\infty=\cup_{n\geq1}\mathcal{\mathcal{X}}_{[\sqrt{p},p]}\left\{ |\xi|\geq p^{-n}\right\} .
\]
Thus
\[
\H^{0}(\mathcal{\mathcal{X}}_{[\sqrt{p},p]}-x_{\infty},\mathcal{\mathcal{\widetilde{E}}}^{\la})\cong\H^{0}(\mathcal{\mathcal{X}}_{[\sqrt{p},p]}-x_{\infty},\mathcal{\mathcal{\widetilde{E}}})^{\pa}.
\]
Repeating the argument which proved part (v), we conclude.
\end{proof}
We place ourselves in the cyclotomic setting so that $\Gamma=\Gamma_{\cyc}$
and $H=\Gal(\overline{K}/K_{\cyc})$, and we write $\B_{\cris}^{+}(\widehat{K}_{\cyc})=(\B_{\cris}^{+})^{H}$.
Following $\mathsection10.2$ of \cite{FF18}, for $n\in\Z$ take
$\widetilde{\mathcal{E}}=\mathcal{O}_{\mathcal{X}}(n)$ to be the
$\Gamma$-line bundle corresponding to the graded module
\[
\mathcal{\bigoplus}_{m\geq0}\B_{\cris}^{+}(\widehat{K}_{\cyc})^{\varphi=p^{m+n}}.
\]

\begin{prop}
\label{4.8}We have
\[
\H^{0}(\mathcal{X},\mathcal{O}_{\mathcal{X}}(n)^{\la})=\begin{cases}
0 & n<0\\
\Q_{p}\left(n\right) & n\geq0
\end{cases}.
\]
\end{prop}

\begin{proof}
To show this, notice first that 
\[
\H^{0}(\mathcal{X},\mathcal{O}_{\mathcal{X}}(n))=\B_{\cris}^{+}(\widehat{K}_{\cyc})^{\varphi=p^{n}}=\begin{cases}
0 & n<0\\
\Q_{p} & n=0\\
\B_{\cris}^{+}(\widehat{K}_{\cyc})^{\varphi=p^{n}} & n>0
\end{cases}.
\]
If $n>0$ then by \cite[6.4.2]{FF18} there is an exact sequence 
\[
0\rightarrow\Q_{p}\left(n\right)\rightarrow\B_{\cris}^{+,\varphi=p^{n}}\rightarrow\B_{\dR}^{+}/t^{n}\B_{\dR}^{+}\rightarrow0.
\]

Take $H$-invariants and locally analytic vectors. By \cite[Théorème 4.11]{BC16}
we know that $(\B_{\dR}^{+}/t^{n}\B_{\dR}^{+})^{H,\la}=K_{\cyc}[[t]]/t^{n}$,
so we are left with an exact sequence
\[
0\rightarrow\Q_{p}\left(n\right)\rightarrow\B_{\cris}^{+}(\widehat{K}_{\cyc})^{\varphi=p^{n},\la}\rightarrow K_{\cyc}[[t]]/t^{n}.
\]
\textbf{Claim.} $\B_{\cris}^{+}(\widehat{K}_{\cyc})^{\varphi=p^{n},\la}=\Q_{p}(n)$. 

Note that a similar statement appears in $\mathsection3.3$ of \cite{BC16}
in the case $n=1$. Given the claim the computation is finished because
part (v) of the Proposition \ref{4.7} implies that
\[
\H^{0}(\mathcal{X},\mathcal{O}_{\mathcal{X}}(n)^{\la})=\B_{\cris}^{+}(\widehat{K}_{\cyc})^{\varphi=p^{n},\la}=\begin{cases}
0 & n<0\\
\Q_{p}\left(n\right) & n\geq0
\end{cases}.
\]

To show the claim, take $x\in\B_{\cris}^{+}(\widehat{K}_{\cyc})^{\varphi=p^{n},\la}$.
Its image in $K_{\cyc}[[t]]/t^{n}$ is killed by the polynomial
\[
P_{n}(\gamma):=\prod_{i=0}^{n-1}(\chi_{\cyc}(\gamma)^{-i}\gamma-1)
\]
for $\gamma$ which generates an open subgroup of $\Gamma$. It follows
that $P_{n}(\gamma)(x)\in\Q_{p}(n)$ for this $\gamma$. Since $P_{n}(\gamma)$
acts on $\Q_{p}(n)$ by a nonzero element we reduce to showing that
$\B_{\cris}^{+}(\widehat{K}_{\cyc})^{\varphi=p^{n},P_{n}(\gamma)=0}$
is 0. In fact, if $K'$ is the subfield of $K_{\cyc}$ corresponding
to $\gamma^{\Z_{p}}\subset\Gamma$ with maximal unramified subextension
$K_{0}'$, we shall compute that
\[
\B_{\cris}(\widehat{K}_{\cyc})^{P_{n}(\gamma)=0}=\bigoplus_{i=0}^{n-1}K_{0}'t^{i},
\]
and in particular there are no nonzero elements with $\varphi=p^{n}$.

To show this latter description of the elements killed by $P_{n}(\gamma)$,
we argue by induction. If $n=1$ then $P_{n}(\gamma)=\gamma-1$ and
the equality follows from the usual description of the Galois invariants
of $\B_{\cris}$. For $n\geq2$, we have $P_{n}(\gamma)/(\gamma-1)=P_{n-1}(\chi_{\cyc}(\gamma)^{-1}\gamma)$
and
\[
\B_{\cris}(\widehat{K}_{\cyc})^{P_{n-1}(\chi_{\cyc}(\gamma)^{-1}\gamma)=0}=t\B_{\cris}(\widehat{K}_{\cyc})^{P_{n-1}(\gamma)=0}=\bigoplus_{i=1}^{n-1}K_{0}'t^{i}.
\]
Thus there is a commutative diagram
\[
\xymatrix{0\ar[r] & K_{0}'\ar[r]\ar[d]^{\cong} & \bigoplus_{i=0}^{n-1}K_{0}'t^{i}\ar[r]\ar[d] & \bigoplus_{i=1}^{n-1}K_{0}'t^{i}\ar[d]^{\cong}\ar[r] & 0\\
0\ar[r] & \B_{\cris}(\widehat{K}_{\cyc})^{\gamma-1=0}\ar[r] & \B_{\cris}(\widehat{K}_{\cyc})^{P_{n}(\gamma)=0}\ar[r] & \B_{\cris}(\widehat{K}_{\cyc})^{P_{n-1}(\chi_{\cyc}(\gamma)^{-1}\gamma)=0}
}
\]

whose rows are exact and whose outer vertical maps are isomorphisms.
We conclude by the applying the five lemma.
\end{proof}
\begin{rem}
\label{4.9}Set $\B_{e}(\widehat{K}_{\infty})=\B_{e}^{H}$ for the
usual ring $\B_{e}=\B_{\cris}^{\varphi=1}$, so that $\B_{e}\subset\H^{0}(\mathcal{X}-x_{\infty},\mathcal{\mathcal{O}}_{\mathcal{X}})$.
This inclusion is not an equality: the ring $\B_{e}$ allows only
meromorphic functions at $x_{\infty}$ while in $\H^{0}(\mathcal{X}-x_{\infty},\mathcal{\mathcal{O}}_{\mathcal{X}})$
there will be functions with essential singularities. The subring
$\B_{e}(\widehat{K}_{\infty})^{\pa}\subset\H^{0}(\mathcal{X}-x_{\infty},\mathcal{\mathcal{O}}_{\mathcal{X}})^{\la}$
is more tractable and we can understand its structure to an extent.
In particular, let us consider the subring $\B_{e}(\widehat{K}_{\infty})^{\pa}=\B_{e}\cap\H^{0}(\mathcal{X}-x_{\infty},\mathcal{\mathcal{O}}_{\mathcal{X}}^{\la})$
in the case $\Gamma=\Gamma_{\cyc}$. We claim that in fact $\B_{e}(\widehat{K}_{\infty})^{\pa}=\Q_{p}$.
To see this, take $x\in\B_{e}(\widehat{K}_{\infty})^{\pa}$, and restrict
it to $\mathcal{\mathcal{X}}_{[\sqrt{p},p]}-x_{\infty}$. Since $\mathcal{\mathcal{\mathcal{Y}}}_{[\sqrt{p},p]}$
maps isomorphically onto $\mathcal{\mathcal{X}}_{[\sqrt{p},p]}$,
the element $t$ gives an element of $\H^{0}(\mathcal{\mathcal{X}}_{[\sqrt{p},p]}-x_{\infty},\mathcal{O}_{\mathcal{X}}^{\la})$.
Multiplying by a bounded power of $t$, the function $t^{n}x$ extends
to an element of 
\[
\H^{0}(\mathcal{\mathcal{X}}_{[\sqrt{p},p]},\mathcal{O}_{\mathcal{X}}^{\la})=\H^{0}(\mathcal{\mathcal{X}}_{[\sqrt{p},p]},\mathcal{O}_{\mathcal{X}})^{\la},
\]
which shows that $x$ itself is actually an element of $\B_{e}(\widehat{K}_{\infty})^{\la}$,
with a pole of order $n$ at $x_{\infty}$. Therefore, $t^{n}x\in\H^{0}(\mathcal{X},\mathcal{O}_{\mathcal{X}}(n)^{\la})$
which is equal to $\Q_{p}(n)$ as was shown in Proposition \ref{4.8}.
This means $x$ is in $\Q_{p}$ and so $\B_{e}(\widehat{K}_{\infty})^{\pa}=\Q_{p}$. 
\end{rem}

\begin{question}
\label{4.10}1. Is it true that $\H^{0}(\mathcal{X}-x_{\infty},\mathcal{\mathcal{O}}_{\mathcal{X}}^{\la})=\Q_{p}$
if $\Gamma\neq\Gamma_{\cyc}$ and $\dim\Gamma=1$? 

2. If $\dim\Gamma>1$ then one can sometimes produce elements in $\B_{e}(\widehat{K}_{\infty})^{\la}$
which do not belong to $\Q_{p}$. For example, in the Lubin-Tate setting,
the element $(t_{-\sqrt{p}}/t_{\sqrt{p}})^{2}$ lies in $\B_{e}(\widehat{K}_{\infty})^{\la}$
, for $t_{\pm\sqrt{p}}$ being the analogue of Fontaine's element
attached to the uniformizer $\pi=\pm\sqrt{p}$ (see $\mathsection8.3$
of \cite{Co02} for the notation appearing here). Is it true that
in some generality $\B_{e}(\widehat{K}_{\infty})^{\la}$ will be $d-1$
dimensional for $d=\dim\Gamma$? See \cite[Théoréme 6.1]{BC16} for
a related statement.
\end{question}

\section{Acyclicity of locally analytic vectors for semilinear representations}

In this section, we shall prove vanishing the of $\mathrm{R}_{\la}^{i}$-groups
for certain semilinear representations. These results will be used
to prove the descent result in $\mathsection6$ but are also of independent
interest. We follow the strategy of \cite{Pa21}, where the case of
a trivial representation and a particular family of algebras $\widetilde{\Lambda}$
is treated.

\subsection{Statement of the results}

To state the main result of this section, we recall the Tate-Sen axioms
of \cite[3]{BC08}. Let $G$ be a profinite group and let $\WL$ be
a $G$-Banach ring endowed with a valuation $\val$ for which the
$G$ action is continuous and unitary. We suppose there is a character
$\chi:G\rightarrow\Z_{p}^{\times}$ with open image and let $H=\ker\chi$.
Given an open normal subgroup $G_{0}\subset G$ we let $H_{0}=G_{0}\cap H$
and $\Gamma_{H_{0}}=G/H_{0}$.

The Tate-Sen axioms are the following.

\textbf{Axiom (TS1).} There exists $c_{1}>0$ such that for any open
subgroup $H_{1}\subset H_{2}$ of $H_{0}$ there exists $\alpha\in\widetilde{\Lambda}^{H_{1}}$
with $\val(\alpha)>-c_{1}$ and $\sum_{\tau\in H_{2}/H_{1}}\tau(\alpha)=1$.

\textbf{Axiom (TS2).} There exists $c_{2}>0$ and for each $H_{0}$
open in $H$ an integer $n(H_{0})$ depending on $H_{0}$ such that
for $n\geq n(H_{0})$, we have the extra data of
\begin{itemize}
\item Closed subalgebras $\Lambda_{H_{0},n}\subset\widetilde{\Lambda}^{H_{0}}$,
and 
\item Trace maps $\R_{H_{0},n}:\widetilde{\Lambda}^{H_{0}}\rightarrow\Lambda_{H_{0},n}$
\end{itemize}
satisfying:
\begin{enumerate}
\item For $H_{1}\subset H_{2}$ we have $\Lambda_{H_{2},n}\subset\Lambda_{H_{1},n}$
and $\R_{H_{1},n}|_{\Lambda_{H_{2},n}}=\R_{H_{2},n}$.
\item $\R_{H_{0},n}$ is $\Lambda_{H_{0},n}$-linear and $\R_{H_{0},n}(x)=x$
for $x\in\Lambda_{H_{0},n}$.
\item $g(\Lambda_{H_{0},n})=\Lambda_{gH_{0}g^{-1},n}$ and $g(\R_{H_{0},n}(x))=\R_{gH_{0}g^{-1},n}(gx)\text{ if }g\in G$.
\item $\lim_{n\rightarrow\infty}\R_{H_{0},n}(x)=x$ for $x\in\widetilde{\Lambda}^{H_{0}}$.
\item If $n\geq n(H_{0})$ and $x\in\widetilde{\Lambda}^{H_{0}}$ then $\val(R_{H_{0},n}(x))\geq\val(x)-c_{2}$.
\end{enumerate}
\textbf{Axiom (TS3).} There exists $c_{3}>0$ and for each open normal
subgroup $G_{0}$ of $G$ an integer $n(G_{0})\geq n(H_{0})$ such
that if $n\geq n(G_{0})$ and $\gamma\in\Gamma_{H_{0}}$ has $n(\gamma)=\val_{p}(\chi(\gamma)-1)\leq n$,
then $\gamma-1$ acts invertibely on $\mathrm{X}_{H_{0},n}=(1-\R_{H_{0},n})(\widetilde{\Lambda}^{H_{0}})$
and $\val((\gamma-1)^{-1}(x))\geq\val(x)-c_{3}.$

We introduce an additional possible axiom which does not appear in
\cite{BC08}.

\textbf{Axiom (TS4)}. For any sufficiently small open normal $G_{0}\subset G$
with $H_{0}=G_{0}\cap H$ and for any $n\geq n(G_{0})$, there exists
a positive real number $t=t(H_{0},n)>0$ such that if $\gamma\in G_{0}/H_{0}$
and $x\in\text{\ensuremath{\Lambda_{H_{0},n}}}$ then 
\[
\val((\gamma-1)(x))\geq\val(x)+t.
\]

We then have the following result.
\begin{thm}
\label{5.1}Let $M$ be a finite free $\widetilde{\Lambda}$-semilinear
representation of $G$. Suppose there exists an open subgroup $G_{0}\subset G$,
a $G$-stable $\widetilde{\Lambda}^{+}$-lattice $M^{+}\subset M$
and an integer $k>c_{1}+2c_{2}+2c_{3}$ such that in some basis of
$M^{+}$, we have $\Mat(g)\in1+p^{k}\Mat_{d}(\widetilde{\Lambda}^{+})$
for every $g\in G_{0}$. Then

(i). If (TS1)-(TS3) are satisfied then for $i\geq2$
\[
\R_{G\hla}^{i}(M)=0.
\]
In fact, $\R_{G_{0}\han}^{i}(M)=0$ for any sufficiently small open
subgroup $G_{0}\subset G$.

(ii). If in addition (TS4) is satisfied then 
\[
\R_{G\hla}^{1}(M)=0.
\]
In fact, for every sufficiently small open subgroup $G_{0}$ there
is an open subgroup $G_{1}\subset G_{0}$ such that the map $\R_{G_{0}\han}^{1}(M)\rightarrow\R_{G_{1}\han}^{1}(M)$
is $0$.

(iii). In particular, if (TS1)-(TS4) are satisfied then $M$ has no
higher locally analytic vectors. 
\end{thm}

\begin{rem}
\label{5.2}The following was pointed out by the anonymous referee:
if the action of $G_{0}$ on $\WL$ was locally analytic, then the
hypothesis of the existence of $M^{+}$ such that $G_{0}$ acts trivially
mod $p^{k}$ on it would imply that the action of $G_{0}$ on $M$
is locally analytic as well, as it can be deduced from Proposition
\ref{2.1} and Lemma \ref{2.2}. So the non locally analyticity comes
only from the coefficients $\WL$.
\end{rem}

The following special case is often useful in applications.
\begin{prop}
\label{5.3}If $G$ and $\WL$ satisfy (TS1)-(TS4) and if in addition
the topology on $\WL$ is $p$-adic, and if $M$ is a finite free
$\widetilde{\Lambda}$-semilinear representation of $G$, then the
higher locally analytic vectors $\R_{\la}^{i}\left(M\right)$ vanish
for $i\geq1$. 
\end{prop}

\begin{proof}
We shall explain how this follows from Theorem \ref{5.1}. Indeed,
we claim that any finite free $\widetilde{\Lambda}$-semilinear representation
of $G$ satisfies the assumptions of the Theorem \ref{5.1} after
possibly replacing $G$ by a smaller open subgroup $G'$. This suffices
because, by Lemma \ref{2.4}, higher locally analytic vectors do not
change when we replace $G$ by $G'$. 

To see why such a $G'$ exists, suppose $M$ is a finite free $\WL$-semilinear
representation of $G$ and choose any $\WL$-basis $e_{1},..,e_{d}$
of $M$. If we take $M^{+}=\bigoplus_{i=1}^{d}\WL^{+}e_{i}$ then
$M^{+}$ is a lattice of $M$, and by continuity we may find an open
subgroup $G'\subset G$ so that $\Mat(g)\in\GL_{d}(\WL^{+})$ for
$g\in G'$. This implies that $M^{+}$ is $G'$-stable. Since the
topology on $\WL$ is $p$-adic, we can find an open subgroup $G_{0}'\subset G'$
such that $\Mat(g)\in1+p^{k}\Mat_{d}(\widetilde{\Lambda}^{+})$ for
every $g\in G_{0}'$. Thus, the assumptions of Theorem \ref{5.1}
hold for this $M^{+}$, $G'$ and $G_{0}'$.
\end{proof}
Before giving the proof of Theorem \ref{5.1}, we record a few applications.
\begin{cor}
\label{5.4}Suppose $G$ and $\WL$ satisfy (TS1)-(TS4) and let $M$
be as in the statement of the theorem. Then for all $i\geq0$, 
\[
\H^{i}(G,M)\cong\H^{i}(G,M^{\la})\cong\H^{i}(\Lie G,M^{\la})^{G}.
\]
\end{cor}

\begin{proof}
Apply \cite[Corollary 1.6 and Theorem 1.7]{RJRC21}.
\end{proof}
Two main cases of interest are the following. To state them, we set
up some notation first. Let $F$ be an infinitely ramified algebraic
extension of $K$ which contains an unramified twist of the cyclotomic
extension, i.e. the field extension of $K$ cut out by $\eta\chi_{\cyc}$
for $\eta$ an unramified character. Suppose also that $\Gal\left(F/K\right)$
is a $p$-adic Lie group. For why we allow an unramified twist of
the cyclotomic extension on what follows, see $\mathsection8$ of
\cite{Be16}.
\begin{example}
\label{5.5}1. Take $G=\Gal\left(F/K\right)$ and $\WL=\widehat{F}$.
Then $G$ and $\WL$ satisfy the axioms (TS1)-(TS3) for arbitrary
$c_{1}>0$, $c_{2}>0$ and $c_{3}>1/(p-1)$. See \cite[Proposition 4.1.1]{BC08}
for the case $F=\overline{K}$, which goes back to Tate. For general
$F$ the same proof works. 

In addition, we claim that $G$ and $\WL$ satisfy the axiom (TS4).
Indeed, if $G_{0}$ is an open subgroup of $G$ corresponding to a
finite extension $L$ of $K$, then $\Lambda_{H_{0},n}=L(\zeta_{p^{n}})$
and $G_{0}/H_{0}=\Gal(L_{\cyc}/L)$. We take $G_{0}$ sufficiently
small so that $L$ contains $\zeta_{p}$. Let $\pi=\zeta_{p^{n}}-1$
be the uniformizer of $L$. For $\gamma\in\Gal(L_{\cyc}/L)$, we have
\[
\val((\gamma-1)(\pi))=\val(\zeta_{p^{n}}^{\gamma-1}-1)=\frac{1}{(p-1)p^{n-2}}.
\]
Using the identity $(\gamma-1)(ab)=(\gamma-1)(a)b+\gamma(a)(\gamma-1)(b)$,
one then shows by induction that 
\[
\val((\gamma-1)(\pi^{m}))\geq\val(\pi^{m})+\frac{1}{p^{n-2}}.
\]
If $x$ is any element of $\Lambda_{H_{0},n}=L(\zeta_{p^{n}})$, we
may write $x=p^{k}\pi^{m}y$ with $k\in\Z$, $m\geq1$ and $0\leq\val(y)<\val(\pi)$.
Since $\mathcal{O}_{L}[\zeta_{p^{n}}]=\mathcal{O}_{L}[\pi]$, we see
by writing $y$ as a polynomial in $\pi$ that 
\[
\val(\gamma-1)(y)\geq\val(\pi)+\frac{1}{p^{n-2}}.
\]
Using the identity for $\gamma-1$, we have
\[
\begin{aligned}\val(\gamma-1)(x) & \geq k+\min(\val((\gamma-1)(\pi^{m})y),\val(\pi^{m}(\gamma-1)(y)))\\
 & \geq k+\min(\val(\pi^{m})+\val(y)+\frac{1}{p^{n-2}},\val(\pi^{m})+\val(\pi)+\frac{1}{p^{n-2}})\\
 & \geq\val(x)+\frac{1}{p^{n-2}},
\end{aligned}
\]
so (TS4) holds with $t=1/p^{n-2}$.

2. Take $G=\Gal\left(F/K\right)$ and for a closed interval $I\subset(p/p-1,\infty)$
let $\WL=\widetilde{\B}_{I}(\widehat{F})$. Then again $G$ and $\WL$
satisfy the axioms (TS1)-(TS4) for arbitrary $c_{1}>0$, $c_{2}>0$
and $c_{3}>1/(p-1)$. Here if $G_{0}\subset G$ is an open subgroup
corresponding a finite extension $L$ of $K$ then one takes $\Lambda_{H_{0},n}=\varphi^{-n}(\B_{p^{n}I,\cyc,L})$
with notations as in $\mathsection4.2$. For (TS1)-(TS3), see \cite[Proposition 1.1.12]{Be08A}.
(TS4) follows from \cite[Corollary 9.5]{Co08}. 
\end{example}

\begin{cor}
\label{5.6}(i). If $M$ is a finite free $\widehat{F}$-semilinear
representation of $\Gal\left(F/K\right)$ then $\R_{\la}^{i}\left(M\right)=0$
for $i\geq1$.

(ii). If $I\subset(p/p-1,\infty)$ is a closed interval and $M$ is
a finite free $\widetilde{\B}_{I}(\widehat{F})$-semilinear representation
of $\Gal\left(F/K\right)$ then $\R_{\la}^{i}\left(M\right)=0$ for
$i\geq1$.
\end{cor}

\begin{proof}
In both of these cases the topology on $\widetilde{\Lambda}$ is $p$-adic,
so the theorem applies by Proposition \ref{5.3}.
\end{proof}
\begin{rem}
\label{5.7}Suppose $F/K$ is any infinitely ramified $p$-adic Lie
extension of $K$ (not necessarily containing an unramified twist
of the cyclotomic extension), and let $M$ be a finite free $\widehat{F}$-semilinear
representation of $\Gal\left(F/K\right)$. Then $\R_{\la}^{i}\left(M\right)=0$
for $i\geq1$. To prove this, one is always allowed to replace $K$
by a finite extension. Then the extension $FK_{\cyc}/F$ can be assumed
to be either trivial or infinite. In the first case, the group $\R_{\la}^{i}\left(M\right)$
vanishes by the corollary. In the second case, one can argue as in
the proof of \cite[Theorem 3.6.1]{Pa21}. We omit the details since
this result will not be used in the article.
\end{rem}

The rest of the chapter is devoted to the proof of Theorem \ref{5.1}.
The proof is inspired by that of \cite[Theorem 3.6.1]{Pa21}. The
strategy is the following:

1. In $\mathsection5.2$ and $\mathsection5.3$, we establish some
results using (TS1), (TS2) and (TS3) that allow to descend certain
infinite rank $\WL$-semilinear representations of $G$ to $\Lambda_{H_{k},n}^{+}$-semilinear
representations of $G_{0}$, which are fixed by $H_{k}$.

2. In $\mathsection5.4$, we apply these results to $\mathcal{C}^{\an}(G,M)$.

3. Using this and the Hochshild-Serre theorem, we show in $\mathsection5.5$
that $\R_{G\hla}^{i}(M)$ vanishes when $i\geq2$, and we give an
explicit description for $\R_{G\hla}^{1}(M)$. It remains to show
this latter cohomology group vanishes.

4. To do this, we decompose $\R_{G\hla}^{1}(M)$ as a sum of two groups.
For the first one, we use an explicit calculation in $\mathsection5.6$
and (TS4) to show its vanishing. For the second one, we show it is
zero in $\mathsection5.7$ by using again (TS4) and a computation
inspired by \cite{BC16}. Both of these computations are of a $p$-adic
functional analysis flavour.

\subsection{Vanishing of $H$-cohomology}

If $t\in\mathbb{R}$ we write
\[
p^{-t}\widetilde{\Lambda}^{+}:=\text{elements in }\widetilde{\Lambda}\text{ with \ensuremath{\val}}\geq-t.
\]

The first result we shall need for the proof of Theorem \ref{5.1}
is the following. 
\begin{prop}
\label{5.8}Suppose that $(G,H,\widetilde{\Lambda})$ satisfies (TS1)
for some $c_{1}>0$. If $H_{0}\subset H$ is an open subgroup, and
$r\geq1$, we have 

(i) The natural map $\H^{r}(H_{0},\widetilde{\Lambda}^{+})\rightarrow\H^{r}(H_{0},p^{-2c_{1}}\widetilde{\Lambda}^{+})$
is 0. 

(ii) Let $M^{+}$ be a finite free $\widetilde{\Lambda}^{+}$-semilinear
representation of $H_{0}$ which has an $H_{0}$-fixed basis. Then
the map $\H^{r}(H_{0},M^{+})\rightarrow\H^{r}(H_{0},p^{-2c_{1}}M^{+})$
is 0.

(iii) Let $M^{+}=\widehat{\bigcup_{k\in\mathbf{N}}M_{k}^{+}}$ be
the completion of an increasing union of finite free $\widetilde{\Lambda}^{+}$-semilinear
representation of $H_{0}$, each having an $H_{0}$-fixed basis. Then
the map $\H^{r}(H_{0},M^{+})\rightarrow\H^{r}(H_{0},p^{-2c_{1}}M^{+})$
is 0.

In particular, in each of the cases (i)-(iii) the rational cohomology
$\H^{r}(H_{0},M)$ is equal to zero.
\end{prop}

\begin{proof}
We have $(i)\Rightarrow(ii)$, since continuous cohomology commutes
with direct sums.

Next, we prove $(ii)\Rightarrow(iii)$. To do this, observe that if
$t\in\Z_{\geq1}$ then $p^{t}M_{k}^{+}$ also a finite free $\widetilde{\Lambda}^{+}$-semilinear
representation of $H_{0}$ which has an $H_{0}$-fixed basis. Taking
long exact cohomologies of the sequences
\[
0\rightarrow p^{t}(\text{\ensuremath{\bigcup_{k\in\mathbf{N}}M_{k}^{+}}})\rightarrow(\bigcup_{k\in\mathbf{N}}M_{k}^{+})\rightarrow M^{+}/p^{t}M^{+}\rightarrow0
\]
and
\[
0\rightarrow p^{t-2c_{1}}(\bigcup_{k\in\mathbf{N}}M_{k}^{+})\rightarrow p^{-2c_{1}}(\bigcup_{k\in\mathbf{N}}M_{k}^{+})\rightarrow p^{-2c_{1}}M^{+}/p^{t-2c_{1}}M^{+}\rightarrow0,
\]
we get from $(ii)$ that the natural map 
\[
\H^{r}(H_{0},M^{+}/p^{t}M^{+})\rightarrow\H^{r}(H_{0},p^{-2c_{1}}M^{+}/p^{t-2c_{1}}M^{+})
\]
is $0$. Now given a cocycle $\xi\in Z^{r}(H_{0},M^{+}),$ write $\xi_{0}$
for its image in $Z^{r}(H_{0},p^{-2c_{1}}M^{+})$. We wish to show
that $\xi_{0}$ is a coboundary. Choose some fixed $t_{0}\geq3c_{1}$.
Then by virtue of the observation above, the right vertical map of
the commutative diagram
\[
\xymatrix{\H^{r}(H_{0},p^{t_{0}}M^{+})\ar[r]\ar[d] & \H^{r}(H_{0},M^{+})\ar[r]\ar[d] & \H^{r}(H_{0},M^{+}/p^{t_{0}}M^{+})\ar[d]\\
\H^{r}(H_{0},p^{t_{0}-2c_{1}}M^{+})\ar[r] & \H^{r}(H_{0},p^{-2c_{1}}M^{+})\ar[r] & \H^{r}(H_{0},p^{-2c_{1}}M^{+}/p^{t_{0}-2c_{1}}M^{+})
}
\]
is 0, which implies that $\xi_{0}=\xi_{1}+\delta(m_{1})$ where $m_{1}$
is an $r-1$ cocycle valued in $p^{-2c_{1}}M^{+}$ and $\xi_{1}$
is an $r$-cocycle valued in $p^{t_{0}-2c_{1}}M^{+}\subset p^{c_{1}}M^{+}$.
Repeating this argument by induction with $M^{+}$ replaced with $p^{ic_{1}}M^{+}$,
we get that we can write $\xi_{i}=\xi_{i+1}+\delta(m_{i+1})$ where
$\xi_{i}$ is valued in $p^{ic_{1}}M^{+}$ and $m_{i+1}$ is valued
in $p^{(i-3)c_{1}}M^{+}$. Hence the series $\sum_{i=1}^{\infty}m_{i}$
converges to an $r-1$ cocycle $m$ valued in $p^{-2c_{1}}M^{+}$,
and we get $\xi_{0}=\delta(m)$, as required.

Finally, we prove $(i)$. This statement is probably well known, but
for lack of a suitable reference, we provide a proof here. It is essentially
a fiber product of the arguments appearing in \cite[3.2, Corollary 1]{Ta67}
and \cite[Proposition 10.2]{Co08}.

Let $\xi\in Z^{r}(H_{0},\widetilde{\Lambda}^{+})$ be an $r$-cocycle
of $H_{0}$ valued in $\widetilde{\Lambda}^{+}$. By a valuation of
a cochain we shall mean the infimum of its valuation on elements.
Writing $\delta$ for the differential, we shall construct a sequence
of $r-1$ cochains $x_{n}\in C^{r-1}(H_{0},p^{-2c_{1}}\widetilde{\Lambda}^{+})$
for $n\geq-1$ such that 

1. $\val(\xi-\delta x_{n})\geq nc_{1}$ for $\sigma\in H_{0}$, and 

2. $\val(x_{n}-x_{n-1})\geq(n-2)c_{1}$ for $n\geq0$.

This will suffice, since $x_{n}\rightarrow x$ for some $x\in C^{r-1}(H_{0},p^{-2c_{1}}\widetilde{\Lambda}^{+})$
which shows that $\xi=\delta x$ is 0 in $\H^{r}(H_{0},p^{-2c_{1}}\widetilde{\Lambda}^{+})$.

To do this, choose $x_{-1}=0$, which clearly satisfies the first
condition. Suppose $x_{n}$ has been constructed, we construct $x_{n+1}$.
Let $\xi_{n}$ be the $r$-cocycle
\[
\xi_{n}:=\xi-\delta x_{n}
\]
which is valued in $p^{nc_{1}}\widetilde{\Lambda}^{+}$. Choose $H_{1}\subset H_{0}$
an open subgroup such that for every $\sigma_{1},...,\sigma_{r}\in H_{0}$
and $\sigma\in H_{1}$ we have 
\[
\val(\xi_{n}(\sigma_{1},...,\sigma_{r})-\xi_{n}(\sigma_{1},...,\sigma_{r}\sigma))\geq(n+2)c_{1}.
\]
Such a choice is possible by the continuity of $\xi_{n}$ as well
as the compactness of $H_{0}$.

Now by the axiom (TS1) there is an element $\alpha\in\widetilde{\Lambda}^{H_{1}}$
such that $\val(\alpha)>-c_{1}$ and $\sum_{\tau\in H_{0}/H_{1}}\tau(\alpha)=1$.
Let $S$ be a system of representatives for $H_{0}/H_{1}$, and define
an $r-1$ cochain
\[
x_{S}(\sigma_{1},...,\sigma_{r-1})=(-1)^{r}\sum_{\tau\in S}(\sigma_{1}\sigma_{2}\cdot...\cdot\sigma_{r-1}\tau)(\alpha)\xi_{n}(\sigma_{1},...,\sigma_{r-1},\tau).
\]
Each term in the sum has $\val\geq(n-1)c_{1}$, so $\val(x_{S})\geq(n-1)c_{1}.$
In particular, $x_{S}\in C^{r-1}(H_{0},p^{-2c_{1}}\widetilde{\Lambda}^{+})$. 

We now compute $(\xi_{n}-\delta x_{S})(\sigma_{1},...,\sigma_{r})$.
We have by definition of $\delta$ an equation

\begin{equation}
\begin{aligned}\delta x_{S}(\sigma_{1},...,\sigma_{r}) & =(-1)^{r}\sum_{\tau\in S}(\sigma_{1}\cdot...\cdot\sigma_{r}\tau)(\alpha)\sigma_{1}(\xi_{n}(\sigma_{2},...,\sigma_{r},\tau))\\
 & +\sum_{j=1}^{r-1}(-1)^{j+r}\sum_{\tau\in S}(\sigma_{1}\cdot...\cdot\sigma_{r}\tau)(\alpha)\xi_{n}(\sigma_{1},...,\sigma_{j}\sigma_{j+1},...,\sigma_{r},\tau)\\
 & +\sum_{\tau\in S}(\sigma_{1}\cdot...\cdot\sigma_{r-1}\tau)(\alpha)\xi_{n}(\sigma_{1},...,\sigma_{r-1},\tau).
\end{aligned}
\end{equation}
\[
.
\]

On the other hand, $\xi_{n}$ is an $r$-cocycle, so that $\delta\xi_{n}(\sigma_{1},...,\sigma_{r},\tau)=0$
for every $\sigma_{1},...,\sigma_{r}$ and $\tau$. Multiplying by
$(-1)^{r}(\sigma_{1}\cdot...\cdot\sigma_{r}\tau)(\alpha)$ and summing
over $\tau\in S$, we get the equation

\begin{equation}
\begin{aligned}0 & =(-1)^{r}\sum_{\tau\in S}(\sigma_{1}\cdot...\cdot\sigma_{r}\tau)(\alpha)\sigma_{1}(\xi_{n}(\sigma_{2},...,\sigma_{r},\tau))\\
 & +\sum_{j=1}^{r-1}(-1)^{j+r}\sum_{\tau\in S}(\sigma_{1}\cdot...\cdot\sigma_{r}\tau)(\alpha)\xi_{n}(\sigma_{1},...,\sigma_{j}\sigma_{j+1},...,\sigma_{r},\tau)\\
 & +\sum_{\tau\in S}(\sigma_{1}\cdot...\cdot\sigma_{r}\tau)(\alpha)\xi_{n}(\sigma_{1},...,\sigma_{r-1},\sigma_{r}\tau)\\
 & -\sum_{\tau\in S}(\sigma_{1}\cdot...\cdot\sigma_{r}\tau)(\alpha)\xi_{n}(\sigma_{1},...,\sigma_{r}).
\end{aligned}
\end{equation}
Substracting (5.2) from (5.1), we get
\[
\begin{aligned}\delta x_{S}(\sigma_{1},...,\sigma_{r}) & =\sum_{\tau\in S}(\sigma_{1}\cdot...\cdot\sigma_{r-1}\tau)(\alpha)\xi_{n}(\sigma_{1},...,\sigma_{r-1},\tau)\\
 & -\sum_{\tau\in S}(\sigma_{1}\cdot...\cdot\sigma_{r}\tau)(\alpha)\xi_{n}(\sigma_{1},...,\sigma_{r-1},\sigma_{r}\tau)\\
 & +\sum_{\tau\in S}(\sigma_{1}\cdot...\cdot\sigma_{r}\tau)(\alpha)\xi_{n}(\sigma_{1},...,\sigma_{r}).
\end{aligned}
\]
Now by choice of $\alpha$, the last term is simply $\xi_{n}(\sigma_{1},...,\sigma_{r})$.
Thus after rearranging, we have for every $\sigma_{1},...,\sigma_{r}\in H_{0}$
the equation
\[
\begin{aligned}(\xi_{n}-\delta x_{S})(\sigma_{1},...,\sigma_{r}) & =\sum_{\tau\in S}(\sigma_{1}\cdot...\cdot\sigma_{r-1}\tau)(\alpha)\xi_{n}(\sigma_{1},...,\sigma_{r-1},\tau)\\
 & -\sum_{\tau\in S}(\sigma_{1}\cdot...\cdot\sigma_{r}\tau)(\alpha)\xi_{n}(\sigma_{1},...,\sigma_{r}\tau).
\end{aligned}
\]
For each $\tau$ in $S$, let $\sigma_{r,\tau}\in H_{1}$ be such
that $\tau\sigma_{r,\tau}\in\sigma_{r}S$. Then the term on the right
hand side of the previous equation becomes
\[
\sum_{\tau\in S}(\sigma_{1}\cdot...\cdot\sigma_{r-1}\tau)(\alpha)[\xi_{n}(\sigma_{1},...,\sigma_{r-1},\tau)-\xi_{n}(\sigma_{1},...,\tau\sigma_{r,\tau})],
\]
so by the the choice of $H_{1}$ we have
\[
\val(\xi-\delta(x_{n}+x_{S}))=\val(\xi_{n}-\delta x_{S})\geq(n+1)c_{1}.
\]
Finally, set $x_{n+1}:=x_{n}+x_{S}$ where $S$ is arbitrary. The
calculations we have done show that $\val(x_{n+1}-x_{n})\geq(n-1)c_{1}$
and $\val(\xi-\delta x_{n+1})\geq(n+1)c_{1},$ as required. This
concludes the induction and with it the proof.
\end{proof}

\subsection{Descent of semilinear representations}

In this subsection we suppose that $G$ and $\WL$ satisfy the axioms
(TS1), (TS2) and (TS3). 

Given an integer $k>c_{1}+2c_{2}+2c_{3}$ and an open subgroup $G_{0}\subset G$
we write $\Mod_{\widetilde{\Lambda}^{+}}^{k}(G,G_{0})$ for the category
of finite free $\widetilde{\Lambda}^{+}$-semilinear representations
$M^{+}$ of $G$ such that in some basis of $M^{+}$, we have $\Mat(g)\in1+p^{k}\Mat_{d}(\widetilde{\Lambda}^{+})$
for every $g\in G_{0}$. 

The following will allow us to descent coefficients from $\WL^{+}$
to the much smaller ring $\Lambda_{H_{0},n}^{+}=\widetilde{\Lambda}^{+}\cap\Lambda_{H_{0},n}$.
It is a simple modification of \cite[Proposition 3.3.1]{BC08} and
is proved in exactly the same way.
\begin{prop}
\label{5.9}Let $M^{+}\in\Mod_{\widetilde{\Lambda}^{+}}^{k}(G,G_{0})$.
Then for $n\geq n(G_{0})$ and $H_{0}=H\cap G_{0}$ there exists a
unique finite free $\Lambda_{H_{0},n}^{+}$-submodule $\D_{H_{0},n}^{+}(M^{+})$
of $M^{+}$ such that 

(1) $\D_{H_{0},n}^{+}(M^{+})$ is fixed by $H_{0}$ and stable by
$G$.

(2) The natural map $\widetilde{\Lambda}^{+}\otimes_{\Lambda_{H_{0},n}^{+}}\D_{H_{0},n}^{+}(M^{+})\rightarrow M^{+}$
is an isomorphism. In particular, $\D_{H_{0},n}^{+}(M^{+})$ is free
of rank = $\rank M^{+}$.

(3) $\D_{H_{0},n}^{+}(M^{+})$ has a basis which is $c_{3}$-fixed
by $G_{0}/H_{0}$, meaning that for $\gamma\in G_{0}/H_{0}$ we have
$\val(\Mat(\gamma)-1)>c_{3}$.
\end{prop}

\begin{cor}
\label{5.10}Let $M^{+}\in\Mod_{\widetilde{\Lambda}^{+}}^{k}(G,G_{0})$,
$M=M^{+}\otimes_{\WL^{+}}\WL$ and $r\geq1$. The map 
\[
\H^{r}(H_{0},M^{+})\rightarrow\H^{r}(H_{0},p^{-2c_{1}}M^{+})
\]
 is 0 and $H^{r}(H_{0},M)=0$.
\end{cor}

\begin{proof}
This follows from Proposition \ref{5.8} since $M^{+}$ has a basis
fixed\textbf{ }by $H_{0}$.
\end{proof}
\begin{lem}
\label{5.11}Let $H_{0}$ be an open subgroup of $H$, $n\geq n(H_{0})$
an integer, $\gamma\in\Gamma_{H}$ an element such that $n(\gamma)\leq n$
and $B\in\mathrm{M}_{l\times d}(\WL^{H_{0}})$ a matrix. Let $d\in\N\cup\{\infty\}$.
Suppose there are $V_{1}\in\GL_{l}(\Lambda_{H_{0},n})$ and $V_{2}\in\GL_{d}(\Lambda_{H_{0},n})$
such that $\val(V_{1}-1),\val(V_{2}-1)>c_{3}$ and $\gamma(B)=V_{1}BV_{2}$.
Then $B\in\mathrm{M}_{l\times d}(\Lambda_{H_{0},n})$.
\end{lem}

\begin{proof}
The proof is exactly the same as that of \cite[Lemma 3.2.5]{BC08}.
The only difference between that lemma and the statement appearing
here is that there one further assumes $l=d$ and $B\in\GL_{d}(\WL^{H_{0}})$,
but these assumptions are not used in the proof. In fact, the very
same argument shows the result holds for matrices with $d=\infty$,
as long as we understand that an infinite matrix has coefficients
which tend to zero as the indexes tend to $\infty$. Namely, if $R$
is a ring with valuation and $l,d\in\N\cup\left\{ \infty\right\} $,
let $\mathrm{M}_{l\times d}\left(R\right)$ be the set of matrices
$A=(a_{ij})$ of size $l\times d$ and $a_{ij}\in R$ such that $\val(a_{ij})\rightarrow\infty$
as $i+j\rightarrow\infty$. The argument then works in the same way.
\end{proof}
Using Lemma 5.11, we have the following description of $\D_{H_{0},n}^{+}(M^{+})$.
It explains why $\D_{H_{0},n}^{+}(M^{+})$ is functorial in $M^{+}$.
\begin{prop}
\label{5.12}Given $M^{+}\in\Mod_{\widetilde{\Lambda}^{+}}^{k}(G,G_{0})$,
the module $\D_{H_{0},n}^{+}(M^{+})$ is the union of all finitely
generated $\Lambda_{H_{0},n}^{+}$-submodules of $M^{+}$ which are
$G$-stable, $H_{0}$-fixed and admit a $c_{3}$-fixed set of generators.
\end{prop}

\begin{proof}
Indeed, if we have a submodule generated by $c_{3}$-fixed elements
$f_{1},...,f_{l}$ and if $e_{1},...,e_{d}$ is a $c_{3}$-fixed basis,
write 
\[
f_{i}=Be_{i}
\]
for some matrix $B\in\mathrm{M}_{l\times d}(\WL^{H_{0},+})$. Then
we have
\[
\Mat_{f_{i}}(\gamma)B=\gamma(B)\Mat_{e_{i}}(\gamma).
\]
Here by $\Mat_{f_{i}}(\gamma)$ we mean any matrix which represents
the action in terms of the $f_{i}$. It is not a priori unique as
the submodule may not be free. Neverthelss, we have $\val(\Mat_{f_{i}}(\gamma)-1)>c_{3}$
by the assumption, and this implies that $\Mat_{f_{i}}(\gamma)$ is
invertible by \cite[Lemma 3.1.2]{BC08}. So by Lemma 5.11
\[
B\in\mathrm{M}_{l\times d}(\Lambda_{H_{0},n})\cap\mathrm{M}_{l\times d}(\widetilde{\Lambda}^{H_{0},+})=\mathrm{M}_{l\times d}(\Lambda_{H_{0},n}^{+}),
\]
hence the submodule generated by the $f_{i}$ is contained in $\D_{H_{0},n}^{+}(M^{+})$. 
\end{proof}
\begin{cor}
\label{5.13}Let $M^{+},N^{+}\in\Mod_{\widetilde{\Lambda}^{+}}^{k}(G,G_{0})$.
Then for $n\geq n(G_{0})$,

(i) There are natural isomorphisms $\D_{H_{0},n}^{+}(M^{+})\otimes_{\Lambda_{H_{0},n}^{+}}\D_{H_{0},n}^{+}(N^{+})\xrightarrow{\sim}\D_{H_{0},n}^{+}(M^{+}\otimes_{\widetilde{\Lambda}^{+}}N^{+})$
and $\D_{H_{0},n}^{+}(M^{+})\oplus\D_{H_{0},n}^{+}(N^{+})\xrightarrow{\sim}\D_{H_{0},n}^{+}(M^{+}\oplus N^{+}).$

(ii) If $M^{+}\subset N^{+}$ then $\D_{H_{0},n}^{+}(M^{+})=\D_{H_{0},n}^{+}(N^{+})\cap M^{+}$.
\end{cor}

\subsection{Descent of $\mathcal{C^{\protect\an}}(G_{0},M)$}

From here on $G$ is a compact $p$-adic Lie group and $G_{0}\subset G$
is a small subgroup, as in $\mathsection2$. We continue to assume
$G$ and $\WL$ satisfy the axioms (TS1), (TS2) and (TS3). The reader
may also to recall our notations and conventions of $\mathsection1.2$
regarding Banach spaces, completions and tensor products.

By Proposition \ref{2.3}, we have for $V_{l}^{+}=V_{l}(G_{0})\cap\mathcal{C}^{\an}(G_{0},\Q_{p})^{+}$
an equality
\[
\widehat{\varinjlim_{l\in\mathbb{N}}V_{l}^{+}}=\mathcal{C}^{\an}(G_{0},\Q_{p})^{+}.
\]
For $M\in\Mod_{\widetilde{\Lambda}^{+}}^{k}(G,G_{0})$ we have
\[
(\varinjlim_{l\in\mathbb{N}}M^{+}\otimes_{\Z_{p}}V_{l}^{+})^{\wedge}\cong M^{+}\widehat{\otimes}_{\Z_{p}}\mathcal{C}^{\an}(G_{0},\Q_{p})^{+}.
\]
Each $M^{+}\otimes_{\Z_{p}}V_{l}^{+}$ is a finite free $\widetilde{\Lambda}^{+}$-semilinear
representation of $G_{0}$. The action of $G_{k}$ on each of the
$V_{l}^{+}$ is trivial mod $p^{k}$ by Lemma \ref{2.2}, and hence
its action on $M^{+}\otimes V_{l}^{+}$ is trivial mod $p^{k}$. So
if $n\geq n(G_{k})$, we may define using Proposition \ref{5.8} a
$\Lambda_{H_{k},n}^{+}$-submodule of $M^{+}\widehat{\otimes}_{\Z_{p}}\mathcal{C}^{\an}(G_{0},\Q_{p})^{+}$
given by
\[
\D_{H_{k},n,\infty}^{+}(M^{+}):=(\varinjlim_{l\in\mathbb{N}}\D_{H_{k},n}^{+}(M^{+}\otimes V_{l}^{+}))^{\wedge}.
\]
The module $\D_{H_{k},n,\infty}^{+}(M^{+})$ is then $G_{0}$-stable
and fixed by $H_{k}$. By Proposition \ref{5.8} we have natural isomorphisms
\[
\widetilde{\Lambda}^{+}\otimes_{\Lambda_{H_{k},n}^{+}}\D_{H_{k},n}^{+}(M^{+}\otimes V_{l}^{+})\xrightarrow{\sim}M^{+}\otimes V_{l}^{+}.
\]
 This shows that $\D_{H_{k},n,\infty}^{+}\left(M^{+}\right)$ is
generated by $c_{3}$-fixed elements which give it the sup norm, and
there is an isometry 
\[
\widetilde{\Lambda}^{+}\widehat{\otimes}_{\Lambda_{H_{k},n}^{+}}\D_{H_{k},n,\infty}^{+}(M^{+})\xrightarrow{\sim}M^{+}\widehat{\otimes}_{\Z_{p}}\mathcal{C}^{\an}(G_{0},\Q_{p})^{+}.
\]

The next Proposition follows from Proposition 5.12.
\begin{prop}
\label{5.14}A finitely generated $\Lambda_{H_{k},n}^{+}$-submodule
of $M^{+}\widehat{\otimes}_{\Z_{p}}\mathcal{C}^{\an}(G_{0},\Q_{p})^{+}$
which is stable by $G_{0}$, fixed by $H_{k}$ and is generated by
a $c_{3}$-fixed set of elements is contained in $\D_{H_{k},n,\infty}^{+}(M^{+})$.
\end{prop}

In particular, we have the function $\log$ defined, by abuse of notation
as the composition of 
\[
\chi:G_{0}\twoheadrightarrow G_{0}/H_{0}\hookrightarrow\Z_{p}^{\times}
\]
and $\log:\Z_{p}^{\times}\rightarrow\Q_{p}$. It lies in $\mathcal{C}^{\an}(G_{0},\Q_{p})^{+}$.
Note that for $g\in G_{0}$, we have
\[
g(\log)=\log+\log(g^{-1})=\log-\log(g).
\]

\begin{lem}
\label{5.15}The elements $1$ and $\log$ of $\widetilde{\Lambda}^{+}\widehat{\otimes}\mathcal{C}^{\an}(G_{0},\Q_{p})^{+}$
lie in $\D_{H_{k},n,\infty}^{+}(\widetilde{\Lambda}^{+})$.
\end{lem}

\begin{proof}
The $\Lambda_{H_{k},n}^{+}$-submodule generated by $1$ and $\log$
in $\widetilde{\Lambda}^{+}\widehat{\otimes}\mathcal{C}^{\an}(G_{0},\Q_{p})^{+}$
is stable under the $G_{0}$ action and fixed by $H_{k}$. Furthermore,
we claim the elements $1$ and $\log$ are $c_{3}$-fixed by the action
of $G_{k}/H_{k}$. This is clear for 1. To show this for $\log$,
notice that if $g^{p^{k}}\in G_{k}/H_{k}$ (recalling that $G_{k}=G_{0}^{p^{k}})$
then 
\[
\val(g^{p^{k}}-1)(\log)\geq k>c_{1}+2c_{2}+2c_{3}>c_{3}.
\]

We conclude by Proposition \ref{5.14}.
\end{proof}
\begin{prop}
\label{5.16}(i) $\D_{H_{k},n,\infty}^{+}(\widetilde{\Lambda}^{+})$
is a subring of $\widetilde{\Lambda}^{+}\widehat{\otimes}\mathcal{C}^{\an}(G_{0},\Q_{p})^{+}$.

(ii) The module structure of $M^{+}\widehat{\otimes}\mathcal{C}^{\an}(G_{0},\Q_{p})^{+}$
over $\widetilde{\Lambda}^{+}\widehat{\otimes}\mathcal{C}^{\an}(G_{0},\Q_{p})^{+}$
restricts to a module structure of $\D_{H_{k},n,\infty}^{+}(M^{+})$
over $\D_{H_{k},n,\infty}^{+}(\widetilde{\Lambda}^{+})$.
\end{prop}

\begin{proof}
$\D_{H_{k},n,\infty}^{+}(\widetilde{\Lambda}^{+})$ contains $1$
by Proposition \ref{5.14}. Next, one has the ring and module structure
maps 
\[
\widetilde{\Lambda}^{+}\otimes\widetilde{\Lambda}^{+}\rightarrow\widetilde{\Lambda}^{+},\widetilde{\Lambda}^{+}\otimes M^{+}\rightarrow M^{+}.
\]
Applying Proposition 5.12, taking the inductive limit and then taking
completions, we get natural maps 
\[
\D_{H_{k},n,\infty}^{+}(\widetilde{\Lambda}^{+})\otimes\D_{H_{k},n,\infty}^{+}(\widetilde{\Lambda}^{+})\rightarrow\D_{H_{k},n,\infty}^{+}(\widetilde{\Lambda}^{+})
\]
and
\[
\D_{H_{k},n,\infty}^{+}(\widetilde{\Lambda}^{+})\otimes\D_{H_{k},n,\infty}^{+}(M^{+})\rightarrow\D_{H_{k},n,\infty}^{+}(M^{+}),
\]
giving the desired ring and module structures.
\end{proof}

\subsection{Computation of higher locally analytic vectors I}

Let $M^{+}\in\Mod_{\widetilde{\Lambda}^{+}}^{k}(G,G_{0})$ and $M=M^{+}\otimes_{\WL^{+}}\WL$.
In this subsection we shall do a first simplifcation towards the computation
of of the groups $\R_{G\hla}^{i}(M)$ for $i\geq1$. 

If $G_{0}$ is any open subgroup of $G$, we have $\R_{G\hla}^{i}(M)=\R_{G_{0}\hla}^{i}(M)$
so that if $G_{n}=G_{0}^{p^{n}}$we have
\[
\R_{G\hla}^{i}(M)=\varinjlim_{n}\H^{i}(G_{n},M\widehat{\otimes}_{\Q_{p}}\mathcal{C}^{\an}(G_{n},\Q_{p})).
\]
Upon possibly making $G_{0}$ smaller, we may assume that $G_{0}$
is small and that $\chi:G_{0}/H_{0}\rightarrow\Z_{p}^{\times}$ has
image isomorphic to $\Z_{p}$. Write $\Gamma_{n}=G_{n}/H_{n}$.
\begin{lem}
\label{5.17}For $i\geq1$, 
\[
\H^{i}(G_{n},M\widehat{\otimes}_{\Q_{p}}\mathcal{C}^{\an}(G_{n},\Q_{p}))\cong\H^{i}(\Gamma_{n+k},(M\widehat{\otimes}_{\Q_{p}}\mathcal{C}^{\an}(G_{n},\Q_{p}))^{H_{n+k}})
\]
\end{lem}

\begin{proof}
By the Hochshild-Serre spectral sequence and the vanishing of $H_{n+k}$
cohomologies in (iii) of Proposition \ref{5.8} (taking the inductive
system $M_{k+k'}^{+}=M^{+}\otimes V_{k+k'}^{+}$ for $k'\geq0$),
we have
\[
\H^{i}(G_{n},M\widehat{\otimes}_{\Q_{p}}\mathcal{C}^{\an}(G_{n},\Q_{p}))\cong\H^{i}(G_{n}/H_{n+k},(M\widehat{\otimes}_{\Q_{p}}\mathcal{C}^{\an}(G_{n},\Q_{p}))^{H_{n+k}}).
\]
Now the inclusion $\Gamma_{n+k}\hookrightarrow G_{n}/H_{n+k}$ induces
an isomorphism 
\[
\H^{i}(G_{n}/H_{n+k},(M\widehat{\otimes}_{\Q_{p}}\mathcal{C}^{\an}(G_{n},\Q_{p}))^{H_{n+k}})\cong\H^{i}(\Gamma_{n+k},(M\widehat{\otimes}_{\Q_{p}}\mathcal{C}^{\an}(G_{n},\Q_{p}))^{H_{n+k}}).
\]
This again follows from Hochshild-Serre, once we notice all the higher
cohomologies of $G_{n}/G_{n+k}$ appearing vanish. This is because
$G_{n}/G_{n+k}$ is finite and the coefficients are rational.
\end{proof}
\begin{cor}
\label{5.18}$\R_{G_{n}\han}^{i}(M)=0$ for $i\geq2$ and $n\geq0$.
\end{cor}

\begin{proof}
Because $\Gamma_{n+k}\cong\Z_{p}$.
\end{proof}
This proves the first part of Theorem \ref{5.1}. It remains to study
the 1st derived group
\[
\R_{G\hla}^{1}(M)=\varinjlim_{n}\H^{1}(\Gamma_{n+k},(M\widehat{\otimes}_{\Q_{p}}\mathcal{C}^{\an}(G_{n},\Q_{p}))^{H_{n+k}}).
\]

Now for $m\geq n(G_{n+k})$, we have by Proposition \ref{5.9} a natural
isomorphism
\[
\widetilde{\Lambda}^{+}\otimes\varinjlim_{\ell}\mathbf{D}_{H_{k},n}^{+}(M^{+}\otimes V_{\ell}^{+})\cong M^{+}\otimes\varinjlim_{\ell\in\mathbb{N}}V_{\ell}^{+}.
\]
Taking the $p$-adic completion, we obtain a natural isomorphism
\[
\widetilde{\Lambda}^{+}\widehat{\otimes}_{\Lambda_{H_{n+k},m}^{+}}\D_{H_{n+k},m,\infty}^{+}(M^{+})\xrightarrow{\sim}M^{+}\widehat{\otimes}\mathcal{C}^{\an}(G_{n},\Q_{p})^{+}
\]

and thus
\[
\widetilde{\Lambda}^{+,H_{n+k}}\widehat{\otimes}_{\Lambda_{H_{n+k},m}^{+}}\D_{H_{n+k},m,\infty}^{+}(M^{+})\xrightarrow{\sim}(M^{+}\widehat{\otimes}\mathcal{C}^{\an}(G_{n},\Q_{p})^{+})^{H_{n+k}}.
\]
On the other hand, recall we have the trace maps 
\[
\R_{H_{n+k},m}:\widetilde{\Lambda}^{H_{n+k}}\rightarrow\Lambda_{H_{n+k},m}
\]
which induce for $\mathrm{X}_{H_{n+k},m}=\ker\R_{H_{n+k},m}$ a decomposition
\[
\widetilde{\Lambda}^{H_{n+k}}=\Lambda_{H_{n+k},m}\oplus X_{H_{n+k},m}.
\]
Therefore, we can decompose
\[
\widetilde{\Lambda}^{H_{n+k}}\widehat{\otimes}_{\Lambda_{H_{n+k},m}}\D_{H_{n+k},m,\infty}\left(M\right)\cong\D_{H_{n+k},m,\infty}\left(M\right)\oplus(\mathrm{X}_{H_{n+k},m}\widehat{\otimes}_{\Lambda_{H_{n+k},m}}\D_{H_{n+k},m,\infty}\left(M\right)),
\]
and so we get the description
\[
\R_{G\hla}^{1}(M)=\varinjlim_{n}\H^{1}(\Gamma_{n+k},\D_{H_{n+k},m,\infty}\left(M\right))\oplus\H^{1}(\Gamma_{n+k},\mathrm{X}_{H_{n+k},m}\widehat{\otimes}_{\Lambda_{H_{n+k},m}}\D_{H_{n+k},m,\infty}^{+}\left(M\right)).
\]
where in each object of the direct limit, we take $m\geq n(G_{n+k})$.

\subsection{Computation of higher locally analytic vectors II}

If $m\geq0$ is an integer and $\gamma$ is an element of a group,
write $\gamma_{m}$ for $\gamma^{p^{m}}$. The following simple lemma
will be used to compare the behaviour of $(\gamma-1)^{m}$ and $\gamma_{m}-1$.
\begin{lem}
\label{5.19}Let $\ell\geq0$. The element $X^{p^{\ell}}-1$ of the
ring $\Z_{p}[X]$ is in the ideal generated by the elements $p^{i}(X-1)^{\ell+1-i}$
for $0\leq i\leq\ell$.
\end{lem}

\begin{proof}
For $\ell\geq1$ we have
\[
\begin{aligned}\begin{aligned}X^{p^{\ell}}-1 & =(X^{p^{\ell-1}}-1)(\sum_{i=1}^{p-1}X^{ip^{\ell-1}})\\
 & =(X^{p^{\ell-1}}-1)(\sum_{i=1}^{p-1}1+(X^{ip^{\ell-1}}-1))\\
 & =(X^{p^{\ell-1}}-1)(p+\sum_{i=1}^{p-1}(X^{ip^{\ell-1}}-1)),
\end{aligned}
\end{aligned}
\]
so that $X^{p^{\ell}}-1$ lies in the ideal 
\[
(X^{p^{\ell-1}}-1)(p,(X^{p^{\ell-1}}-1))=(p(X^{p^{\ell-1}}-1),(X^{p^{\ell-1}}-1)^{2}).
\]

Let $I_{\ell}$ be the ideal generated by the elements $p^{i}(X-1)^{\ell+1-i}$
for $0\leq i\leq\ell$. It is easy to check that $(pI_{\ell-1},I_{\ell-1}^{2})$
is contained in $I_{\ell}$. Hence, induction on $\ell$ shows that
$X^{p^{\ell}}-1$ belong to $I_{\ell}$.
\end{proof}
So far we have only used the axioms (TS1), (TS2) and (TS3). We shall
now use the final axiom (TS4), which proves us with a positive number
$t>0$.
\begin{prop}
\label{5.20}If (TS4) holds, then 

(i) $\Lambda_{H,n}$ is $\Gamma_{t}$-analytic for an open subgroup
of $\Gamma$ depending on $t$.

(ii) There exists an element $s=s(t,c_{3})=s(n,m,G_{0},c_{3})$ such
that for $\gamma\in G_{n+k}/H_{n+k}$ we have
\[
(\gamma-1)\D_{H_{n+k},m,\infty}^{+}(M^{+})\subset p^{s}\D_{H_{n+k},m,\infty}^{+}(M^{+}).
\]
(iii) $\D_{H_{n+k},m,\infty}(M)$ is $\Gamma$-analytic for some open
subgroup $\Gamma$ of $\Gamma_{n+k}$ which depends on $n,m,G$ and
$c_{3}$.
\end{prop}

\begin{proof}
Once (ii) is established, we claim parts (i) and (iii) follow from
\cite[Example 2.1.9]{Pa21}. Let us elaborate a little bit. Take $\ell$
large enough so that 
\[
(\ell-i)+(i+1)t=\ell+t+(t-1)i\geq2
\]
for each $0\leq i\leq\ell$. Then for such $\ell$ (which only depends
on $t)$ we have by Lemma \ref{5.19}
\[
(\gamma_{\ell}-1)(\Lambda_{H,n}^{+})\subset p^{2}\Lambda_{H,n}^{+},
\]
So that if $b\in\Lambda_{H,n}$, the series
\[
\gamma_{\ell}^{x}(b)=\sum_{n\geq0}{x \choose n}(\gamma_{\ell}-1)^{n}(b)
\]
converges. This shows $b$ is analytic for the subgroup generated
by $\gamma_{\ell}$. The argument for (iii) given (ii) is similar.

To show part (ii), recall the identity
\[
(\gamma-1)\left(ab\right)=(\gamma-1)(a)b+\gamma(a)(\gamma-1)(b).
\]
(TS4) implies that if $a\in\Lambda_{H,m}^{+}$ and $b\in\D_{H_{n+k},m,\infty}^{+}(M^{+})$
is $c_{3}$-fixed, then $ab$ is $\min(c_{3},t)$-fixed. Since the
$c_{3}$-fixed elements topologically generate $\D_{H_{n+k},m,\infty}^{+}(M^{+})$,
it follows that every element of $\D_{H_{n+k},m,\infty}^{+}(M^{+})$
is $s=\min(c_{3},t)$-fixed.
\end{proof}
Using this we can show
\begin{lem}
\label{5.21}Given $n$ there is $m$ sufficiently large depending
only on $n$ (and not on $M$) such that $\H^{1}(\Gamma_{n+k},\mathrm{X}_{H_{n+k},m}\widehat{\otimes}_{\Lambda_{H_{n+k},m}}\D_{H_{n+k},m,\infty}\left(M\right))=0$.
\end{lem}

\begin{proof}
(This argument is adapted from \cite[Lemma 3.6.6]{Pa21}) Fix $m_{0}\geq n(G_{n+k})$.
From the discussion after Corollary \ref{5.18}, for $m\geq m_{0}$
we have a natural isomorphism
\[
\widetilde{\Lambda}^{H_{n+k}}\widehat{\otimes}_{\Lambda_{H_{n+k},m}}\D_{H_{n+k},m,\infty}(M)\cong\D_{H_{n+k},m,\infty}(M)\oplus(\mathrm{X}_{H_{n+k},m}\widehat{\otimes}_{\Lambda_{H_{n+k},m}}\D_{H_{n+k},m,\infty}(M)).
\]

By Proposition \ref{5.12}, we have an isomorphism 
\[
\Lambda_{H_{n+k},m}\widehat{\otimes}\D_{H_{n+k},m_{0},\infty}(M)\cong\D_{H_{n+k},m,\infty}(M).
\]
Let $\mathrm{X}_{H_{n+k},m}^{+}=\mathrm{X}_{H_{n+k},m}\cap\widetilde{\Lambda}^{+}$.
We get an induced isomorphism
\[
\mathrm{X}_{H_{n+k},m}^{+}\widehat{\otimes}_{\Lambda_{H_{n+k},m}^{+}}\D_{H_{n+k},m,\infty}^{+}(M^{+})\cong\mathrm{X}_{H_{n+k},m}^{+}\widehat{\otimes}_{\Lambda_{H_{n+k},m_{0}}^{+}}\D_{H_{n+k},m_{0},\infty}^{+}(M).
\]
Let $\gamma$ be a generator of $\Gamma_{n+k}$. By Proposition \ref{5.20},
there is some $s$ such that
\[
(\gamma-1)\D_{H_{n+k},m_{0},\infty}^{+}(M^{+})\subset p^{s}\D_{H_{n+k},m_{0},\infty}^{+}(M^{+}).
\]
If $\ell$ is sufficiently large Proposition \ref{5.20} implies that
\[
(\gamma_{\ell}-1)\D_{H_{n+k},m_{0},\infty}^{+}(M^{+})\subset p^{2c_{3}}\D_{H_{n+k},m_{0},\infty}^{+}\left(M^{+}\right)
\]
(we take $2c_{3}$ rather than $c_{3}$ to take of convergence later
in this argument). Choose such an $\ell$, and take $m$ large enough
so that $n(\gamma_{\ell})\leq m$. Then by (TS3) we have $\val((\gamma_{\ell}-1)^{-1}\left(x\right))\geq\val(x)-c_{3}$
for $x\in\mathrm{X}_{H_{n+k},m}^{+}$.

We will now show that any element of $\mathrm{X}_{H_{n+k},m}\widehat{\otimes}_{\Lambda_{H_{n+k},m}}\D_{H_{n+k},m,\infty}\left(M\right)$
is in the image of $\gamma_{\ell}-1$. This will also imply any element
is in the image of $\gamma-1$, since $\gamma_{\ell}-1$ is divisible
by $\gamma-1$, and hence it will further imply that the cohomology
\[
\H^{1}(\Gamma_{n+k},\mathrm{X}_{H_{n+k},m}\widehat{\otimes}_{\Lambda_{H_{n+k},m}}\D_{H_{n+k},m,\infty}\left(M\right))\cong\mathrm{X}_{H_{n+k},m}\widehat{\otimes}_{\Lambda_{H_{n+k},m}}\D_{H_{n+k},m,\infty}\left(M\right)/(\gamma-1)
\]
is 0. 

To do this last step, it suffices to show that each simple tensor
\[
a\otimes b\in\mathrm{X}_{H_{n+k},m}^{+}\widehat{\otimes}_{\Lambda_{H_{n+k},m_{0}}^{+}}\D_{H_{n+k},m_{0},\infty}^{+}(M^{+})\cong\mathrm{X}_{H_{n+k},m}^{+}\widehat{\otimes}_{\Lambda_{H_{n+k},m}^{+}}\D_{H_{n+k},m,\infty}^{+}(M^{+})
\]
is in the image of $\gamma_{\ell}-1$. Choose an integer $r$ so that
$p^{r}a$ is in the image of $(\gamma_{l}-1)^{-1}$ restricted to
$\mathrm{X}_{H_{n+k},m}^{+}$ (choose any $r\geq c_{3}$). It suffices
to show $p^{r}a\otimes b$ is in the image of $\gamma_{\ell}-1$.
So write $p^{r}a=(\gamma_{\ell}-1)^{-1}(c)$ for $c\in\mathrm{X}_{H_{n+k},m}^{+}$,
and consider the series
\[
y=\sum_{i=0}^{+\infty}(\gamma_{l}^{-1}-1)^{-i}(c)\otimes(\gamma_{l}-1)^{i}(b)=\sum_{i=0}^{+\infty}\gamma_{l}^{i}(1-\gamma_{l})^{-i}(c)\otimes(\gamma_{l}-1)^{i}(b).
\]
This series converges, because by our choices $\val((\gamma_{\ell}-1)^{-1}\left(x\right))\geq\val(x)-c_{3}$
on $X_{H_{n+k},m}^{+}$ and $(\gamma_{\ell}-1)(x)\geq\val(x)+2c_{3}$
on $\D_{H_{n+k},m_{0},\infty}^{+}(M^{+})$! A direct computation then
gives 
\[
(\gamma_{\ell}-1)(y)=(\gamma_{\ell}-1)\left(c\right)\otimes b=p^{r}a\otimes b,
\]
so $p^{r}a\otimes b$ is in the image of $\gamma_{\ell}-1$, as required.
\end{proof}
Combing Lemma \ref{5.21} with the discussion after Corollary \ref{5.18},
we get the following description of $\R_{G\hla}^{1}(M)$.
\begin{prop}
\label{5.22}We have
\[
\R_{G\hla}^{1}(M)=\varinjlim_{n,m}\H^{1}(\Gamma_{n+k},\D_{H_{n+k},m,\infty}\left(M\right))
\]

where the direct limit is taken over pairs $n,m$.
\end{prop}

\subsection{Computation of higher locally analytic vectors III}

We are now almost ready to prove our theorem. First we prove a lemma
that will be used.
\begin{lem}
\label{5.23}Let $\Gamma=\gamma^{\Z_{p}}$ and let $B$ be a Banach
representation of $\Gamma$. Suppose $B=B^{\Gamma\han}$, and that
\[
\left|\left|\gamma-1\right|\right|<p^{-\frac{1}{p-1}}.
\]
Then $\left|\left|b\right|\right|=\left|\left|b\right|\right|_{\Gamma\han}$
for any $b\in B$.
\end{lem}

\begin{proof}
We have for $x\in\Z_{p}$ that
\[
\gamma^{x}(b)=\sum\frac{\nabla_{\gamma}^{k}(b)}{k!}x^{k}
\]
where $\nabla_{\gamma}=\log(\gamma)$. By definition
\[
\left|\left|b\right|\right|_{\Gamma\han}=\sup_{k\geq0}\left\{ \left|\left|\nabla_{\gamma}^{k}(b)/k!\right|\right|\right\} .
\]
Now recall we have
\[
\nabla_{\gamma}=\left(\gamma-1\right)\sum_{m\geq0}\left(-1\right)^{m}\frac{\left(\gamma-1\right)^{m}}{m+1},
\]
so $\left|\left|\nabla_{\gamma}\left(b\right)\right|\right|\leq\left|\left|\gamma-1\right|\right|\left|\left|b\right|\right|$,
and more generally
\[
\left|\left|\nabla_{\gamma}^{k}\left(b\right)\right|\right|\leq\left|\left|\gamma-1\right|\right|^{k}\left|\left|b\right|\right|.
\]
It follows that for $k\geq1$ we have
\[
\left|\left|\nabla_{\gamma}^{k}(b)/k!\right|\right|\leq p^{-\frac{k}{p-1}}\left|\left|\gamma-1\right|\right|^{k}\left|\left|b\right|\right|<\left|\left|b\right|\right|,
\]
so that $\left|\left|b\right|\right|_{\Gamma\han}=\left|\left|b\right|\right|$.
\end{proof}
\emph{Proof of Theorem }\ref{5.1}\emph{.} By the Proposition $\R_{G\hla}^{1}(M)=\varinjlim_{n,m}\H^{1}(\Gamma_{n+k},\D_{H_{n+k},m,\infty}(M))$.
Fix $n$ and $m$. Given $b\in\D_{H_{n+k},m,\infty}(M)$ we shall
show it becomes zero in some $\H^{1}(\Gamma_{l+k},\D_{H_{l+k},m',\infty}(M))$
for some $\ell\geq n$, $m'\geq m$ - this will show the direct limit
is zero. By Proposition \ref{5.20} we know there is an open subgroup
$\Gamma\subset\Gamma_{n+k}$ such that $\D_{H_{n+k},m,\infty}(M)$
is $\Gamma$-analytic. Writing $\gamma$ for a generator of $\Gamma$,
we may take $\Gamma$ small enough so that $\left|\left|\gamma-1\right|\right|<p^{-\frac{1}{p-1}}$,
and hence Lemma \ref{5.23} applies. Thus, writing $\left|\left|\cdot\right|\right|_{n}$
for the norm on $\D_{H_{n+k},m,\infty}(M)$ induced from its inclusion
into $M\widehat{\otimes}\mathcal{C}^{\an}(G_{n},\Q_{p})$, we have
$\left|\left|b\right|\right|_{n}=\left|\left|b\right|\right|_{\Gamma\han}$
for $b\in\D_{H_{n+k},m,\infty}\left(M\right)$. We know there is a
real number $D>0$ such that if $b\in\D_{H_{n+k},m,\infty}\left(M\right)$
then
\[
\left|\left|\nabla_{\gamma}\left(b\right)\right|\right|_{n}=\left|\left|\nabla_{\gamma}\left(b\right)\right|\right|_{\Gamma\han}\leq D\left|\left|b\right|\right|_{\Gamma\han}=D\left|\left|b\right|\right|_{n}.
\]
Now choose $\ell\geq n$ such that $\Gamma_{l}$ has index $p^{t}$
in $\Gamma$, where $t$ is taken large enough so that
\[
2p^{\frac{1}{p-1}}D\leq p^{t}.
\]
Let $\gamma_{t}=\gamma^{p^{t}}$ be the generator of $\Gamma_{\ell}$,
and let $\log_{\ell}\in\mathcal{C}^{\an}(G_{\ell},\Q_{p}):G_{\ell}\twoheadrightarrow G_{\ell}/H_{\ell}\rightarrow\Z_{p}$
be the logarithm so that $\log_{\ell}(\gamma_{t})=1$. Now let $m'\geq m$
be large enough so that $\D_{H_{\ell+k},m',\infty}\left(M\right)$
is defined. Recall that by Lemma \ref{5.15}, $\log_{\ell}\in\D_{H_{\ell+k},m',\infty}(\widetilde{\Lambda}^{+})$.
Let $\Gamma'\subset\Gamma_{\ell+k}$ be an open subgroup so that $\D_{H_{l+k},m',\infty}\left(M\right)$
is $\Gamma'$-analytic and write $p^{q}$ for the index of $\Gamma'$
in $\Gamma_{\ell+k}$. Finally, write $\gamma'$ for the generator
of $\Gamma'$. Again by making $\Gamma'$ smaller we may assume $\left|\left|\gamma'-1\right|\right|<p^{\frac{-1}{p-1}}$
on $\D_{H_{\ell+k},m',\infty}\left(M\right)$. We have 
\[
\gamma'=(\gamma_{t}^{p^{k}})^{p^{q}}=\gamma^{p^{t+k+q}}.
\]
Let $z_{\ell}=\log_{\ell}/p^{k+q}\in\D_{H_{\ell+k},m',\infty}(\widetilde{\Lambda})$,
the one computes that $\gamma'(z_{\ell})=z_{\ell}+1$. Therefore,
$\nabla_{\gamma'}(z_{\ell})=1$. Now consider the series
\[
bz_{\ell}-\nabla_{\gamma'}(b)\frac{z_{\ell}^{2}}{2!}+\nabla_{\gamma'}^{2}(b)\frac{z_{\ell}^{3}}{3!}-...
\]
in $\D_{H_{\ell+k},m',\infty}\left(M\right)$. We claim first it converges
with respect to the norm $\left|\left|\cdot\right|\right|_{\ell}$
of $\D_{H_{\ell+k},m',\infty}\left(M\right)$. Indeed, we have
\[
\left|\left|z_{\ell}\right|\right|_{\ell}=p^{k+q}
\]
and (noting that $\nabla_{\gamma'}^{i}=p^{i\left(t+k+q\right)}\nabla_{\gamma}^{i})$
\[
\left|\left|\nabla_{\gamma'}^{i}(b)\right|\right|_{\ell}=p^{-i\left(t+k+q\right)}\left|\left|\nabla_{\gamma}^{i}(b)\right|\right|_{\ell}\leq p^{-i\left(t+k+q\right)}\left|\left|\nabla_{\gamma}^{i}(b)\right|\right|_{n}
\]
\[
\leq p^{-i\left(t+k+q\right)}D^{i}\left|\left|b\right|\right|_{n},
\]
so the general term of series has size 
\[
\left|\left|\nabla_{\gamma'}^{i}(b)/\left(i+1\right)!\cdot z_{\ell}^{i+1}\right|\right|_{\ell}\ll p^{-i\left(t+k+q\right)}D^{i}p^{i(k+q)}p^{\frac{i}{p-1}}
\]
\[
=(p^{-t}Dp^{\frac{1}{p-1}})^{i}\leq2^{-i},
\]
so the series converges in the in the $\left|\left|\cdot\right|\right|_{\ell}$
norm. But then the series must also converge with respect to $\left|\left|\cdot\right|\right|_{\Gamma'\han}$
because of Lemma \ref{5.23}. So if we write $y$ for the sum of the
series, it makes sense to speak of the derivative $\nabla_{\gamma'}(y)$,
and one computes that $\nabla_{\gamma'}(y)=b$. So $b$ is in the
image of $\nabla_{\gamma'}:\D_{H_{\ell+k},m',\infty}\left(M\right)\rightarrow\D_{H_{\ell+k},m',\infty}\left(M\right)$,
hence also in the image of $\gamma'-1$, which divides $\nabla_{\gamma'}$.
But $\gamma'=\gamma_{t+k}^{p^{q}}$ so $\gamma_{t+k}-1$ divides $\gamma'-1$.
It follows that $b$ is also in the image of $\gamma_{t+k}-1$. This
means that $b$ is 0 in 
\[
\D_{H_{\ell+k},m',\infty}\left(M\right)/(\gamma_{t+k}-1)\cong\H^{1}(\Gamma_{\ell+k},\D_{H_{\ell+k},m',\infty}\left(M\right))
\]
and we are done! $\square$
\begin{rem}
\label{5.24}1. Since the choices of $\ell$ and $m'$ did not depend
on $b$, each $\D_{H_{n+k},m,\infty}\left(M\right)$ maps in its entirety
to 0 in some $\D_{H_{l+k},m',\infty}(M)$. This shows that $M$ is
strongly $\mathfrak{LA}$-acyclic in the sense of \cite[2.2]{Pa21}.
After this work was completed, Pan proved that strong $\mathfrak{LA}$-acyclicity
is in fact automatic in this setting, see \cite[Proposition 2.3.6]{Pa22}.

2. The proof of Theorem \ref{5.1} shows the vanishing of $\varinjlim_{n,m}\H^{1}(\Lie(\Gamma_{n+k}),\D_{H_{n+k},m,\infty}(M))$,
which is a priori stronger than the vanishing of $\varinjlim_{n,m}\H^{1}(\Gamma_{n+k},\D_{H_{n+k},m,\infty}(M))$.
\end{rem}

\section{Descent to locally analytic vectors}

Work again in the setting of $\mathsection3-\mathsection4$. We shall
assume in this section that $K_{\infty}$ contains an unramified twist
of the cyclotomic extension. The purpose of this section is to prove
the following theorem. 
\begin{thm}
\label{6.1}The functor $\mathcal{E}\mapsto\mathcal{O}_{\mathcal{X}}\otimes_{\mathcal{O}_{\mathcal{X}}^{\la}}\mathcal{E}$
gives rise to an equivalence of categories 
\[
\left\{ \text{locally analytic vector bundles on }\mathcal{X}\right\} \cong\left\{ \text{\ensuremath{\Gamma}-vector bundles on }\mathcal{X}\right\} .
\]
The inverse functor is given by $\widetilde{\mathcal{E}}\mapsto\mathcal{\widetilde{E}}^{\la}$.
\end{thm}

In the rest of this section, we shall prove that given a $\Gamma$-vector
bundle $\widetilde{\mathcal{E}}$ on $\mathcal{X}$, the natural map
\[
\mathcal{O}_{\mathcal{X}}\otimes_{\mathcal{O}_{\mathcal{X}}^{\la}}\mathcal{\widetilde{\mathcal{E}}}^{\la}\rightarrow\widetilde{\mathcal{E}}
\]
 is an isomorphism. This is enough for proving Theorem \ref{6.1}.
Indeed, if this isomorphism is granted, then in particular it follows
from Proposition \ref{2.1} that $\mathcal{\widetilde{\mathcal{E}}}^{\la}$
is locally free over $\mathcal{O}_{\mathcal{X}}^{\la}$, so that the
functor $\widetilde{\mathcal{E}}\mapsto\mathcal{\widetilde{E}}^{\la}$
is valued in the correct category and is fully faithful. On the other
hand, it follows from Example \ref{4.5}(2) that it is also essentially
surjective.

\subsection{Computations at the stalk}

In this section, w let $\widetilde{\mathcal{E}}$ be a $\Gamma$-vector
bundle. We have the fiber $\mathcal{\widetilde{E}}_{k(x_{\infty})}$
at $x_{\infty}$, a finite dimensional $\widehat{K}_{\infty}$-semilinear
representation of $\Gamma$, and the completed stalk $\widetilde{\mathcal{E}}_{x_{\infty}}^{\wedge,+}$,
a finite free $\B_{\dR}^{+}(\widehat{K}_{\infty})=\B_{\dR}^{+,H}$-module.
We define
\[
\D_{\mathrm{Sen}}(\widetilde{\mathcal{E}})=(\mathcal{\widetilde{E}}_{k(x_{\infty})})^{\la}
\]
and
\[
\D_{\dif}^{+}(\mathcal{\widetilde{E}})=(\widetilde{\mathcal{E}}_{x_{\infty}}^{\wedge,+})^{\pa}.
\]
If $V$ is a $p$-adic representation and $\mathcal{\widetilde{\mathcal{E}}}=\widetilde{\mathcal{E}}(V)$
as in Example \ref{3.4}, and if $\Gamma=\Gamma_{\cyc}$, then we
recover the classical invariant $\D_{\Sen}(V)$ according to \cite[Théorème 3.2]{BC16}.
The invariant $\D_{\dif}^{+}\left(V\right)$ is also recovered, see
\cite[Proposition 3.3.]{Po20}. It is therefore natural to extend
these definitions to arbitrary $\widetilde{\mathcal{E}}$ and $\Gamma$
as we have done here.

There is the following decompletion result.
\begin{thm}
\label{6.2}(i) The natural map $\widehat{K}_{\infty}\otimes_{\widehat{K}_{\infty}^{\la}}\D_{\mathrm{Sen}}(\widetilde{\mathcal{E}})\rightarrow\widetilde{\mathcal{E}}_{k(x_{\infty})}$
is an isomorphism.

(ii) The natural map $\B_{\dR}^{+}(\widehat{K}_{\infty})\otimes_{\B_{\dR}^{+}(\widehat{K}_{\infty})^{\pa}}\D_{\dif}^{+}(\mathcal{\widetilde{E}})\rightarrow\widetilde{\mathcal{E}}_{x_{\infty}}^{\wedge,+}$
is an isomorphism.
\end{thm}

\begin{proof}
The fiber $\mathcal{\widetilde{\mathcal{E}}}_{k(x_{\infty})}$ is
a finite dimensional $\widehat{K}_{\infty}$-semilinear representation
of $\Gamma$. So (i) follows from \cite[Théorème 3.4]{BC16}. For
(ii), write $I_{\theta}$ for the maximal ideal of $\B_{\dR}^{+}(\widehat{K}_{\infty})$.
It suffices to prove that for $n\geq1$ the natural map
\[
\left(*\right)\ \ \B_{\dR}^{+}(\widehat{K}_{\infty})/I_{\theta}^{n}\otimes_{(\B_{\dR}^{+}/I_{\theta}^{n})^{\la}}(\mathcal{\widetilde{\mathcal{E}}}_{x_{\infty}}/I_{\theta}^{n})^{\la}\rightarrow\widetilde{\mathcal{E}}_{x_{\infty}}/I_{\theta}^{n}
\]
is an isomorphism. 

By Theorem \ref{5.1} (more precisely, Corollary \ref{5.6}(i)), we
have $\mathrm{R}_{\la}^{1}(I_{\theta}^{n-1}\widetilde{\mathcal{E}}_{x_{\infty}}/I_{\theta}^{n})=0$,
so by devissage the map 
\[
(\mathcal{\widetilde{\mathcal{E}}}_{x_{\infty}}/I_{\theta}^{n})^{\la}\rightarrow(\mathcal{\widetilde{\mathcal{E}}}_{x_{\infty}}/I_{\theta})^{\la}=\D_{\mathrm{Sen}}(\widetilde{\mathcal{E}})
\]
 is surjective. It follows from the case $n=1$ and Nakayama's lemma
that $\left(*\right)$ is surjective too. 

For injectivity, we argue as follows. Let $\overline{e}_{1},...,\overline{e}_{d}$
be a basis of $\D_{\mathrm{Sen}}(\widetilde{\mathcal{E}})$ over the
field $\widehat{K}_{\infty}^{\la}$. By what was just proved, we may
choose a lifting $e_{1},...,e_{d}$ of this basis to $(\mathcal{\widetilde{\mathcal{E}}}_{x_{\infty}}/I_{\theta}^{n})^{\la}$.
Then $1\otimes e_{1},...,1\otimes e_{d}$ generate 
\[
\B_{\dR}^{+}(\widehat{K}_{\infty})/I_{\theta}^{n}\otimes_{(\B_{\dR}^{+}/I_{\theta}^{n})^{\la}}(\mathcal{\widetilde{\mathcal{E}}}_{x_{\infty}}/I_{\theta}^{n})^{\la}
\]
according to Nakayama's lemma. 

Now suppose that 
\[
\sum x_{i}\otimes e_{i}\in\B_{\dR}^{+}(\widehat{K}_{\infty})/I_{\theta}^{n}\otimes_{(\B_{\dR}^{+}/I_{\theta}^{n})^{\la}}(\mathcal{\widetilde{\mathcal{E}}}_{x_{\infty}}/I_{\theta}^{n})^{\la}
\]
is in the kernel of $(*)$, so its image is $0$ mod $I_{\theta}^{n}$.
Choose a generator $\xi$ of $I_{\theta}$. Reducing mod $I_{\theta}$
and using the injectivity of $(*)$ for $n=1$, we get the relation
$\sum\overline{x}_{i}\otimes\overline{e}_{i}=0$. As the $\overline{e}_{i}$
form a basis, each $x_{i}$ must be divisible by $\xi$. Writing $x_{i}=\xi x_{i}'$,
we have 
\[
\sum x_{i}\otimes e_{i}=\sum\xi x_{i}'\otimes e_{i}=\xi\sum x_{i}'\otimes e_{i},
\]
so the image of 
\[
\sum x_{i}'\otimes y_{i}\in\B_{\dR}^{+}(\widehat{K}_{\infty})/I_{\theta}^{n-1}\otimes_{(\B_{\dR}^{+}/I_{\theta}^{n-1})^{\la}}(\mathcal{\widetilde{\mathcal{E}}}_{x_{\infty}}/I_{\theta}^{n-1})^{\la}
\]
 in $\mathcal{\widetilde{\mathcal{E}}}_{x_{\infty}}/I_{\theta}^{n-1}$
is $0$. The injectivity now follows from induction.\textbf{}
\end{proof}
Let $I$ be a closed interval with $\left|\log\left(I\right)\right|<\log(p)$
and let 
\[
\widetilde{M}_{I}=\H^{0}(\mathcal{X}_{I},\widetilde{\mathcal{E}}).
\]
Theorem \ref{5.1} of $\mathsection5$ allows us to prove the following
proposition. In \ref{6.5} below we shall prove a stronger statement.
\begin{prop}
\label{6.3}There are natural isomorphisms $\D_{\mathrm{Sen}}(\mathcal{\widetilde{\mathcal{E}}})\cong\widetilde{M}_{I}^{\la}/(I_{\theta}\widetilde{M}_{I})^{\la}$
and $\D_{\dif}^{+}(\mathcal{\widetilde{\mathcal{E}}})\cong\varprojlim_{n}\widetilde{M}_{I}^{\la}/(I_{\theta}^{n}\widetilde{M}_{I})^{\la}$.
\end{prop}

\begin{proof}
As $I_{\theta}$ is principal, $I_{\theta}\widetilde{M}_{I}$ is finite
free over $\widetilde{\B}_{I}$. By Corollary \ref{5.6}(ii), the
cohomology $\mathrm{R}_{\la}^{1}(I_{\theta}\widetilde{M}_{I})$ vanishes.
Applying $\la$ to the short exact sequence 
\[
0\rightarrow I_{\theta}\widetilde{M}_{I}\rightarrow\widetilde{M}_{I}\rightarrow\widetilde{M}_{I}I_{\theta}/\widetilde{M}_{I}\rightarrow0
\]
we get $\widetilde{M}_{I}^{\la}/(I_{\theta}\widetilde{M}_{I})^{\la}\xrightarrow{\sim}(\widetilde{M}_{I}/I_{\theta}\widetilde{M}_{I})^{\la}=\D_{\Sen}(\mathcal{\mathcal{\widetilde{\mathcal{E}}}})$,
which gives the first isomorphism. By the same argument $\widetilde{M}_{I}^{\la}/(I_{\theta}^{n}\widetilde{M}_{I})^{\la}\xrightarrow{\sim}(\widetilde{M}_{I}/I_{\theta}^{n}\widetilde{M}_{I})^{\la}$
for $n\geq1$. To get the second isomorphism, take the limit over
$n$.
\end{proof}

\subsection{Descent to locally analytic vectors }

In this subsection we will give a proof of Theorem \ref{6.1}. We
continue with the notation of $\mathsection6.1$.

We start with the following key proposition, which builds upon all
of the work done in $\mathsection4$, $\mathsection5$ and the previous
subsections of $\mathsection6$.
\begin{prop}
\label{6.4}Let $I=[r,(p-1)p^{n}]$ be an interval with $n\geq1$
and $\left|\log(I)\right|<\log(p)$. Then the natural map 
\begin{equation}
\widetilde{\B}_{I}\otimes_{\widetilde{\B}_{I}^{\la}}\widetilde{M}_{I}^{\la}\rightarrow\widetilde{M}_{I}
\end{equation}
 is an isomorphism.
\end{prop}

\begin{proof}
First let us explain how to reduce to the cyclotomic case. After an
unramified twist, which causes no obstructions to descent, we may
assume $K_{\cyc}\subset K_{\infty}$. Set
\[
\widetilde{M}_{I,\cyc}:=\widetilde{M}_{I}^{\Gal(K_{\infty}/K_{\cyc})}
\]
We then have 
\[
\widetilde{M}_{I}\cong\widetilde{\B}_{I}\otimes_{\widetilde{\B}_{I,\cyc}}\widetilde{M}_{I,\cyc}
\]
(see for example \cite[Corollarie 3.2.2]{BC08}), and if the conclusion
of the Proposition holds for the cyclotomic case, we have 
\[
\widetilde{M}_{I,\cyc}\cong\widetilde{\B}_{I,\cyc}\otimes_{\widetilde{\B}_{I,\cyc}^{\la}}\widetilde{M}_{I,\cyc}^{\la}
\]
and hence
\[
\widetilde{M}_{I}\cong\widetilde{\B}_{I}\otimes_{\widetilde{\B}_{I,\cyc}^{\la}}\widetilde{M}_{I,\cyc}^{\la}.
\]
This shows that $\widetilde{M}_{I}$ has a basis of locally analytic
vectors and by Proposition \ref{2.1} the map $(6.1)$ is an isomorphism. 

It remains to establish the proposition in the cyclotomic case where
$\widetilde{\B}_{I}=\widetilde{\B}_{I,\cyc}$. By Proposition \ref{4.2},
$\widetilde{\B}_{I,\cyc}$ is flat as a $\widetilde{\B}_{I,\cyc}^{\la}$-module.
Since $\widetilde{M}_{I,\cyc}^{\la}$ is torsionfree as a $\widetilde{\B}_{I,\cyc}^{\la}$-module,
it follows from \cite[0AXM]{Sta} that $\widetilde{\B}_{I,\cyc}\otimes_{\widetilde{\B}_{I,\cyc}^{\la}}\widetilde{M}_{I,\cyc}^{\la}$
is also torsionfree. By Proposition \ref{6.3}, the completion at
$I_{\theta}\subset\widetilde{\B}_{I,\cyc}$ of $(6.1)$ is nothing
but the map
\[
\B_{\dR}^{+}\otimes_{\B_{\dR}^{+,\pa}}\D_{\dif}^{+}(\widetilde{\mathcal{E}})\rightarrow\widetilde{\mathcal{E}}_{x_{\infty}}^{\wedge,+},
\]
so by Theorem \ref{6.2}, the map $(6.1)$ is an isomorphism at least
after taking this completion. As $\widetilde{\B}_{I,\cyc}$ is a PID
(cf. Proposition \ref{3.1}), it follows that $(6.1)$ is injective
with cokernel supported at finitely many maximal ideals. These maximal
ideals correspond to a finite set of points on $\mathcal{X}$, and
this set must form a finite orbit under the action of $\Gamma$. But
by \cite[Proposition 10.1.1]{FF18}, the only point with finite orbit
under the $\Gamma$-action is $x_{\infty}$! Thus the cokernel of
$(6.1)$ is supported at $I_{\theta}$. But then it must be $0$,
as we have just shown the completion at $I_{\theta}$ is an isomorphism. 
\end{proof}
\emph{Proof of Theorem }\ref{6.1}. Let $U$ be an open subaffinoid
of $\mathcal{X}_{I}$ for $I=[r,(p-1)p^{n}]$. Then we claim that
the natural map
\[
\mathcal{O}_{\mathcal{X}}\left(U\right)\otimes_{\mathcal{O}_{\mathcal{X}}^{\la}\left(U\right)}\H^{0}(U,\mathcal{\widetilde{E}}^{\la})\rightarrow\H^{0}(U,\widetilde{\mathcal{E}})
\]
is an isomorphism. Indeed, we have
\[
\H^{0}(U,\widetilde{\mathcal{E}})\cong\mathcal{O}_{\mathcal{X}}\left(U\right)\otimes_{\widetilde{\B}_{I,\cyc}}\widetilde{M}_{I,\cyc}\cong\mathcal{O}_{\mathcal{X}}\left(U\right)\otimes_{\widetilde{\B}_{I,\cyc}^{\la}}\widetilde{M}_{I,\cyc}^{\la}.
\]

Thus $\H^{0}(U,\widetilde{\mathcal{E}})$ has a basis of locally analytic
elements. By Proposition \ref{2.1}, we have an isomorphism
\[
\mathcal{O}_{\mathcal{X}}\left(U\right)\otimes_{\mathcal{O}_{\mathcal{X}}(U)^{\la}}\H^{0}(U,\widetilde{\mathcal{E}})^{\la}\rightarrow\H^{0}(U,\widetilde{\mathcal{E}}),
\]
from which the claim follows.

Now let $(\mathcal{O}_{\mathcal{X}}\otimes_{\mathcal{O}_{\mathcal{X}}^{\la}}\mathcal{\widetilde{\mathcal{E}}}^{\la})^{\circ}$
be the presheaf on $\mathcal{X}$ sending 
\[
U\mapsto\mathcal{O}_{\mathcal{X}}\left(U\right)\otimes_{\mathcal{O}_{\mathcal{X}}^{\la}\left(U\right)}\mathcal{\widetilde{\mathcal{E}}}^{\la}\left(U\right).
\]
The $\mathcal{X}_{I}$ for various $I$'s of the form $I=[r,(p-1)p^{n}]$
with $|\log(I)|<\log(p)$ give a covering of $\mathcal{X}$, so the
claim shows that the natural map
\[
(\mathcal{O}_{\mathcal{X}}\otimes_{\mathcal{O}_{\mathcal{X}}^{\la}}\mathcal{\widetilde{\mathcal{E}}}^{\la})^{\circ}\rightarrow\widetilde{\mathcal{E}}
\]
is an isomorphism on stalks. Theorem \ref{6.1} follows. $\square$

The proof of Theorem \ref{6.1} essentially shows that $\mathcal{E}$
is quasi-coherent. This leads to a simple interpertation of $\D_{\Sen}$
and $\D_{\dif}^{+}$ in terms of $\mathcal{E}$ as follows. Given
a locally analytic vector bundle define 
\[
\D_{\Sen}\left(\mathcal{E}\right)=\mathcal{E}_{k(x_{\infty})},
\]
the fiber of $\mathcal{E}$ at $x_{\infty}$, and
\[
\D_{\dif}^{+}\left(\mathcal{E}\right)=\widehat{\mathcal{E}}_{x_{\infty}}^{+},
\]
the completed stalk of $\mathcal{E}$ at $x_{\infty}$. These would
not \emph{a priori }be the same as $\D_{\mathrm{Sen}}(\mathcal{\widetilde{E}})$
and $\D_{\dif}^{+}(\mathcal{\widetilde{E}})$, because quotients in
general do not commute with locally analytic vectors, but they do
in this case.
\begin{thm}
\label{6.5}Let $\widetilde{\mathcal{E}}=\mathcal{O}_{\mathcal{X}}\otimes_{\mathcal{O}_{\mathcal{X}}^{\la}}\mathcal{E}$.
There are natural isomorphisms $\D_{\mathrm{Sen}}(\mathcal{\widetilde{E}})\cong\D_{\mathrm{Sen}}(\mathcal{E})$
and $\D_{\dif}^{+}(\mathcal{\widetilde{E}})\cong\D_{\dif}^{+}\left(\mathcal{E}\right)$.
\end{thm}

\begin{proof}
For $I=[r,(p-1)p^{n}]$ with $|\log(I)|<\log(p)$ write $\widetilde{M}_{I}=\H^{0}(\mathcal{X}_{I},\mathcal{\widetilde{E}})$.
For any sufficiently small $U$ containing $x_{\infty}$, the proof
of Theorem \ref{6.1} shows that 
\[
\H^{0}\left(U,\mathcal{E}\right)\cong\mathcal{O}_{\mathcal{X}}\left(U\right)^{\la}\otimes_{\widetilde{\B}_{I}^{\la}}\widetilde{M}_{I}^{\la}.
\]
It follows that the quotient $\mathcal{E}_{x_{\infty}}/m_{x_{\infty}}^{n}\mathcal{E}_{x_{\infty}}$
of the stalk $\mathcal{E}_{x_{\infty}}$ by the $n$'th power of the
maximal idea $m_{x_{\infty}}\subset\mathcal{O}_{\mathcal{X},x_{\infty}}^{\la}$
is identified with the quotient $\widetilde{M}_{I}^{\la}/(I_{\theta}^{n}\widetilde{M}_{I})^{\la}$.
Now use Proposition \ref{6.3}.
\end{proof}

\section{The comparison with $(\varphi,\Gamma)$-modules}

In this section, we give reminders on $(\varphi,\Gamma)$-modules
and compare them to locally analytic vector bundles. We keep the notation
from $\mathsection6$ and the assumption that $K_{\cyc}^{\eta}\subset K_{\infty}$
for some $\eta$.

\subsection{Galois representations and $(\varphi,\Gamma)$-modules}

Recall the notations from $\mathsection3$ and let 
\[
\widetilde{\B}_{\rig}^{\dagger}=\widetilde{\B}_{\rig}^{\dagger}(\widehat{K}_{\infty})=\varinjlim_{r}\H^{0}(\mathcal{Y}_{[r,\infty)},\mathcal{O}_{\mathcal{Y}})=\varinjlim_{r}\varprojlim_{s\geq r}\H^{0}(\mathcal{Y}_{[r,s]},\mathcal{O}_{\mathcal{Y}})
\]
 be the extended Robba ring. The $(\varphi,\Gamma)$-actions on $\mathcal{Y}$
induce actions on $\widetilde{\B}_{\rig}^{\dagger}$.
\begin{defn}
\label{7.1}A $(\varphi,\Gamma)$-module over $\widetilde{\B}_{\rig}^{\dagger}$
is a finite free $\widetilde{\B}_{\rig}^{\dagger}$-module with commuting
semilinear $(\varphi,\Gamma)$-actions such that in some basis $\Mat\left(\varphi\right)\in\GL_{d}(\widetilde{\B}_{\rig}^{\dagger})$.
\end{defn}

We can compare these objects to $\left(\varphi,\Gamma\right)$-vector
bundles using two functors. On the one hand, if $\widetilde{\mathcal{M}}$
is a $(\varphi,\Gamma)$-vector bundle, then $\widetilde{\M}_{\rig}^{\dagger}=\varinjlim_{r}\H^{0}(\mathcal{Y}_{[r,\infty)},\mathcal{\widetilde{M}})$
is a $(\varphi,\Gamma)$-module. Here, the nontrivial thing one needs
to check is that $\widetilde{\M}_{\rig}^{\dagger}$ is free, and this
follows from $\widetilde{\B}_{\rig}^{\dagger}$ being Bézout (\cite[Theorem 3.20]{Ke04}). 

One the other hand, given a $(\varphi,\Gamma)$-module $\widetilde{\M}_{\rig}^{\dagger}$
we define a $(\varphi,\Gamma)$-vector bundle $\FT(\widetilde{\M}_{\rig}^{\dagger})$
as follows. If $\widetilde{\M}_{\rig}^{\dagger}$ is a $(\varphi,\Gamma)$-module
then for every $r\gg0$ we have a finite free $\widetilde{\B}_{[r,\infty)}$-semilinear
$\Gamma$-representation $\widetilde{\M}_{[r,\infty)}$ together with
isomorphisms
\[
\varphi^{*}\widetilde{\B}_{[r,\infty)}\otimes_{\widetilde{\B}_{[r/p,\infty)}}\widetilde{\M}_{[r/p,\infty)}\xrightarrow{\sim}\widetilde{\M}_{[r,\infty)}
\]
as well as identifications
\[
\widetilde{\B}_{\rig}^{\dagger}\otimes_{\widetilde{\B}_{[r,\infty)}}\widetilde{\M}_{[r,\infty)}\xrightarrow{\sim}\widetilde{\M}_{\rig}^{\dagger}.
\]
Using the isomorphisms $\varphi:\widetilde{\B}_{[r,\infty)}\xrightarrow{\sim}\widetilde{\B}_{[r/p,\infty)}$
we can then uniquely extend this to all $r>0$ by inductivey defining
$\widetilde{\M}_{[r/p^{n},\infty)}$ through the isomorphisms
\[
\varphi^{*}\widetilde{\B}_{[r/p^{n-1},\infty)}\otimes_{\widetilde{\B}_{[r/p^{n},\infty)}}\widetilde{\M}_{[r/p^{n},\infty)}\xrightarrow{\sim}\widetilde{\M}_{[r/p^{n-1},\infty)}.
\]
Setting for every $r>0$
\[
\H^{0}(\mathcal{Y}_{[r,\infty)},\FT(\widetilde{\M}_{\rig}^{\dagger})):=\widetilde{\M}_{[r,\infty)}
\]
and for every $s\geq r$
\[
\H^{0}(\mathcal{Y}_{[r,s]},\FT(\widetilde{\M}_{\rig}^{\dagger})):=\widetilde{\M}_{[r,\infty)}\otimes_{\widetilde{\B}_{[r,\infty)}}\widetilde{\B}_{[r,s]}
\]
we obtain a $(\varphi,\Gamma)$-vector bundle $\FT(\widetilde{\M}_{\rig}^{\dagger})$.
\begin{prop}
\label{7.2}The functors $\widetilde{\mathcal{M}}\mapsto\varinjlim_{r}\H^{0}(\mathcal{Y}_{[r,\infty)},\mathcal{\widetilde{M}})$
and $\FT$ induce an equivalence of categories
\[
\left\{ \text{\ensuremath{\left(\varphi,\Gamma\right)}-vector bundles on \ensuremath{\mathcal{Y}_{(0,\infty)}}}\right\} \cong\left\{ \text{\ensuremath{\left(\varphi,\Gamma\right)}-modules over \ensuremath{\widetilde{\B}_{\rig}^{\dagger}}}\right\} .
\]
\end{prop}

\begin{proof}
This is well known. See for example the discussion appearing directly
after \cite[Definition 13.4.3]{SW20}. The treatment there is given
in the situation where there is no $\Gamma$-action present, but the
same proof works in our setting.
\end{proof}
The following theorem due to Fontaine and Kedlaya gives the relation
of these objects with Galois representations. To formulate it, we
need to introduce some terminology. Let $y$ be the point of $\mathcal{Y}$
corresponding to $p=0$. A $(\varphi,\Gamma)$-module over over $\widetilde{\B}_{\rig}^{\dagger}$
is called étale if it has a basis for which $\Mat(\varphi)\in\GL_{d}(\mathcal{O}_{\mathcal{Y},y})$.
We also have the notion of a semistable slope 0 vector bundle on $\mathcal{X}$
- we refer the reader to \cite[Définition 5.5.1, Exemple 5.5.2.1]{FF18}.
\begin{thm}
\label{7.3}The following categories are equivalent.

1. Finite dimensional $\Q_{p}$-representations of $G_{K}$.

2. Étale $(\varphi,\Gamma)$-modules over $\widetilde{\B}_{\rig}^{\dagger}$.

3. $\Gamma$-vector bundles on $\mathcal{X}$ which are semistable
of slope $0$.
\end{thm}

\begin{proof}
The equivalence of 2 and 3 follows from Proposition \ref{7.2} and
Proposition \ref{3.3}. The category in 1 is equivalent to $(\varphi,\Gamma)$-modules
over $\widetilde{\B}=\widehat{\mathcal{O}}_{\mathcal{Y},y}[1/p]$,
where $\widehat{\mathcal{O}}_{\mathcal{Y},y}$ is the $p$-adic completion
of $\mathcal{O}_{\mathcal{Y},y}$, by the theorem of Fontaine \cite[Théorème 3.4.3 and Remarque 3.44 (c)]{Fo90}.
Next, by a relatively elementary argument, this category is equivalent
to the category of $(\varphi,\Gamma)$-modules over $\widetilde{\B}^{\dagger}$,
see for example \cite[Theorem 2.4.5]{Ke15} or \cite[Theorem 4.3]{dSP19}.
Finally, one can replace $\widetilde{\B}^{\dagger}$ by $\widetilde{\B}_{\rig}^{\dagger}$
by \cite[Proposition 5.11, Corollary 5.12]{Ke04}. See also \cite[Proposition 11.2.24]{FF18}.
\end{proof}

\subsection{The comparison with locally analytic vector bundles}

Let $\widetilde{\B}_{\rig}^{\dagger,\pa}$ be the subring of pro-analytic
vectors in $\widetilde{\B}_{\rig}^{\dagger}$ for the action of $\Gamma$.
We have a corresponding version of $(\varphi,\Gamma)$-modules.
\begin{defn}
\label{7.4}A $(\varphi,\Gamma)$-module $\M_{\rig}^{\dagger}$ over
$\widetilde{\B}_{\rig}^{\dagger,\pa}$ is a finite free $\widetilde{\B}_{\rig}^{\dagger,\pa}$-module
with commuting semilinear $(\varphi,\Gamma)$-actions such that in
some basis $\Mat(\varphi)\in\GL_{d}(\widetilde{\B}_{\rig}^{\dagger,\pa})$,
and such that the action of $\Gamma$ is pro-analytic. It is étale
if $\widetilde{\B}_{\rig}^{\dagger}\otimes_{\widetilde{\B}_{\rig}^{\dagger,\pa}}\mathscr{\M_{\rig}^{\dagger}}$
is so.
\end{defn}

The following theorem explains the relationship between $(\varphi,\Gamma)$-modules
and locally analytic vector bundles.
\begin{thm}
\label{7.5}The following categories are all equivalent.

1. $(\varphi,\Gamma)$-modules over $\widetilde{\B}_{\rig}^{\dagger}$.

2. $(\varphi,\Gamma)$-modules over $\widetilde{\B}_{\rig}^{\dagger,\pa}$.

3. $(\varphi,\Gamma)$-vector bundles over $\mathcal{Y}_{(0,\infty)}$.

4. Locally analytic $\varphi$-vector bundles on $\mathcal{Y}_{(0,\infty)}$.

5. $\Gamma$-vector bundles on $\mathcal{X}$.

6. Locally analytic vector bundles on $\mathcal{X}$.

\end{thm}

\begin{proof}
The equivalences $1\Leftrightarrow3\Leftrightarrow5$ are Proposition
\ref{7.2} and Proposition \ref{3.3}. $4\Leftrightarrow6$ is similar
to Proposition \ref{3.3}. The proof of $5\Leftrightarrow6$ was given
in Theorem \ref{6.1}, and $3\Leftrightarrow4$ can be proved in a
similar way. It remains to give an equivalence between 2 and 4. The
Frobenius trick functor of $\mathsection7.1$ induces a functor
\[
\FT:\{(\varphi,\Gamma)\text{-modules over }\widetilde{\B}_{\rig}^{\dagger,\pa}\}\rightarrow\{\text{Locally analytic \ensuremath{\varphi}-vector bundles on \ensuremath{\mathcal{Y}_{(0,\infty)}}}\}.
\]
In the other direction we map a locally analytic $\varphi$-vector
bundle $\mathcal{M}$ to $\mathcal{M}_{\rig}^{\dagger}=\varinjlim_{r}\H^{0}(\mathcal{Y}_{[r,\infty)},\mathcal{M})$.
It is easy to check from the definitions these two are inverses to
each other once we know that $\mathcal{M}\mapsto\mathcal{M}_{\rig}^{\dagger}$
is valued in the correct category. So it remains to prove the following.

\textbf{Claim. $\mathcal{M}_{\rig}^{\dagger}$ }is $(\varphi,\Gamma)$-module
over $\widetilde{\B}_{\rig}^{\dagger,\pa}$.

\emph{Proof of the claim. }We only need to explain why $\mathcal{M}_{\rig}^{\dagger}$
is a free $\widetilde{\B}_{\rig}^{\dagger,\pa}$-module. Since we
can always descend along unramified extensions, we may assume $K_{\cyc}\subset K_{\infty}$.
Then $\mathcal{M}$ and $\mathcal{M}_{\rig}^{\dagger}$ are both base
changed from their cyclotomic counterparts $\mathcal{M}^{\Gal(K_{\infty}/K_{\cyc})}$
and $\mathcal{M}_{\rig}^{\dagger,\Gal(K_{\infty}/K_{\cyc})}$, so
we reduce to the cyclotomic case. 

To deal with this case, recall the rings $\B_{I,\cyc}$ from $\mathsection4$.
The (cyclotomic) Robba ring is defined as 
\[
\B_{\rig,\cyc}^{\dagger}=\varinjlim_{r}\varprojlim_{s\geq r}\B_{[r,s],\cyc}.
\]
The maps $\B_{[r,s],\cyc}\hookrightarrow\widetilde{\B}_{I,\cyc}$
of $\mathsection4$ induce an embedding $\B_{\rig,\cyc}^{\dagger}\hookrightarrow\widetilde{\B}_{\rig,\cyc}^{\dagger}=\widetilde{\B}_{\rig}^{\dagger}(\widehat{K}_{\cyc})$.
By \cite[Theorem B]{Be16} we have 
\[
\widetilde{\B}_{\rig}^{\dagger,\pa}=\bigcup_{n\geq0}\varphi^{-n}(\B_{\rig,\cyc}^{\dagger}),
\]
and since each $\varphi^{-n}(\B_{\rig,\cyc}^{\dagger})$ is a Bézout
domain (\cite{La62}), the conclusion follows. 
\end{proof}
In particular, we recover a decompletion result entirely phrased in
terms of $(\varphi,\Gamma)$-modules:
\[
\{(\varphi,\Gamma)\text{-modules over }\widetilde{\B}_{\rig}^{\dagger}\}\cong\{(\varphi,\Gamma)\text{-modules over }\widetilde{\B}_{\rig}^{\dagger,\pa}\}.
\]

This result recovers the decompletion theorem of Cherbonnier-Colmez
\cite{CC98} and Kedlaya \cite{Ke04}.
\begin{thm}
\label{7.6}If $K_{\infty}=K_{\cyc}$, base extension induces an equivalence
of categories
\[
\{(\varphi,\Gamma)\text{-modules over }\B_{\rig,\cyc}^{\dagger}\}\cong\{(\varphi,\Gamma)\text{-modules over }\widetilde{\B}_{\rig,\cyc}^{\dagger}\}.
\]
\end{thm}

\begin{proof}
If $M$ is a $\left(\varphi,\Gamma\right)$-module over $\widetilde{\B}_{\rig,\cyc}^{\dagger,\pa}=\bigcup_{n}\varphi^{-n}(\B_{\rig,\cyc}^{\dagger})$
then there exists $n\gg0$ such that $M$ is defined over $\text{\ensuremath{\varphi^{-n}}(\ensuremath{\B_{\rig,\cyc}^{\dagger}})}$.
If $e_{1},...,e_{d}$ is a basis of $M$ then $\varphi^{n}(e_{1}),...,\varphi^{n}(e_{d})$
is a basis defined over $\B_{\rig,\cyc}^{\dagger}$. Therefore the
category of $\left(\varphi,\Gamma\right)$-modules over $\B_{\rig,\cyc}^{\dagger}$
is equivalent to the category of $\left(\varphi,\Gamma\right)$-modules
over $\widetilde{\B}_{\rig,\cyc}^{\dagger,\pa}$. But this latter
category is equivalent to $(\varphi,\Gamma)$-modules over $\widetilde{\B}_{\rig,\cyc}^{\dagger}$
by Theorem \ref{7.5}.
\end{proof}

\section{Locally analytic vector bundles and $p$-adic differential equations}

\subsection{Modifications of locally analytic vector bundles}

We first introduce the following category. It is the locally analytic
version of Berger's category of $\B$-pairs, see \cite{Be08A}.
\begin{defn}
\label{8.1}A locally analytic $\B$-pair is a pair $\mathcal{W}=(\mathcal{W}_{e},W_{\dR}^{+})$
where $\mathcal{W}_{e}$ is a locally free $\mathcal{O}_{\mathcal{X}\text{-}\left\{ \infty\right\} }^{\la}=\mathcal{O}_{\mathcal{X}}^{\la}|_{\mathcal{\mathcal{X}\text{-}\left\{ \infty\right\} }}$-module
with a semilinear $\Gamma$-action and $W_{\dR}^{+}\subset\B_{\dR}^{\pa}\otimes_{\mathcal{O}_{\mathcal{X}\text{-}\left\{ \infty\right\} }^{\la}}\mathcal{W}_{e}$
is a $\Gamma$-stable $\B_{\dR}^{+,\pa}$-lattice.
\end{defn}

\begin{prop}
\label{8.2}The functor from locally analytic vector bundles to locally
analytic $\B$-pairs mapping $\mathcal{E}$ to $(\mathscr{\mathcal{E}}|_{\mathcal{X}\text{-}\left\{ \infty\right\} },\D_{\dif}^{+}(\mathcal{E}))$
is an equivalence of categories.
\end{prop}

\begin{proof}
There is an obvious functor from the category of locally analytic
$\B$-pairs to the category of $\B$-pairs. This leads to a commutative
diagram
\[
\xymatrix{\left\{ \text{locally analytic vector bundles}\right\} \ar[r]\ar[d]^{\cong} & \left\{ \text{locally analytic \ensuremath{\B}-pairs}\right\} \ar[d]\\
\left\{ \text{\ensuremath{\Gamma}-vector bundles}\right\} \ar[r]^{\cong} & \left\{ \B\text{-pairs}\right\} 
}
.
\]
The left vertical arrow is an equivalence by Theorem \ref{6.1}. The
lower horizontal arrow is also an equivalence, as explained in $\mathsection10.1.2$
of \cite{FF18}. It follows that the functor from locally analytic
$\B$-pairs to $\B$-pairs is essentially surjective, so every $\B$-pair
comes from a locally analytic $\B$-pair by extending scalars. It
now follows from Proposition \ref{2.1} that such a locally analytic
$\B$-pair is unique. This allows us to define a functor from $\B$-pairs
to locally analytic $\B$-pairs, which gives a quasi inverse to right
vertical morphism. It therefore has to be an equivalence. By commutativity
of the diagram, the upper horizontal arrow is also an equivalence,
as required.
\end{proof}
\begin{defn}
\label{8.3}Given two locally analytic vector bundles $\mathcal{E}_{1}$
and $\mathcal{E}_{2}$ we say that $\mathcal{E}_{2}$ is a modification
of $\mathscr{\mathcal{E}}_{1}$ if $\mathcal{E}_{1}|_{\mathcal{X}\text{-}\left\{ \infty\right\} }\cong\mathscr{\mathcal{E}}_{2}|_{\mathcal{X}\text{-}\left\{ \infty\right\} }$.
\end{defn}

Note that in particular any $\Gamma$-stable $\B_{\dR}^{+,\pa}$-lattice
$N\subset\D_{\dif}(\mathcal{E})$ defines a modification of $\mathcal{E}$
by taking the pair $(\mathcal{E}|_{\mathcal{X}\text{-}\left\{ \infty\right\} },N)$.
\begin{rem}
We could have also defined this notion of modification in terms of
usual $\B$-pairs. Our choice of presentation is meant to illustrate
that one can speak of modifications without leaving the locally analytic
realm.
\end{rem}

\subsection{de Rham and $\protect\C_{p}$-admissible locally analytic vector
bundles}

Let $\mathcal{E}$ be a locally analytic vector bundle. We say that
\begin{itemize}
\item $\mathcal{E}$ is $\C_{p}$-admissible if $\dim_{K}\mathcal{E}_{x_{\infty}}^{\Gamma=1}=\rank\left(\mathcal{E}\right)$.
\item $\mathcal{E}$ is de Rham if $\D_{\dR}(\mathcal{E}):=\dim_{K}\widehat{\mathscr{\mathcal{E}}}_{x_{\infty}}^{\Gamma=1}=\rank(\mathcal{E})$. 
\end{itemize}
If $V$ is a $p$-adic representation and $\mathcal{E}=\widetilde{\mathcal{E}}(V)^{\la}$
then $\mathcal{E}_{x_{\infty}}^{\Gamma=1}=(\C_{p}\otimes V)^{G_{K}}$
and $\D_{\dR}\left(\mathcal{E}\right)=\D_{\dR}\left(V\right)$, so
this extends the usual definitions.

In what follows, note that $\D_{\dR}\left(\mathcal{E}\right)$ has
a natural filtration induced from the $I_{\theta}$ filtration on
$\mathscr{\widehat{\mathcal{E}}}_{x_{\infty}}$.
\begin{defn}
\label{8.5}Suppose $\mathcal{E}$ is de Rham.

1. $\mathrm{\mathcal{N}}_{\dR}\left(\mathcal{E}\right)$ is the modification
of $\mathcal{E}$ given by the lattice $\D_{\dR}\left(\mathcal{E}\right)\otimes_{K}\B_{\dR}^{+,\pa}\subset\D_{\dif}\left(\mathcal{E}\right)$.
It is $\C_{p}$-admissible.

2. $\mathcal{M}_{\dR}\left(\mathcal{E}\right)$ is the locally analytic
$\varphi$-vector bundle corresponding to $\mathrm{\mathcal{N}}_{\dR}\left(\mathcal{E}\right)$.
\end{defn}

\subsection{The surfaces $\mathcal{Y}_{\log,L}$ and $\mathcal{X}_{\log,L}$ }

In $\mathsection10.3.3$ of \cite{FF18}, Fargues and Fontaine define
a scheme $X_{\log}$. It is a line bundle over the schematic Fargues-Fontaine
curve $X_{\FF}=X_{\FF}(\C_{p})$ with a natural projection $\pi:X_{\log}\rightarrow X$;
further, it has a $G_{K}$-action and $\pi$ is $G_{K}$-equivariant.

We let $\mathcal{X}_{\log}$ be the analytification of $X_{\log}$.
If $L$ is a finite extension of $K$, we set
\[
\mathcal{X}_{\log,L}:=\mathcal{X}_{\log}/\Gal(\overline{K}/L_{\infty}).
\]
(Alternatively, this can be defined as the analytification of the
quotient of $X_{\log}$ by $\Gal(\overline{K}/L_{\infty})$). Similarly,
write $\mathcal{\mathcal{Y}}_{\log}=\mathcal{Y}_{(0,\infty)}\times_{\mathcal{X}}\mathcal{X}_{\log}$
and $\mathcal{\mathcal{Y}}_{\log,L}=\mathcal{\mathcal{Y}}_{\log}/\Gal(\overline{K}/L_{\infty})$;
then $\mathcal{\mathcal{Y}}_{\log,L}/\varphi=\mathcal{X}_{\log,L}$.
These spaces have an action of $\Gal(L_{\infty}/L)$, an open subgroup
of $\Gamma$.

Write $p_{L}$ (resp. $p_{\log,L}$) for the projection maps $\mathcal{\mathcal{Y}}_{L}\rightarrow\mathcal{Y}$
or $\mathcal{X}_{L}\rightarrow\mathcal{X}$ (resp. $\mathcal{Y}_{\log,L}\rightarrow\mathcal{Y}$
or $\mathcal{X}_{\log,L}\rightarrow\mathcal{X}$). If $I\subset\left(0,\infty\right)$
is closed interval, let $\mathcal{Y}_{\log,L,I}=p_{\log,L}^{-1}(\mathcal{Y}_{I})$
and similarly $\mathcal{\mathcal{X}}_{\log,L,I}=p_{\log,L}^{-1}(\mathcal{X}_{I})$
for $\mathcal{X}$ if $I$ is sufficiently small.

Define
\[
\widetilde{\B}_{\log,L,I}=\H^{0}(\mathcal{Y}_{\log,L,I},\mathcal{O}_{\mathcal{\mathcal{Y}}_{\log,L}}).
\]
As explained in loc. cit., there is a natural $G_{K}$-equivariant
morphism of sheaves
\[
d:\mathcal{O}_{X_{\log}}\rightarrow\Omega_{X_{\log}/X}^{1}\cong p_{\log}^{*}\mathcal{O}_{X}(-1)
\]
which for every vector bundle $\mathcal{E}$ over $\mathcal{X}$ induces
an $\mathcal{O}_{\mathcal{X}}$-linear morphism
\[
N:p_{\log}^{*}\mathcal{E}\rightarrow p_{\log}^{*}\mathcal{E}\otimes\Omega_{\mathcal{X}_{\log}/\mathcal{X}}^{1}.
\]
See \cite[Lemma 10.3.9]{FF18} and the subequent discussion. Similarly,
$N$ can be pulled back to $\mathcal{Y}_{\log}$. This then further
induces a $\widetilde{\B}_{L,I}$-linear differential operator $N:\widetilde{\B}_{\log,L,I}\rightarrow\widetilde{\B}_{\log,L,I}$.
If $T\in\widetilde{\B}_{\log,L,I}$ is such that $N(T)=1$ then $\widetilde{\B}_{\log,L,I}=\widetilde{\B}_{L,I}[T]$
and $N=d/dT$. Such a $T$ exists: if $\varpi$ is any nonunit $\varpi\in\widehat{L}_{\infty}^{\times}$
and $\varpi^{\flat}=(\varpi,\varpi^{1/p},...)$, take $T=\log[\varpi^{\flat}]$. 
\begin{lem}
\label{8.6}There exists $T\in\widetilde{\B}_{\log,L,I}^{\la}$ with
$N(T)=1$. Consequently, $\widetilde{\B}_{\log,L,I}^{\la}=\widetilde{\B}_{L,I}^{\la}[T]$.
\end{lem}

\begin{proof}
The second claim follows the first claim, Proposition \ref{2.1} and
the fact that taking locally analytic vectors commutes with filtered
colimits. To find such an element $T$, consider the exact sequence
\[
0\rightarrow\widetilde{\B}_{L,I}\rightarrow\widetilde{\B}_{\log,L,I}^{N^{2}=0}\xrightarrow{N}\widetilde{\B}_{L,I}\rightarrow0.
\]
After taking locally analytic vectors the sequence stays exact by
Theorem \ref{5.1}. Thus the sequence
\[
0\rightarrow\widetilde{\B}_{L,I}^{\la}\rightarrow\widetilde{\B}_{\log,L,I}^{\la,N^{2}=0}\xrightarrow{N}\widetilde{\B}_{L,I}^{\la}\rightarrow0
\]
is exact. This means we can lift $1$ to an element $T$ with $N(T)=1$,
as required.
\end{proof}
\begin{prop}
\label{8.7}Suppose $\varphi^{\Z}(x_{\infty})\cap\mathcal{Y}_{I}\neq\emptyset$.
Then

(i) If $M$ is a finite extension of $L$ contained in $L_{\infty}$,
then $\widetilde{\B}_{\log,L,I}^{\Gal(L_{\infty}/M)}=M_{0}$ where
$M_{0}$ is the maximal unramified extension of $\Q_{p}$ contained
in $M$.

(ii) $\widetilde{\B}_{\log,L,I}^{\la,\Lie\Gamma=0}=L_{0}'$, the maximal
unramified extension of $\Q_{p}$ contained in $L_{\infty}$.
\end{prop}

\begin{proof}
(i) follows from \cite[Proposition 10.3.15]{FF18} and (ii) follows
from (i).
\end{proof}
One way to construct de Rham locally analytic vector bundles is as
follows. Write $\Mod_{\Q_{p}^{\mathrm{un}}}^{\Fil,\varphi,N}(G_{K})$
for the category of finite dimensional vector spaces $D$ over $\Q_{p}^{\mathrm{un}}$
together with a semilinear action of $\varphi$, a monodromy operator
$N$ with $\varphi N=pN\varphi$, a filtration on $D\otimes_{\Q_{p}^{\mathrm{un}}}K^{\mathrm{un}}$
and a discrete action of $G_{K}$ on $D$ which respects the filtration.
For example, if $V$ is a potentially semistable representation then
$\D_{\mathrm{pst}}(V)$ is an object of $\Mod_{\Q_{p}^{\mathrm{un}}}^{\Fil,\varphi}(G_{K})$.

There is a functor
\[
\mathcal{E}:\Mod_{\Q_{p}^{\mathrm{un}}}^{\Fil,\varphi}(G_{K})\rightarrow\{\text{de Rham locally analytic vector bundles}\}
\]
defined as follows: given $D\in\Mod_{\Q_{p}^{\mathrm{un}}}^{\Fil,\varphi}(G_{K})$,
choose $L$ such that $D$ is defined over $L$, i.e. $D=\Q_{p}^{\mathrm{un}}\otimes_{L_{0}}D_{0}$.
Such an $L$ exists because the action of $G_{K}$ is discrete. Then
$\mathcal{E}(D)$ is defined to be the locally analytic vector bundle
corresponding to the pair
\[
((\mathcal{O}_{\mathcal{Y}_{\log,L}-p_{\log,L}^{-1}\left(\infty\right)}^{\la}\otimes_{L_{0}}D)^{\varphi=1,N=0,\Gal(L_{\infty}/K_{\infty})},\Fil^{0}(\B_{\dR}^{H_{L},\pa}\otimes_{L_{0}}D_{0})^{\Gal(L_{\infty}/K_{\infty})}).
\]
It is de Rham because 
\[
D\subset\B_{\dR}^{H_{K},\pa}\otimes\Fil^{0}(\B_{\dR}^{H_{L},\pa}\otimes_{L_{0}}D_{0})^{\Gal(L_{\infty}/K_{\infty})}
\]
is fixed by an open subgroup of $\Gamma$. If we choose any larger
$L$ we get the same pair, so the construction $D\mapsto\mathcal{E}(D)$
is independent of the choice of $L$.

\subsection{Sheaves of smooth functions}

In this subsection we introduce certain sheaves of functions on $\mathcal{X}$.
All of these can be defined equally well for $\mathcal{Y}_{(0,\infty)}$.
\begin{defn}
\label{8.8}We define the following sheaves of functions on $\mathcal{X}$.

(i) Smooth functions: $\mathcal{O}_{\mathcal{X}}^{\sm}=\mathcal{O}_{\mathcal{X}}^{\la,\Lie\Gamma=0}$.

(ii) For $[L:K]<\infty$, $L$-smooth functions: $\mathcal{O}_{\mathcal{X}}^{L\text{-}\sm}=p_{L,*}\left(p_{L}^{*}\mathcal{O}_{\mathcal{X}}^{\la}\right)^{\Lie\Gamma=0}.$

(iii) For $[L:K]<\infty$, $L$ log-smooth functions: $\mathcal{O}_{\mathcal{X}}^{L\text{-}\lsm}=p_{\log,L,*}\left(p_{\log,L}^{*}\mathcal{O}_{\mathcal{X}}^{\la}\right)^{\Lie\Gamma=0}$.

(iv) Potentially smooth functions: $\mathcal{O}_{\mathcal{X}}^{\psm}=\varinjlim_{[L:K]<\infty}\mathcal{O}_{\mathcal{X}}^{L\text{-}\sm}.$

(v) Potentially log-smooth functions: $\mathcal{O}_{\mathcal{X}}^{\mathrm{plsm}}=\varinjlim_{[L:K]<\infty}\mathcal{O}_{\mathcal{X}}^{L\text{-}\lsm}.$
\end{defn}

The following proposition has been essentially explained to us by
Kedlaya.
\begin{prop}
\label{8.9}Let $U$ be a connected open affinoid subset of $\mathcal{X}$.

(i) The sections of each of $\mathcal{O}_{\mathcal{X}}^{\sm}$, $\mathcal{O}_{\mathcal{X}}^{L\text{-}\sm}$
and $\mathcal{O}_{\mathcal{X}}^{\psm}$ at $U$ is a field which injects
(noncanonically) into $\C_{p}$.

(ii) If $x_{\infty}\in U$ then there are canonical injections $\H^{0}(U,\mathcal{O}_{\mathcal{X}}^{\sm})\hookrightarrow K_{\infty}$,
$\H^{0}(U,\mathcal{O}_{\mathcal{X}}^{L\text{-}\sm})\hookrightarrow L_{\infty}$
and $\H^{0}(U,\mathcal{O}_{\mathcal{X}}^{\psm})\hookrightarrow\overline{K}$.

(iii) If $x_{\infty}\in U$ and $U=\mathcal{X}_{I}$, we have $\H^{0}(\mathcal{X}_{I},\mathcal{O}_{\mathcal{X}}^{\sm})=K_{0}'$,
$\H^{0}(\mathcal{X}_{I},\mathcal{O}_{\mathcal{X}}^{L\text{-}\sm})=L_{0}'$
and $\H^{0}(\mathcal{X}_{I},\mathcal{O}_{\mathcal{X}}^{\psm})=K_{0}^{\mathrm{un}}$.

(iv) We have $\mathcal{O}_{\mathcal{X},x_{\infty}}^{\sm}=K_{\infty}$,
$\mathcal{O}_{\mathcal{X},x_{\infty}}^{L\text{-}\sm}=L_{\infty}$
and $\mathcal{O}_{\mathcal{X},x_{\infty}}^{\psm}=\overline{K}$.
\end{prop}

\begin{proof}
Each of the assertions (i)-(iv) for $\mathcal{O}_{\mathcal{X}}^{\psm}$
follows from the corresponding assertion for $\mathcal{O}_{\mathcal{X}}^{L\text{-}\sm}$.
We shall give below arguments proving (i)-(iv) for $\mathcal{O}_{\mathcal{X}}^{\sm}$;
the proofs for $\mathcal{O}_{\mathcal{X}}^{L\text{-}\sm}$ are the
same once $K$ is replaced by $L$.

After passing to an open subgroup of $\Gamma$, we may assume $\Gamma$
stabilizes $U$. By \cite[Theorem 8.8]{Ke16}, the ring $\mathcal{O}_{\mathcal{X}}(U)$
is a Dedekind domain. Each rank 1 point $x$ of $U$ defines a maximal
ideal of $\mathcal{O}_{\mathcal{X}}(U)$, so $f\in\mathcal{O}_{\mathcal{X}}(U)$
can belong to only finitely many of these points. If $f\in\mathcal{O}_{\mathcal{X}}(U)$
is killed by $\Lie\Gamma$ then $f$ is fixed by a finite subgroup
of $\Gamma$, so these finitely many maximal ideals must form a finite
orbit under the $\Gamma$-action. But the only rank 1 point with finite
orbit is the point $x_{\infty}$, again by \cite[Proposition 10.1.1]{FF18}.
So every $f\in\mathcal{O}_{\mathcal{X}}^{\sm}(U)$ either vanishes
only at $x_{\infty}$ or is invertible.

If $x_{\infty}\notin U$, this proves that $\mathcal{O}_{\mathcal{X}}^{\sm}(U)$
is a field. In particular, it injects into the residue field of each
rank 1 point, and there is a dense subset of $\mathcal{X}$ with residue
field a subfield of $\C_{p}$. This proves (i) in this case. On the
other hand, if $x_{\infty}\in U$ then there is a $\Gamma$-equivariant
embedding of $\mathcal{O}_{\mathcal{X}}^{\la}(U)$ into $\B_{\dR}^{+}(\widehat{K}_{\infty})^{\la}$
which gives an embedding of $\mathcal{O}_{\mathcal{X}}^{\sm}(U)$
into $\B_{\dR}^{+}(\widehat{K}_{\infty})^{\la,\Lie\Gamma=0}=K_{\infty}$.
This simultaenously proves (i) and (ii) for $\mathcal{O}_{\mathcal{X}}^{\sm}$. 

Next, (iii) follows immediately from Proposition \ref{8.7}. For (iv),
we have already shown that $\mathcal{O}_{\mathcal{X}}^{\sm}(U)\subset K_{\infty}$
for each $U$ which contains $x_{\infty}$, so $\mathcal{O}_{\mathcal{X},x_{\infty}}^{\sm}\subset K_{\infty}$.
To show the converse inclusion, use the henselian property of local
rings of adic spaces \cite[III.6.3.7]{Mo19} to show first that $K_{\infty}\subset\mathcal{O}_{\mathcal{X},x_{\infty}}$.
It then follows that $K_{\infty}\subset\mathcal{O}_{\mathcal{X},x_{\infty}}^{\sm}$,
which concludes the proof.
\end{proof}
We raise a few questions to which we expect a positive answer but
have not answered in this article.
\begin{question}
\label{8.10}1. We can show that $\overline{K}\subset\mathcal{O}_{\mathcal{X}}^{\psm}$
if $x$ is any rank 1 point. Indeed, any untilt of $\C_{p}^{\flat}$
is algebraically closed, and one can use this to show that the completed
local rings $\B_{\dR,x}^{+}$ contain $\overline{K}$. This implies
by the same argument that $\overline{K}\subset\mathcal{O}_{\mathcal{X},x}$.
But every element of $\overline{K}$ has finite degree over $K_{0}$,
which is fixed by $G_{K}$. This implies that every $x\in\overline{K}$
is fixed by an open subgroup $G_{K}$ so $\overline{K}\subset\mathcal{O}_{\mathcal{X},x}^{\psm}$.

Is it true that $\overline{K}=\mathcal{O}_{\mathcal{X},x}^{\psm}$
for any rank 1 point $x$?

2. Is it true that for every connected open affinoid $U\subset\mathcal{X}$,
the field $\mathcal{O}_{\mathcal{X}}^{\psm}(U)$ is a finite extension
of $K_{0}^{\mathrm{un}}$? In particular, this would imply a positive
answer to question 1.

3. Is it true that $\mathcal{O}_{\mathcal{X}}^{L\text{-}\sm}=\mathcal{O}_{\mathcal{X}}^{L\text{-}\lsm}$
(and hence $\mathcal{O}_{\mathcal{X}}^{\psm}=\mathcal{O}_{\mathcal{X}}^{\mathrm{plsm}}$)?
If $x_{\infty}\in U$ then $\mathcal{O}_{\mathcal{X}}^{L\text{-}\sm}(U)=\mathcal{O}_{\mathcal{X}}^{L\text{-}\lsm}(U)$.
This can be seen by using the embedding into $\B_{\dR}^{+}$ as in
the proof of Proposition \ref{8.7}.
\end{question}

\subsection{The solution functor}

In this subsection, we assume $\mathcal{E}$ is a de Rham locally
analytic vector bundle. Given $L$ finite over $K$, we define the
sheaves of solutions on $\mathcal{X}$:

1. $\mathrm{Sol}_{L}(\mathcal{E}):=p_{L,*}(p_{L}^{*}\mathcal{N}_{\dR}\left(\mathcal{E}\right))^{\Lie\Gamma=0}$,
a module over $\mathcal{O}_{\mathcal{X}}^{L\text{-}\sm}$,

2. $\mathrm{Sol}_{\log,L}(\mathcal{E}):=p_{\log,L,*}(p_{\log,L}^{*}\mathcal{N}_{\dR}\left(\mathcal{E}\right))^{\Lie\Gamma=0}$,
a module over $\mathcal{O}_{\mathcal{X}}^{L\text{-}\lsm}$,

3. $\Sol\left(\mathcal{E}\right):=\varinjlim_{[L:K]<\infty}\Sol_{\log,L}\left(\mathcal{E}\right)$,
a module over $\mathcal{O}_{\mathcal{X}}^{\plsm}$.

We have similar versions of these sheaves on $\mathcal{\mathcal{Y}}_{(0,\infty)}$,
denoted by $\mathrm{Sol}_{*}^{\varphi}(\mathcal{E})$ for $*\in\{L,\left\{ \log,L\right\} ,\emptyset\}$.
Since the $\varphi$ action on $\mathcal{\mathcal{Y}}_{\log,L}$ is
$\Gamma$-equivariant, there are natural identifications $\mathrm{Sol}_{*}(\mathcal{E})=(\mathrm{Sol}_{*}^{\varphi}(\mathcal{E}))^{\varphi=1}$
and $\mathrm{Sol}_{*}^{\varphi}(\mathcal{E})\cong\mathcal{O}_{\mathcal{Y}_{(0,\infty)}}^{\bullet}\otimes_{\mathcal{O}_{\mathcal{X}}^{\bullet}}\mathrm{Sol}_{*}(\mathcal{E})$
where $(*,\bullet)=\{(L,L\text{-}\sm),(\left\{ \log,L\right\} ,L\text{-}\lsm),(\emptyset,\plsm)\}$.

To make the link with $\mathcal{E}$ clear, we shall need the following
form of the $p$-adic monodromy theorem due to André \cite{An02},
Kedlaya \cite{Ke04} and Mebkhout \cite{Me02}. 
\begin{prop}
\label{8.11}There exists a finite extension $L$ over $K$ such that
if $U$ is an open subset of $\mathcal{Y}_{[r,\infty)}$ for some
$r\gg0$ then the natural map
\[
\mathcal{O}_{\mathcal{Y}_{\log,L}}^{\la}(p_{\log,L}^{-1}U)\otimes_{\mathcal{O}_{\mathcal{Y}_{(0,\infty)}}^{L\text{-}\lsm}(U)}\mathrm{Sol}_{\log,L}^{\varphi}(\mathcal{E})(U)\rightarrow\mathcal{O}_{\mathcal{Y}_{\log,L}}^{\la}(p_{\log,L}^{-1}U)\otimes_{\mathcal{O}_{\mathcal{Y}_{(0,\infty)}}^{\la}(U)}\mathcal{\mathcal{M}}_{\dR}(\mathcal{E})(U).
\]
is an isomorphism. Consequently, if $U\subset\mathcal{X}_{I}$ for
some $I$ then
\[
\mathcal{O}_{\mathcal{X}_{\log,L}}^{\la}(p_{\log,L}^{-1}U)\otimes_{\mathcal{O}_{\mathcal{X}}^{L\text{-}\lsm}(U)}\mathrm{Sol}_{\log,L}(\mathcal{E})(U)\xrightarrow{\sim}\mathcal{O}_{\mathcal{X}_{\log,L}}^{\la}(p_{\log,L}^{-1}U)\otimes_{\mathcal{O}_{\mathcal{X}}^{\la}\left(U\right)}\mathcal{\mathcal{N}}_{\dR}(\mathcal{E})(U).
\]
\end{prop}

\begin{proof}
Let $\widetilde{\D}_{\rig}^{\dagger}$ be the $(\varphi,\Gamma)$-module
corresponding to $\mathcal{\mathcal{M}}_{\dR}\left(\mathcal{E}\right)$.
By the $p$-adic monodromy theorem, we know there is an isomorphism
\[
\widetilde{\B}_{\log,L}^{\dagger,\pa}\otimes_{L_{0}^{\prime}}(\widetilde{\B}_{\log,L}^{\dagger,\pa}\otimes_{\widetilde{\B}_{\mathrm{rig},K}^{\dagger,\pa}}\widetilde{\D}_{\rig}^{\dagger,\pa})^{\Lie\Gamma=0}\xrightarrow{\sim}\widetilde{\B}_{\log,L}^{\dagger,\pa}\otimes_{\widetilde{\B}_{\mathrm{rig},K}^{\dagger,\pa}}\widetilde{\D}_{\rig}^{\dagger,\pa}
\]
in the cyclotomic setting (see \cite[III.2.1]{Be08B}). More generally,
we may descend along unramified extensions to give it in the twisted
cyclotomic case, and by base changing we get it in our setting as
well by the usual argument. 

It follows that for $r\gg0$ we also have an isomorphism
\[
\widetilde{\B}_{\log,[r,\infty),L}^{\pa}\otimes_{L_{0}^{\prime}}(\widetilde{\B}_{\log,[r,\infty),L}^{\pa}\otimes_{\widetilde{\B}_{[r,\infty),K}^{\pa}}\widetilde{\D}_{[r,\infty)}^{\pa})^{\Lie\Gamma=0}\xrightarrow{\sim}\widetilde{\B}_{\log,[r,\infty),L}^{\pa}\otimes_{\widetilde{\B}_{[r,\infty),K}^{\pa}}\widetilde{\D}_{[r,\infty)}^{\pa}.
\]
Pulling back along Frobenius, we obtain this isomorphism for any $r$.
Then by finding $r\gg0$ so that $U\subset\mathcal{Y}_{[r,\infty)}$,
we can base change the isomorphism along the map $\widetilde{\B}_{\log,[r,\infty),L}^{\pa}\rightarrow\mathcal{O}_{\mathcal{Y}_{\log,L}}^{\la}(p_{\log,L}^{-1}U)$
to conclude.
\end{proof}
Note that whether we need to adjoin $\log$ and/or perform a finite
extension $L$ of $K$ depends exactly on whether $\mathcal{E}$ becomes
crystalline or semistable after restricting $G_{K}$ to $G_{L}$.
Applying this observation and taking $\Lie\Gamma=0$ of both sides
of the proposition, we obtain the following.
\begin{thm}
\label{8.12}The sheaf $\Sol\left(\mathcal{E}\right)$ is a locally
free $\mathcal{O}_{\mathcal{X}}^{\plsm}$-module of rank equal to
$\rank(\mathcal{E})$. More precisely:

i. If $\mathcal{E}$ becomes crystalline after restricting $G_{K}$
to $G_{L'}$ for some $L\subset L'\subset L_{\infty}$ then $\mathrm{Sol}_{L}(\mathcal{E})$
is a locally free $\mathcal{O}_{\mathcal{X}}^{L\text{-}\sm}$-module
of rank equal to $\rank(\mathcal{E})$, and there is a natural isomorphism
\[
\mathcal{O}_{\mathcal{X}_{L}}^{\la}\otimes_{\mathcal{O}_{\mathcal{X}}^{L\text{-}\sm}}\mathrm{Sol}_{L}(\mathcal{E})\xrightarrow{\sim}\mathcal{O}_{\mathcal{X}_{L}}^{\la}\otimes_{\mathcal{O}_{\mathcal{X}}^{\la}}\mathcal{N}_{\dR}\left(\mathcal{E}\right).
\]

ii. If $\mathcal{E}$ becomes semistable after restricting $G_{K}$
to $G_{L'}$ for some $L\subset L'\subset L_{\infty}$ then $\mathrm{Sol}_{\log,L}(\mathcal{E})$
is a locally free $\mathcal{O}_{\mathcal{X}}^{L\text{-}\lsm}$-module
of rank equal to $\rank(\mathcal{E})$, and there is a natural isomorphism
\[
\mathcal{O}_{\mathcal{X}_{\log,L}}^{\la}\otimes_{\mathcal{O}_{\mathcal{X}}^{L\text{-}\lsm}}\mathrm{Sol}_{\log,L}(\mathcal{E})\xrightarrow{\sim}\mathcal{O}_{\mathcal{X}_{\log,L}}^{\la}\otimes_{\mathcal{O}_{\mathcal{X}}^{\la}}\mathcal{N}_{\dR}\left(\mathcal{E}\right).
\]
\end{thm}

\begin{lem}
\label{8.13}For each sufficiently small open connected affinoid $U$
of $\mathcal{Y}_{\left(0,\infty\right)}$ which contains an element
of $\varphi^{\Z}(x_{\infty})$, and for $L$ large enough so that
$G_{L}$ stablizes $U$, there is a natural $G_{L}$-embedding $\H^{0}(U,\mathrm{Sol}_{\log,L}^{\varphi}(\mathcal{E}))\hookrightarrow L_{\infty}\otimes_{K}\D_{\dR}\left(\mathcal{E}\right).$
\end{lem}

\begin{proof}
Taking the completed stalk at a $\varphi$-translate of $x_{\infty}$,
we obtain an injection
\[
\mathcal{O}_{\mathcal{Y}_{\log,L}}^{\la}(p_{\log,L}^{-1}U)\otimes_{\mathcal{O}_{\mathcal{Y}_{(0,\infty)}}^{\la}(U)}\mathcal{\mathcal{M}}_{\dR}(\mathcal{E})(U)\hookrightarrow\widehat{L}_{\infty}^{\la}\otimes_{\widehat{K}_{\infty}^{\la}}\D_{\dif}\left(\mathcal{E}\right).
\]
On the other hand, Proposition \ref{8.7} gives an isomorphism
\[
\mathcal{O}_{\mathcal{Y}_{\log,L}}^{\la}(p_{\log,L}^{-1}U)\otimes_{\mathcal{O}_{\mathcal{Y}_{(0,\infty)}}^{L\text{-}\lsm}(U)}\mathrm{Sol}_{\log,L}^{\varphi}(\mathcal{E})(U)\xrightarrow{\sim}\mathcal{O}_{\mathcal{Y}_{\log,L}}^{\la}(p_{\log,L}^{-1}U)\otimes_{\mathcal{O}_{\mathcal{Y}_{(0,\infty)}}^{\la}(U)}\mathcal{\mathcal{M}}_{\dR}(\mathcal{E})(U).
\]

Applying $\Lie\Gamma=0$ to the composition of these maps gives the
desired embedding.
\end{proof}
We can now give an interpertation of the stalk at $x_{\infty}$:
\begin{prop}
\label{8.14}There following are each naturally isomorphic to each
other.

1. The stalk $\mathrm{Sol}(\mathcal{E})_{x_{\infty}}$.

2. The stalk $\mathrm{Sol}(\mathcal{E})_{y}^{\varphi}$ for any $y\in\varphi^{\Z}(x_{\infty})$.

3. $\overline{K}\otimes_{K}\D_{\dR}\left(\mathcal{E}\right)$.

In particular, $\mathrm{Sol}(\mathcal{E})_{x_{\infty}}$ is naturally
a filtered $\overline{K}$-representation of $G_{K}$ of dimension
$\rank(\mathcal{E})$ and $G_{K}$-fixed points $\D_{\dR}\left(\mathcal{E}\right)$.
\end{prop}

\begin{proof}
It is clear 1 and 2 are isomorphic. By Lemma \ref{8.13}, we have
a natural embedding of $\mathrm{Sol}(\mathcal{E})_{y}$, and hence
of $\mathrm{Sol}(\mathcal{E})_{x_{\infty}}$ into $\overline{K}\otimes_{K}\D_{\dR}\left(\mathcal{E}\right)$.
By Theorem \ref{8.12}, $\mathrm{Sol}(\mathcal{E})_{x_{\infty}}$
is a finite free module of rank equal to $\dim_{K}\D_{\dR}\left(\mathcal{E}\right)$
over $\mathcal{O}_{\mathcal{X},x_{\infty}}^{\plsm}$. But by Proposition
\ref{8.7} $\mathcal{O}_{\mathcal{X},x_{\infty}}^{\plsm}=\overline{K}$
so this embedding must be an isomorphism.
\end{proof}
Finally, we consider the global solutions to the differential equation,
namely
\[
D(\mathcal{E})=\H^{0}(\mathcal{Y}_{(0,\infty)},\mathrm{Sol}^{\varphi}(\mathcal{E}))=\H^{0}(\mathcal{Y}_{(0,\infty)},\mathcal{O}_{\mathcal{Y}_{(0,\infty)}}^{\plsm}\otimes_{\mathcal{O}_{\mathcal{X}}^{\plsm}}\mathrm{Sol}(\mathcal{E})).
\]

\begin{prop}
\label{8.15}$D(\mathcal{E})$ is naturally an object of $\Mod_{\Q_{p}^{\mathrm{un}}}^{\Fil,\varphi,N}(G_{K})$
and $\dim_{\Q_{p}^{\mathrm{un}}}D(\mathcal{E})=\rank\left(\mathcal{E}\right)$.
\end{prop}

\begin{proof}
We know each $\H^{0}(\mathcal{Y}_{(0,\infty)},\mathrm{Sol}_{\log,L}^{\varphi}(\mathcal{E}))$
is an $L_{0}'$ vector space for $U$ sufficiently small (independentely
of $L)$, so $D(\mathcal{E})$ is a $\Q_{p}^{\mathrm{un}}$-vector
space. The filtration is induced from the embedding $\H^{0}\left(\mathcal{Y}_{(0,\infty)},\mathrm{Sol}^{\varphi}(\mathcal{E})\right)\hookrightarrow\mathrm{Sol}(\mathcal{E})_{x_{\infty}}\cong\overline{K}\otimes_{K}\D_{\dR}\left(\mathcal{E}\right)$.
The $\varphi$-action is induced from the map $\varphi:\mathcal{Y}_{(0,\infty)}\rightarrow\mathcal{Y}_{(0,\infty)}$.
The monodromy operator $N$ is induced from the equivariant connection
$p_{\log,L}^{*}\mathcal{\mathcal{\mathcal{M}}}_{\dR}\left(\mathcal{E}\right)\rightarrow p_{\log,L}^{*}\mathcal{\mathcal{N}}_{\dR}\left(\mathcal{E}\right)\otimes\Omega_{\mathcal{\mathcal{Y}}_{\log}/\mathcal{Y}_{(0,\infty)}}^{1}$.
Finally, $G_{K}$ acts on the smooth elements in $p_{\log,L}^{*}\mathcal{\mathcal{\mathcal{M}}}_{\dR}\left(\mathcal{E}\right)$,
and this action is discrete because every element is killed by $\Lie\Gamma$,
hence by an open subgroup of $\Gal(L_{\infty}/L)$. To compute the
dimension use Theorem \ref{8.12}. 
\end{proof}
Using this language,\emph{ }Berger's theorem (\cite[Théoréme III.2.4]{Be08B})
admits the following interpertation.
\begin{thm}
\label{8.16}The functors $D\mapsto\mathcal{E}(D)$ and $\mathcal{E}\mapsto D(\mathcal{E})$
are mutual inverses and induce an equivalence of categories
\[
\Mod_{\Q_{p}^{\mathrm{un}}}^{\Fil,\varphi,N}(G_{K})\cong\{\text{de Rham locally analytic vector bundles}\}.
\]
\end{thm}

\begin{rem}
\label{8.17}If $\mathcal{E}$ is the locally analytic vector bundle
associated to a $p$-adic representation $V$, we see that the global-to-local
map
\[
\H^{0}(\mathcal{Y}_{(0,\infty)},\mathrm{Sol}^{\varphi}(\mathcal{E}))\hookrightarrow\mathrm{Sol}(\mathcal{E})_{x_{\infty}}
\]
is nothing but the more familiar map
\[
\D_{\mathrm{pst}}(V)\hookrightarrow\overline{K}\otimes_{K}\D_{\dR}(V).
\]
\end{rem}

\begin{question}
\label{8.18}Theorem \ref{8.16} allows us to consider objects of
$\Mod_{\Q_{p}^{\mathrm{un}}}^{\Fil,\varphi,N}(G_{K})$ as global solutions
to $p$-adic differential equations. The filtration is coming from
the behavious of orders of vanishing at $x_{\infty}=0$, while the
$(\varphi,N,G_{K})$-structure comes from some sort of monodromy of
the map $\varprojlim_{L}\mathcal{Y}_{\log,L}\rightarrow\mathcal{X}$.
In our description the space $\varprojlim_{L}\mathcal{Y}_{\log,L}$
behaves as a substitute for a universal cover of $\mathcal{X}$. It
would be interesting if it can be replaced by a more literal cover
of $\mathcal{X}$ for which the $(\varphi,N,G_{K})$-actions can be
interperted as monodromy actions. One could even speculate that in
an appropriate sense, the analytic fundamental group of $\mathcal{X}(\C_{p})_{\overline{K}}$
should be a tame Weil group with its two dimensions reflecting the
$\varphi$ and $N$ operators.
\end{question}

We conclude with an example.
\begin{example}
\label{8.19}Take $\alpha\in\Z_{p}^{\times}$, and given $g\in\Gal(\overline{\Q}_{p}/\Q_{p})$
let $\xi_{\alpha}(g)\in\Z_{p}$ be the element such that $\zeta_{p^{n}}^{\xi_{\alpha}(g)}=g(\alpha^{1/p^{n}})/\alpha^{1/p^{n}}$
for each $n\geq1$. The Kummer extension 
\[
0\rightarrow\Q_{p}(\chi_{\cyc})\rightarrow V=V_{\alpha}\rightarrow\Q_{p}\rightarrow0
\]
is given by mapping in a basis $e,f$ the element $g$ to the matrix
\[
\left(\begin{array}{cc}
\chi_{\cyc}(g) & \xi_{\alpha}(g)\\
0 & 1
\end{array}\right).
\]
The associated locally analytic vector bundle $\mathcal{E}$ sits
in an exact sequence
\[
0\rightarrow\mathcal{O}_{\mathcal{X}}^{\la}(\chi_{\cyc})\rightarrow\mathcal{E}\rightarrow\mathcal{O}_{\mathcal{X}}^{\la}\rightarrow0.
\]
We have 
\[
\mathcal{N}_{\dR}\left(\mathcal{E}\right)=\mathcal{O}_{\mathcal{X}}^{\la}x\oplus\mathcal{O}_{\mathcal{X}}^{\la}y\cong\mathcal{O}_{\mathcal{X}}^{\la}(1)\oplus\mathcal{O}_{\mathcal{X}}^{\la}
\]
where at a neighborhood of $x_{\infty}$ we have $x=t^{-1}e$ and
$y=-\log[\alpha^{\flat}]t^{-1}e+f$. Thus
\[
\H^{0}(\mathcal{Y}_{(0,\infty)},\Sol_{\Q_{p}}^{\varphi}(\mathcal{E}))=\H^{0}(\mathcal{O}_{\mathcal{\mathcal{Y}}_{(0,\infty)}}^{\sm}x\oplus\mathcal{O}_{\mathcal{\mathcal{Y}}_{(0,\infty)}}^{\sm}y)
\]
\[
=\Q_{p}x\oplus\Q_{p}y.
\]
The action of $\varphi$ is given by $\varphi(x)=p^{-1}x$ and $\varphi(y)=y$.
This gives the underlying $\varphi$-module of $\D_{\cris}(V)$.

To get the filtration, we consider the stalk of $\Sol_{\Q_{p}}(\mathcal{E})$
at $x_{\infty}$. Observe that $\Fil^{0}$ consists exactly of these
smooth sections which do not have a pole at $x_{\infty}$. As $\log[\alpha^{\flat}]\equiv\log_{p}\alpha$
mod $t$, we have $\Fil^{0}\Sol_{\Q_{p}}(\mathcal{E})_{x_{\infty}}=\Q_{p,\cyc}(x\log_{p}\alpha+y)$
and so the filtration on $\D_{\cris}(V)$ is given by
\[
\Fil^{-1}=\D_{\cris}(V)\supset\Fil^{0}=\Q_{p}(x\log_{p}\alpha+y)\supset\Fil^{1}=0.
\]
\end{example}

\end{document}